\documentclass[12pt, leqno]{article}
\usepackage[T1]{fontenc}
\usepackage[utf8]{inputenc}
\usepackage[top=2cm,bottom=2cm,left=3cm,right=3cm,marginparsep=0.3cm,marginparwidth=3cm,includefoot]{geometry}
\usepackage{amssymb, amsmath, amsthm}
\usepackage[style=alphabetic, sorting = nyt, backref=true]{biblatex}
\usepackage{authblk} 
\usepackage[colorlinks = true, urlcolor = blue, linkcolor=red, citecolor=blue]{hyperref}
\usepackage{cleveref}
\usepackage{nameref} 
\usepackage{enumitem}
\usepackage[dvipsnames]{xcolor} 
\usepackage{lineno} 

\usepackage{stmaryrd} 
\usepackage{tikz}
\usepackage{tikz-cd} 
\usepackage{xparse} 
\usepackage{relsize} 
\usepackage[thinc]{esdiff} 
\usepackage{tabularx} 

\usepackage{graphicx}
\usepackage{subcaption}

\numberwithin{equation}{section}

\setlist[itemize]{label=\tiny\textbullet}
\setlist[enumerate]{label=\normalfont(\roman*)}

\crefname{equation}{equation}{equations}
\crefname{figure}{figure}{figures}
\crefname{lem}{lemma}{lemmas}
\crefname{ex}{example}{examples}

\tikzset{
  symbol/.style={
    draw=none,
    every to/.append style={
      edge node={node [sloped, allow upside down, auto=false]{$#1$}}}
  }
}

\theoremstyle{plain}
\newtheorem{thm}{Theorem}[section]

\newtheorem{lem}[thm]{Lemma}
\newtheorem{cor}[thm]{Corollary}
\newtheorem{prop}[thm]{Proposition}

\theoremstyle{definition}
\newtheorem{defi}[thm]{Definition}

\newtheorem{assus}[thm]{Assumptions}
\newtheorem{ex}[thm]{Example}

\newtheorem{rk}[thm]{Remark}
\newtheorem{rks}[thm]{Remarks}

\newtheorem*{idea*}{Idea}

\newtheorem*{que*}{Question}

\newtheorem*{ques*}{Questions}
\newtheorem{nota}[thm]{Notation}
\newtheorem*{ackno}{Acknowledgements}

\theoremstyle{remark}

\newtheorem*{pb*}{Problem}

\newtheorem*{intui*}{Intuition}

\newtheorem{claim}[thm]{Claim}

\def\N{{\mathbb N}}    
\def\Z{{\mathbb Z}}    
\def\R{{\mathbb R}}    
\def\Q{{\mathbb Q}}    
\def\C{{\mathbb C}}    
\def\E{{\mathbb E}}    
\def\P{\mathbb{P}}
\def\V{\mathbb{V}}
\def\W{\mathbb{W}}
\def\kk{\R}

\def\d{\,\mathrm{d}}
\def\D{\mathrm{D}}
\def\st{\, ; \,} 

\NewDocumentCommand{\Lloc}{O{\R} O{1}}{%
\mathrm{L}^{#2}_{\textnormal{loc}}(#1)%
}
\NewDocumentCommand{\Lp}{O{\R} O{1}}{%
\mathrm{L}^{#2}(#1)%
}

\newcommand{\expect}[1]{\E\left[#1\right]}

\newcommand{\eps}{\varepsilon} 
\renewcommand{\phi}{\varphi} 
\newcommand{\libr}{\llbracket} 
\newcommand{\ribr}{\rrbracket} 
\newcommand{\integint}[1]{{\libr 1;#1\ribr}} 
\newcommand{\partint}[1]{\left\lfloor #1\right\rfloor}
\DeclareMathOperator{\Ima}{Im}
\DeclareMathOperator{\Ker}{Ker}

\def\supp{\textnormal{supp}}
\def\id{\mathrm{Id}}

\def\sgn{\textnormal{sgn}}

\newcommand{\closure}[1]{\overline{#1}}
\def\Int{\textnormal{Int}}
\DeclareMathOperator*{\limit}{lim}
\def\pt{\{\mathrm{pt}\}}

\newcommand{\ball}[2]{B(#1,#2)}
\def\dom{\mathrm{dom}}

\renewcommand{\lim}[1][ ]{\varprojlim_{#1}}

\def\Rd{\mathrm{R}} 
\def\transpose{{}^{t}}

\def\field{\kk}

\newcommand{\dual}[1]{{#1}^{*}}
\newcommand{\dualdot}[2]{\langle #1, #2\rangle}
\NewDocumentCommand{\projective}{d[]}{
\IfNoValueTF{#1}{\def\argument{}}{\def\argument{\hspace{-0.1em}\left(#1\right)}}%
\mathbb{P}\argument%
}
\NewDocumentCommand{\dprojective}{d[]}{
\IfNoValueTF{#1}{\def\argument{}}{\def\argument{\hspace{-0.2em}\left(#1\right)}}%
\dual{\mathbb{P}}\argument%
}
\NewDocumentCommand{\sphere}{d[]}{
\IfNoValueTF{#1}{\def\argument{}}{\def\argument{\left(#1\right)}}%
\mathbb{S}\argument%
}

\newcommand{\dualcone}[1]{{#1}^{\circ}}

\newcommand{\antipodal}[1]{{#1}^{a}}
\newcommand{\intpolant}[1]{\Int(\dualcone{\antipodal{#1}})}

\newcommand{\Ck}[1][k]{C^{#1}}
\newcommand{\Conv}[1]{\mathrm{Conv}\left(#1\right)}

\newcommand{\LKcurv}[2][j]{\mathcal{L}_{#1}\left(#2\right)}
\newcommand{\GMcurv}[2][j]{\mathcal{M}_{#1}\left(#2\right)}

\NewDocumentCommand{\dist}{O{\V} O{x}}{%
\mathcal{D}_{#2}'(#1)%
}
\NewDocumentCommand{\tempdist}{O{\V} O{x}}{%
\mathcal{S}_{#2}'(#1)%
}
\NewDocumentCommand{\schwartz}{O{\V} O{x}}{%
\mathcal{S}_{#2}(#1)%
}
\NewDocumentCommand{\Fourier}{m O{x}}{%
\mathcal{F}_{#2}(#1)%
}

\def\Laplace{\mathcal{L}}
\def\Fourier{\mathcal{F}}
\def\Bessel{\mathcal{B}}

\def\Rc{\R\textnormal{c}}       					
\newcommand{\gf}[1][\gamma]{{#1}^{a \circ}}       	
\def\PL{\textnormal{PL}}        					
\def\cpct{\mathrm{c}}              			

\NewDocumentCommand{\Mod}{O{} O{\V} O{\field}}{%
\textnormal{Mod}_{#1}({#3}_{#2})%
}
\NewDocumentCommand{\Db}{O{} O{\V} O{\mathbf{k}}}{%
\textnormal{D}^{\textnormal{b}}_{#1}({#3}_{#2})%
}

\def\conv{\star}

\newcommand{\g}[1][\gamma]{{#1}}       			
\NewDocumentCommand{\CF}{O{} O{\V}}{%
\textnormal{CF}_{#1}(#2)%
}

\def\1{{\mathbf{1}}} 
\NewDocumentCommand{\Euler}{d[]}{
\IfNoValueTF{#1}{\def\argument{}}{\def\argument{\left(#1\right)}}%
\chi\argument%
}
\NewDocumentCommand{\Eulerc}{d[]}{
\IfNoValueTF{#1}{\def\argument{}}{\def\argument{\left(#1\right)}}%
\chi_c\argument%
}
\NewDocumentCommand{\EPind}{d[]}{
\IfNoValueTF{#1}{\def\argument{}}{\def\argument{\left(#1\right)}}%
\chi_{\mathrm{loc}}\argument%
}

\NewDocumentCommand{\epigraph}{O{f} O{}}{%
\Gamma_{#1}^{#2}%
}
\NewDocumentCommand{\sublevelCF}{m O{\gamma}}{%
\phi_{#1}^{#2}%
}
\NewDocumentCommand{\levelCF}{m}{%
\phi_{#1}%
}

\def\Radon{\mathcal{R}}
\def\ECT{\mathrm{ECT}}

\NewDocumentCommand{\transform}{s O{\phi} d<> O{\kappa}}{%
	\IfBooleanTF #1%
	{\mathrm{T}_{#4}}%
	{\IfNoValueTF{#3}{\def\argument{}}{\def\argument{\left(#3\right)}}%
	\mathrm{T}_{#4}\left[#2\right]{\argument}
	}%
}

\NewDocumentCommand{\subleveltransform}{O{\kappa} d<> O{f_{|Z}}}{%
\IfNoValueTF{#2}{\def\argument{}}{\def\argument{\left(#2\right)}}%
\mathrm{Sub}_{#1}\left[#3\right]{\argument}%
}

\def\EL{\mathcal{E}\hspace{-0.1em}\Laplace}
\NewDocumentCommand{\ELaplace}{O{\phi} d<>}{%
\IfNoValueTF{#2}{\def\argument{}}{\def\argument{\left(#2\right)}}%
\EL\left[#1\right]{\argument}%
}

\def\EF{\mathcal{E}\hspace{-0.1em}\Fourier}
\NewDocumentCommand{\EFourier}{O{\phi} d<>}{%
\IfNoValueTF{#2}{\def\argument{}}{\def\argument{\left(#2\right)}}%
\EF\left[#1\right]{\argument}%
}
\NewDocumentCommand{\GREFourier}{O{\phi} d<>}{%
\IfNoValueTF{#2}{\def\argument{}}{\def\argument{\left(#2\right)}}%
\EF^\mathrm{GR}\left[#1\right]{\argument}%
}
\def\EB{\mathcal{E}\Bessel}
\NewDocumentCommand{\EBessel}{O{\phi} d<>}{%
\IfNoValueTF{#2}{\def\argument{}}{\def\argument{\left(#2\right)}}%
\EB\left[#1\right]{\argument}%
}

\def\cdEuler{\,\lceil\d\Euler\rceil}
\def\fdEuler{\,\lfloor\d\Euler\rfloor}

\title{Hybrid transforms of constructible functions}

\author{Vadim Lebovici\footnote{\textit{Université Paris-Saclay, CNRS, Inria, Laboratoire de Mathématiques d'Orsay, 91405, Orsay, France.} \texttt{vadim.lebovici@ens.fr}\\
\textnormal{MSC: 32B20, 44A05, 55N31},\\
\textnormal{Keywords:} topological data analysis, integral transforms, constructible functions.}}
\date{\today}

\begin{document}

\maketitle
\vspace{-0.7cm}

\begin{abstract}
    We introduce a general definition of hybrid transforms for constructible functions. These are integral transforms combining Lebesgue integration and Euler calculus. Lebesgue integration gives access to well-studied kernels and to regularity results, while Euler calculus conveys topological information and allows for compatibility with operations on constructible functions. We conduct a systematic study of such transforms and introduce two new ones: the Euler-Fourier and Euler-Laplace transforms. We show that the first has a left inverse and that the second provides a satisfactory generalization of Govc and Hepworth's persistent magnitude to constructible sheaves, in particular to multi-parameter persistent modules. Finally, we prove index-theoretic formulae expressing a wide class of hybrid transforms as generalized Euler integral transforms. This yields expectation formulae for transforms of constructible functions associated to (sub)level-sets persistence of random Gaussian filtrations.
\end{abstract}

\begin{small}
    \renewcommand{\baselinestretch}{0.75}\small
    \tableofcontents
    \renewcommand{\baselinestretch}{1.0}\normalsize
\end{small}

\section{Introduction}
Originally developed by O. Viro \cite{V88} in the complex setting and by P. Schapira \cite{S88, S91} in the real analytic setting, \emph{Euler calculus} --- the integral calculus of constructible functions with respect to the Euler characteristic --- is of increasing interest in topological data analysis (TDA). Already in \cite{S88}, Euler calculus was developed as an alternative definition of convolution for polygonal tracings with multiplicities, a useful notion in robotics \cite{GRS83}. In TDA, one can formulate problems of target detection by sensor networks using the Euler calculus formalism. This paradigm allows to express the number of targets detected by the network as the integral of a constructible function~\cite{BG09}, and even suggests the possibility of target reconstruction thanks to Schapira's inversion result \cite[Thm.~3.1]{S95} for specific networks, such as beam sensor networks \cite[Sec.~20.2]{Cur12}. In persistence theory, Schapira's inversion positively answers an important inverse question~\cite[Thm.~4.11]{CMT18}: are two constructible subsets of~$\R^n$ with same persistent homology in all degrees and for all height filtrations equal? More generally, the constructible functions naturally associated to multi-parameter persistent modules stand as simple, informative and well-behaved, albeit incomplete, invariants of these objects. The most problematic aspect of the Euler calculus for applications is its instability under numerical approximations: errors can (and probably will) be made when calculating the integral of a constructible function, no matter how finely its domain is sampled \cite[Sec.~16]{Cur12}.

\medskip

In this article, we introduce a general definition of integral transforms combining Lebesgue integration and Euler calculus for constructible functions in the hope of getting the best of both worlds. The former calculus offers well-studied kernels (Fourier, Laplace, ...) yielding smooth integral transforms that are stable in numerical approximations. The latter conveys topological information and is compatible with operations on constructible functions. Our transforms generalize the \emph{Bessel} and \emph{Fourier} transforms introduced by R. Ghrist and M. Robinson \cite{GR11}, as well as the \emph{Euler characteristic of barcodes} introduced by O. Bobrowski and M. Borman~\cite{BobBor12}.

We conduct here a systematic study of hybrid transforms. First, we prove their regularity and their compatibility with operations on constructible functions. Then, we introduce two new ones. The \emph{Euler-Laplace} transform, which appears as a satisfactory --- in view of the properties of hybrid transforms --- generalization of D. Govc and R. Hepworth's \emph{persistent magnitude}~\cite[Def.~5.1]{GH21} to constructible sheaves, so in particular to multi-parameter persistent modules. Then, the \emph{Euler-Fourier} transform (see \Cref{fig:EF-spiral}), which has a left inverse and paves the way for a hybrid Fourier theory. Numerous examples are presented for these two transforms, illustrating their characteristics and their differences from their classical analogues. Finally, we prove index-theoretic formulae for a wide class of hybrid transforms on constructible functions arising in (sub)level-sets multi-persistence, generalizing existing ones for the magnitude \cite[Thm.~6.1]{GH21}, the Euler characteristic of barcodes \cite[Prop.~6.2]{BobBor12} and the Bessel and Fourier transforms \cite[Thms.~4.2,~4.4]{GR11}. In particular, this yields expectation formulae for random Gaussian filtrations. 
\begin{figure}[ht] 
    \centering
    \begin{subfigure}[bl]{0.48\linewidth}
        \includegraphics[width=.85\linewidth]{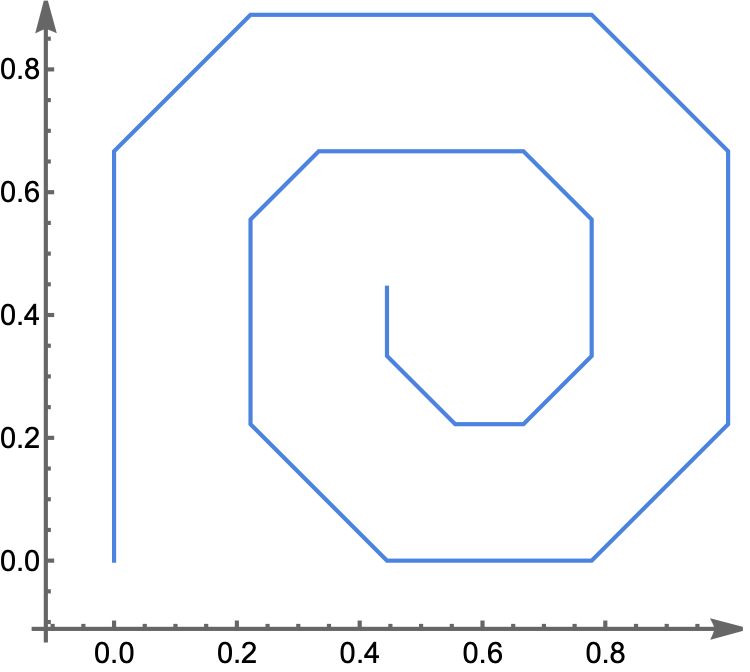} 
        \caption{$\1_C$}
    \end{subfigure}
    \begin{subfigure}[br]{0.48\linewidth}
            \includegraphics[width=\linewidth]{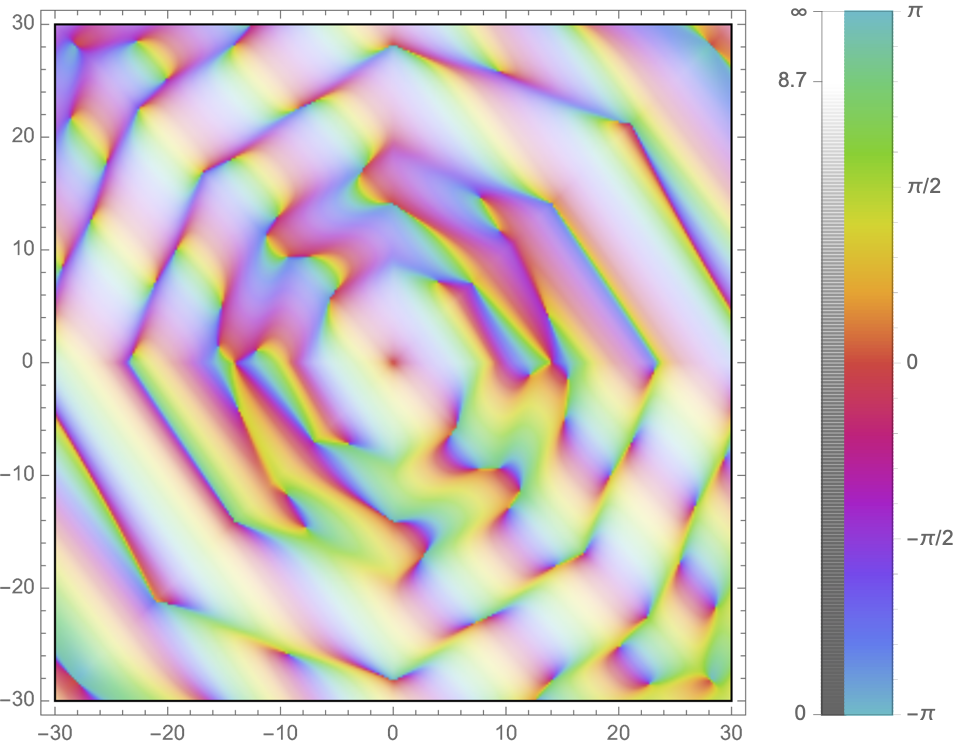}
        \caption{$\EFourier[\1_C]$} 
    \end{subfigure} 
    \begin{subfigure}[b]{0.5\linewidth}
      \centering
      \includegraphics[width=.95\linewidth]{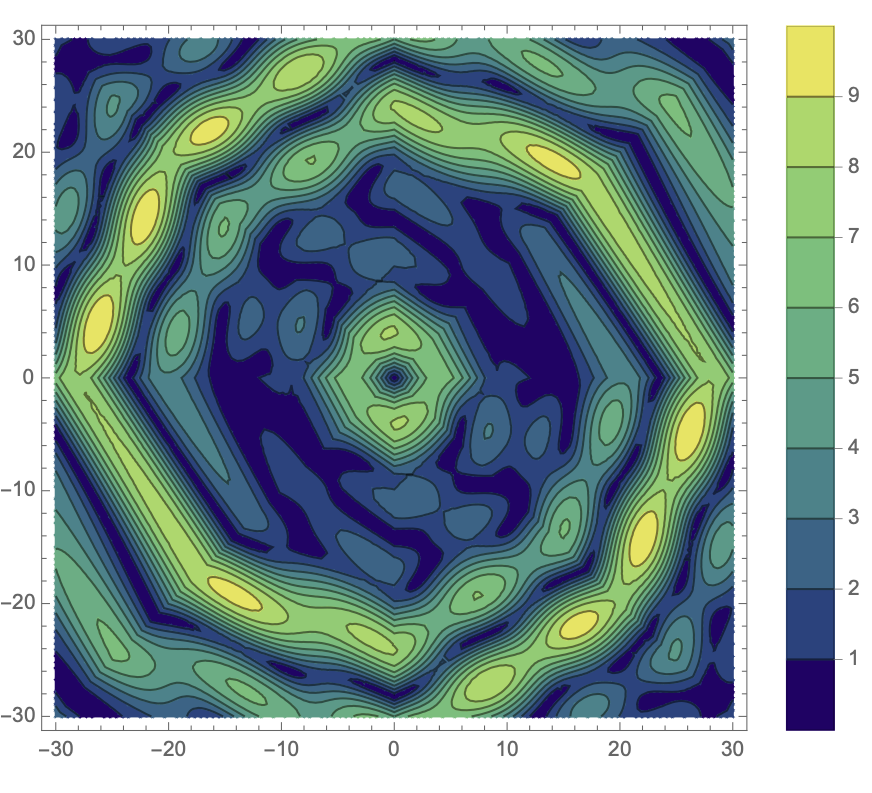} 
      \caption{$|\EFourier[\1_C]|$} 
    \end{subfigure}
    \begin{subfigure}[b]{0.5\linewidth}
      \centering
      \includegraphics[width=.95\linewidth]{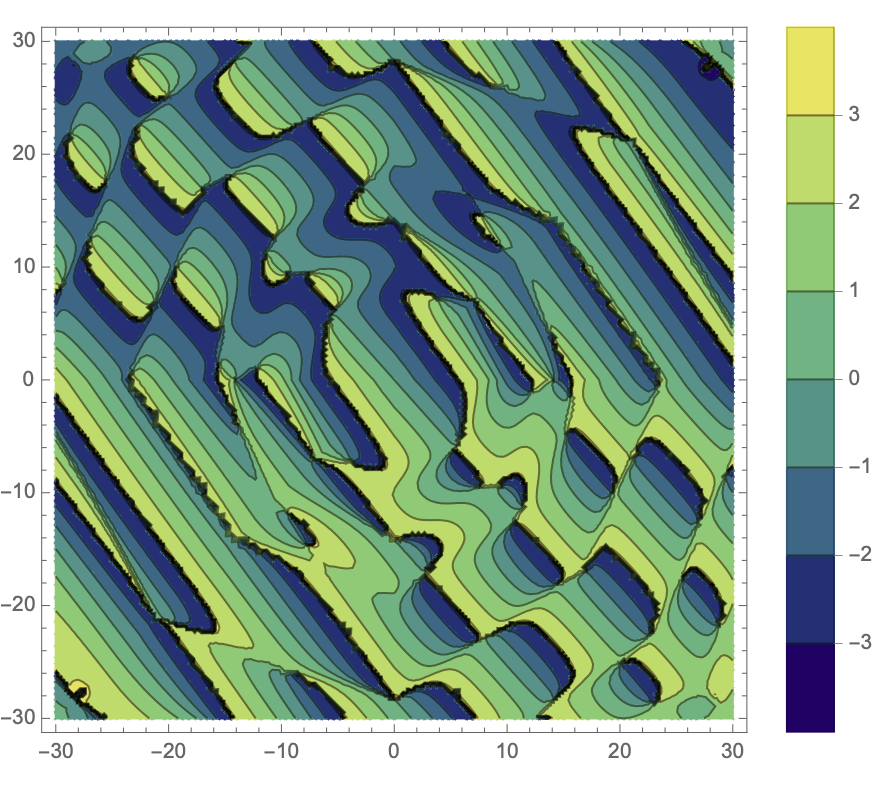} 
      \caption{$\mathrm{Arg}(\EFourier[\1_C])$} 
    \end{subfigure} 
    \caption{(a) a piecewise-linear closed curve~$C$ in~$\R^2$, (b) the Euler-Fourier transform of the constructible function~$\1_C$ as well as (c) its absolute value and (d) its argument. Plots are done following \Cref{rk:computations}.}
    \label{fig:EF-spiral}
\end{figure}

\paragraph{Outline.} 
\begin{itemize}
    \item[Sec.~\ref{sec:preliminaries}.] We set our notations and recall some basic definitions and results on Euler calculus and constructible functions that are not explicitely found elsewhere in the literature. 
    \item[Sec.~\ref{sec:hybrid-transforms}.] We introduce the general definition of hybrid transforms, as well as the \emph{Euler-Laplace} transform and the \emph{Euler-Fourier} transform. Illustrated examples are provided. We also present the method used to numerically compute our transforms on PL-constructible functions, i.e. those whose defining strata are polyhedral.
    \item[Sec.~\ref{sec:regularity}.] We prove that hybrid transforms are continuous when restricted to the set of PL-constructible functions, and even~$\Ck[p+1]$ on the interior of cones partitioning their domain when their kernel is~$\Ck[p]$.
    \item[Sec.~\ref{sec:compatibility}.] We show that hybrid transforms are compatible with operations on constructible functions such as pushforwards and duality. In addition, we show that the Euler-Laplace and Euler-Fourier transforms turn (constructible) convolutions into products under mild assumptions.
    \item[Sec.~\ref{sec:reconstruction-EFourier}.] We establish a reconstruction formula for the Euler-Fourier transform of~$\gamma$\mbox{-}\nobreak\hspace{0pt}constructible functions, i.e. those whose strata are~$\gamma$-locally closed. 
    Using the inverse of the classical Fourier transform, one can recover, from the knowledge of the Euler-Fourier transform, the values of the (constructible) Radon transform on the set of affine hyperplanes whose defining conormal is in the cone~$\gamma$.
    We show that the Radon transform is in fact fully recovered, as it is zero for any other hyperplane. All that remains is to invert the Radon transform with Schapira's formula \cite{S95}.
    \item[Sec.~\ref{sec:persistent-cohomology}.] We define the so-called \emph{sublevel-sets} and \emph{level-sets constructible functions} associated to a continuous subanalytic filtration~$f: M \to \V$ and a cone~$\gamma\subseteq \V$ where~$\V$ is a finite-dimensional vector space and~$M$ a real analytic manifold. 
    They are simply the constructible functions associated to persistent cohomology sheaves introduced by M. Kashiwara and Schapira \cite{KS18A}. 
    We show that one can reduce the study of hybrid transforms of (sub)level-sets constructible functions associated to vector-valued filtrations to those associated to real-valued filtrations. It leads to the definition of \emph{sublevel-sets transforms}, the simple form that hybrid transforms take in the case of multi-persistence. The key ingredient in the proofs is an expression of the sublevel-sets constructible function as a convolution of the level-sets constructible function with the indicator function of the antipodal of~$\gamma$.

    \item[Sec.~\ref{sec:index-theoretic}.] We begin by recalling Bobrowski and Borman's definition of \emph{continuous Euler integral} \cite{BobBor12} that extends Euler calculus to the wider class of \emph{tame functions}. This class contains continuous subanalytic functions on compact real analytic manifolds.
	Then, we prove index-theoretic formulae expressing (sub)level-sets transforms as continuous Euler integral transforms. 
	This allows us to prove an expectation formula for the Euler-Bessel transform using our index-theoretic formula and Bobrowski and Borman's formula for the expectation of continuous Euler integrals of random Gaussian related fields \cite[Thm.~4.1]{BobBor12}.
\end{itemize}

\begin{ackno}
    The author is grateful to Fran\c cois Petit and Steve Oudot for their continuous support and their scientific advice through the development of this paper. We also express our gratitude to Nicolas Berkouk for taking the time to discuss and offer useful explanations on sheaf theory.
\end{ackno}

\section{Preliminaries}\label{sec:preliminaries}
\begin{enumerate}
    \item Throughout the paper, we consider two~$\kk$-vector spaces~$\V$ and~$\V'$ of finite positive dimension, a real analytic manifold~$X$ and a compact real analytic manifold~$M$.
    \item We denote by~$\R_{>0}$ the set of positive real numbers and~$\R_{\geq 0} = \R_{>0}\cup \{0\}$. Similarly, we use the notations~$\R_{<0}$ and~$\R_{\leq 0}$. We also denote by~$\R^\times$ the multiplicative group~$\R\setminus\{0\}$.
    \item The dual of a vector space~$\V$ is denoted by~$\dual{\V}$ and~$\R^n$ will always be identified with its dual under the canonical isomorphism. For~$\xi\in\dual{\V}$ and~$x\in\V$, we often denote by~$\dualdot{\xi}{x} = \xi(x)$.

    \item We say that a subset~$\gamma$ of~$\V$ is a \emph{cone} if~$\R_{>0}\cdot \gamma \subseteq \gamma$. We call \emph{antipodal} of a cone~$\gamma$, denoted by~$\antipodal{\gamma}$, the cone~$-\gamma$. A closed cone $\gamma$ is \emph{proper} if~$\gamma \cap \antipodal{\gamma} = \{0\}$. We call \emph{polar} of a cone~$\gamma$, denoted by~$\dualcone{\gamma}$, the cone of~$\dual{\V}$ defined by:
    \begin{equation*}
        \dualcone{\gamma} = \left\{ \xi\in \dual{\V}\st \xi(\gamma)\subseteq \R_{\geq 0} \right\}.
    \end{equation*}
    \item Let~$I$ be an interval of~$\R$ and denote by~$\Lp[I]$ (resp.~$\Lloc[I]$) the space of integrable (resp. locally integrable) complex-valued functions on~$I$.
    \item Let $\mathbf{k}$ be a field. Throughout the paper, we assume that $\mathbf{k}$ is of characteristic zero. This assumption is necessary for the sheaf-function correspondence to hold, see \cite[Chap.9, Thm.~9.7.1]{KS90}.
    
    \item Following the notations of \cite{KS18A}, we denote by~$\Db[\Rc]$ the derived category of constructible sheaves of $\mathbf{k}$-vector spaces on~$\V$, by~$\Db[\cpct,\Rc]$ its full subcategory generated by compactly supported objects and by~$\Db[\Rc, \gf]$ its full subcategory generated by constructible~$\gamma$-sheaves. 
\end{enumerate}

\subsection{Euler calculus}\label{sec:Euler-calculus}
We refer to \cite[Sec.~8.2]{KS90} for a concise exposition of the useful definitions and results on subanalytic sets. A \emph{constructible function} on a real analytic manifold~$X$ is a function~$\phi : X \to \Z$ such that the sets~$\phi^{-1}(m)$ are subanalytic and the family~$\{\phi^{-1}(m)\}_{m\in \Z}$ is locally finite. 

Using classical results from subanalytic geometry, one can show that the set~$\CF[][X]$ of constructible functions on~$X$ is a commutative unital algebra for the pointwise operations of addition and multiplication. Several operations on constructible functions (e.g. pushforward, pullback, tensor product) coming from sheaf theory have been defined by Viro \cite{V88} and Schapira \cite{S88, S91}, the link between constructible functions and constructible sheaves being made precise by the function-sheaf correspondence \cite[Thm.~9.7.1]{KS90}. We shall use their definition and the classical properties they satisfy without proofs, refering to loc. cit. for more details. 

Let us recall only two definitions, central to this work. Denote by~$\CF[\cpct][X]$ the set of constructible functions with compact support on~$X$. As mentioned in Chap. IX of loc. cit., Hardt's triangulation theorem \cite{H76} allows to write any~$\phi\in \CF[\cpct][X]$ as a finite sum~$\phi = \sum_{i=1}^n m_i \1_{K_i}$, where the~$m_i$'s are integers and the~$K_i$'s are compact contractible subanalytic subsets. We can then define the \emph{integral of~$\phi$ with respect to the Euler characteristic}, by:
\begin{equation*}
    \int_X \phi \d\Euler = \sum_{i=1}^n m_i, 
\end{equation*}
which does not depend on the decomposition of~$\phi$, see loc. cit.. Compared to \cite{S88}, we add the notation~$\d\Euler$ to make this integration easier to distinguish from Lebesgue integration in formulae combining the two.
\begin{ex}
    For any two real numbers~$a < b$, one has:
    \begin{enumerate}
        \item~$\int_\R \1_{[a,b]}\d\Euler = 1$, 
        \item~$\int_\R \1_{[a,b)}\d\Euler = \int_\R \1_{(a,b]}\d\Euler = 0$,
        \item~$\int_\R \1_{(a,b)}\d\Euler = -1$.
    \end{enumerate}
\end{ex}
\begin{ex}
    If~$Z$ is a locally closed relatively compact subanalytic subset of~$X$, then:
    \begin{equation*}
        \int_X \1_Z \d\Euler = \Eulerc[Z],
    \end{equation*}
    where~$\Eulerc[Z]$ is the Euler-Poincaré index with compact support of~$Z$, defined by 
    \begin{equation*}
        \Eulerc[Z] = \sum_{j\in\Z} (-1)^j \dim_\Q\left(H_c^j(Z; \Q_{Z})\right),
    \end{equation*}
    where~$\Q_{Z}$ denotes the constant sheaf on~$Z$ with coefficients in the field~$\Q$ and~$H_c$ denotes cohomology with compact supports, i.e.~$H_c^j(Z; \Q_{Z}) = \Rd^j\Gamma_c(Z; \Q_{Z})$ following the notations of~\cite{KS90}. Note that if~$Z$ is compact, then~$\Eulerc[Z]$ is the Euler characteristic. 
\end{ex}
%

For~$\phi \in \CF[][X]$ and~$f : X \to Y$ a morphism of real analytic manifolds which is proper on~$\supp(\phi)$, we define the \emph{pushforward}~$f_*\phi \in \CF[][Y]$, for any~$y\in Y$, by:
\begin{equation*}
    f_*\phi(y) = \int_X \1_{f^{-1}(y)} \cdot \phi \d\Euler.
\end{equation*}
\begin{rk}
    \label{rk:int-pushforward}
    One has that~$\int_X \phi \d\Euler = a_{X*} \phi$ where~$a_X : X \to \pt$, and where the function~$a_{X*}\phi\in\CF[][\pt]$ is identified with its value at the point. Since the pushforward is functorial \cite[Thm.~2.3~(ii)]{S91}, Euler calculus enjoys a \emph{Fubini theorem}, that is integration is invariant by pushforward: for any morphism of real analytic manifolds~$f:X \to Y$ which is proper on~$\supp(\phi)$, one has

    \begin{equation*}
        \int_Y f_*\phi \d\Euler = \int_X \phi\d\Euler.
    \end{equation*}
\end{rk}

We finish this section with three well-known results useful all along the paper that are not explicitly written elsewhere in the literature.

\begin{lem}
    \label{lem:support-pushforward}
    Let~$\phi \in \CF[][X]$ and~$f:X \to Y$ be a morphism of real analytic manifolds which is proper on~$\supp(\phi)$. Then,~$\supp(f_*\phi)\subseteq f(\supp(\phi))$ with equality if~$f$ is injective. 
\end{lem}
\begin{proof}
    For~$y\in Y$, we have:
    \begin{equation}
        \label{eq:expression-pushforward}
        f_*\phi(y) = \int_X \phi \cdot \1_{f^{-1}(y)} \d\Euler = \int_X \phi \cdot \1_{f^{-1}(y)\cap\supp(\phi)} \d\Euler.
    \end{equation}
    Therefore, if~$y\not\in f(\supp(\phi))$, then~$f^{-1}(y)\cap\supp(\phi)= \emptyset$, hence~$f_*\phi(y) = 0$. Thus,~$f_*\phi$ vanishes on the complement of the closed\footnote{Since~$f_{|\supp(\phi)}$ is proper and~$Y$ is Hausdorff and locally compact, it is a closed map.} set~$f(\supp(\phi))$, hence the result. 

    \medskip

    If~$f$ is injective then \eqref{eq:expression-pushforward} becomes:
    \begin{equation}
        \label{eq:expression-pushforward-injective}
        f_*\phi(y) =
        \begin{cases}
            \phi(f^{-1}(y)) &\mbox{if } y\in \Ima(f), \\
            0 &\mbox{else}.
        \end{cases}
    \end{equation}
    Let~$y\not\in \supp(f_*\phi)$. There exists a neighborhood~$U$ of~$y$ such that for any~$z\in U$, we have~$f_*\phi (z) = 0$. If~$y\in\Ima(f)$, the open neighborhood~$f^{-1}(U)$ of~$f^{-1}(y)$ in~$X$ satisfies that for any~$x\in f^{-1}(U)$,~$\phi(x)= \phi(f^{-1}(z)) = f_*\phi(z) = 0$, where~$z= f(x)\in U \cap \Ima(f)$. Therefore,~$f^{-1}(y)\not\in \supp(\phi)$, i.e.~$y\not\in f(\supp(\phi))$. Moreover, if~$y\not\in\Ima(f)$, then obviously~$y\not\in f(\supp(\phi))$. Hence,~$f(\supp(\phi))\subseteq \supp(f_*\phi)$.
\end{proof}
\begin{rk}
    \label{rk:pushforward-indicatingCF-injective}
    The proof of the previous lemma~\eqref{eq:expression-pushforward-injective} also implies that if~$f$ is injective, then for any subanalytic subset~$Z\subseteq X$, we have~$f_*\1_Z = \1_{f(Z)}$.
\end{rk}

\begin{lem}
    \label{lem:int-boxtimes-pushforward}
    Let~$\phi \in \CF[][X]$ and~$\psi \in \CF[][Y]$. Let~$f : X \to W$ and~$g : Y \to Z$ be two morphisms of real analytic manifolds which are respectively proper on~$\supp(\phi)$ and~$\supp(\psi)$. Then, denoting~$f\times g : X\times Y \longrightarrow W\times Z$ the natural map, we have:

    \begin{equation*}
        (f\times g)_*(\phi\boxtimes\psi) = (f_*\phi)\boxtimes(g_*\psi),
    \end{equation*}
    where by definition~$(\phi\boxtimes\psi)(x,y) = \phi(x)\cdot\psi(y)$ for any~$x\in X$ and~$y\in Y$.
\end{lem}
\begin{proof}
    For~$(w,z)\in W\times Z$, we compute:
    \begin{align*}
        (f\times g)_*(\phi\boxtimes\psi)(w,z) &= \int_{X\times Y} \left(\1_{f^{-1}(w)}\cdot \phi\right)\boxtimes \left(\1_{g^{-1}(z)}\cdot \psi\right) \d\Euler \\
        &= \left(\int_{X} \1_{f^{-1}(w)}\cdot \phi \d\Euler\right) \cdot \left(\int_{Y} \1_{g^{-1}(z)}\cdot \psi \d\Euler\right),
    \end{align*}
    where the first equality follows from a direct computation and the second equality follows from the definition of Euler integration.
\end{proof}

\begin{cor}
    \label{cor:pushforward-linear-conv}
    Let~$\phi,\psi\in \CF$, and~$f : \V \to \V'$ be linear and proper on~$\supp(\phi)+\supp(\psi)$. Then,~$f$ is proper on~$\supp(\phi)$ and on~$\supp(\psi)$, and: 

    \begin{equation*}
        f_*\left(\phi\conv \psi\right) = (f_*\phi)\conv(f_*\psi),
    \end{equation*}
    where by definition~$\phi\conv\psi = s_*(\phi\boxtimes\psi)$ with~$s:\V\times\V\to\V$ the addition.
\end{cor}
\begin{proof}
    We have~$f\circ s = s \circ (f \times f)$ by linearity of~$f$, so that by \Cref{lem:int-boxtimes-pushforward}, $f_*\left(\phi\conv \psi\right) = s_*\left(f_*\phi\boxtimes f_*\psi\right) = (f_*\phi)\conv(f_*\psi)$.
\end{proof}


\subsection{Constructibility up to infinity}
To ensure well-definedness of our hybrid transforms, we often restrict ourselves to a subclass of constructible functions on~$\V$ defined in \cite{S21}, called constructible up to infinity in the projective compactification of~$\V$, or simply \emph{constructible up to infinity} in the present article. They correspond to constructible functions that are still constructible when extended by~$0$ to the projective compactification of~$\V$. 

Setting~$\W = \V\oplus\R$, we denote by~$\projective[\V]$ the \emph{projective compactification of~$\V$}, i.e. the set of linear subspaces of~$\W$ of dimension~$1$, formally defined as the quotient:
\begin{equation*}
    \projective[\V] \simeq \big(\W\setminus \{0\}\big)/\R^\times.
\end{equation*}
Any point~$x\in \projective[\V]$ can thus be written in homogeneous coordinates as a class~$x = [v:\lambda]$ with~$(v,\lambda)\in \W\setminus\{0\}$ and there is an open embedding~$ j : v\in \V \hookrightarrow [v:1]\in \projective[\V]$. We denote by~$\dprojective[\V]$ the projective compactification of the dual of~$\V$:
\begin{equation*}
    \dprojective[\V] \simeq \big(\W'\setminus \{0\}\big)/\R^\times,
\end{equation*}
where~$\W' = \dual{\V}\oplus\R$. Any element~$y\in \dprojective[\V]$ can be written in homogeneous coordinates as a class~$y = [\xi:t]$ with~$(\xi,t)\in \W'\setminus \{0\}$. We call \emph{hyperplane at infinity}, denoted by~$h_\infty$, the element~$[0:1]\in\dprojective[\V]$ where $0$ is understood here as an element of $\dual{\V}$. There is a bijection between~$\dprojective[\V]\setminus\{h_\infty\}$ and the set of affine hyperplanes of~$\V$ which sends a class~$[\xi:t]\ne [0:1]$ to the affine hyperplane~$\xi^{-1}(t)\subseteq \V$.

\begin{defi}[\textnormal{\cite[Def.~4.1]{S21}}]
    Let~$\phi\in\CF$. We say that~$\phi$ is \emph{constructible up to infinity} if:
    \begin{enumerate}
        \item for all~$m\in \Z$,~$j(\phi^{-1}(m))$ is subanalytic in~$\projective[\V]$,
        \item the family~$\{\phi^{-1}(m)\}_{m\in\Z}$ is finite.
    \end{enumerate}
\end{defi} 
\noindent We denote by~$\CF[\infty]$ the group of functions that are constructible up to infinity.
\begin{ex}\label{ex:cpctly-supped-CF-infty}
	Any compactly supported constructible function on~$\V$ is constructible up to infinity. Indeed, the open embedding~$j : \V \hookrightarrow\projective[\V]$ is proper on any compact subanalytic subset~$K$ of~$\V$ so that~$j(K)$ is subanalytic in~$\projective[\V]$ \cite[Prop.~8.2.2.(iii)]{KS90}.
\end{ex}

\subsection{Convexes, cones and polyhedra}

In this section, we define a subclass of constructible functions of prime interest in applications:~$\PL$-constructible functions. We refer to R. Rockafellar \cite{R15} for a complete study of convex sets and functions and to R. Schneiders \cite[Chap.~1]{Schnei14} for a short and clear exposition. We call \emph{convex polyhedron} an intersection of a finite number of open or closed affine half-spaces of~$\V$, and \emph{convex polytope} a bounded convex polyhedron. A \emph{polyhedral cone} is a convex polyhedron that is also a cone. If~$C$ is a non-empty convex subset of~$\V$, we denote by~$h_C : \dual{\V} \to \R\cup \{+\infty\}$ its support function, defined by~$h_C(\xi) = \sup_{x\in C} \dualdot{\xi}{x}$. The following lemma is clear, yet useful all along the paper, especially in \Cref{sec:regularity,sec:compatibility,sec:persistent-cohomology}.
\begin{lem} 
    \label{lem:pushforward-convex}
    Let~$C\subseteq \V$ be a closed convex subset. Then, for any~$\xi\in \dual{\V}$ proper on~$C$, we have~$\xi_*\1_{C} = \1_{[-h_C(-\xi), h_C(\xi)]}$. 
\end{lem}
\begin{nota}
    For the sake of simplicity, we abusively denoted~$[x,+\infty] := [x,+\infty)$ and~$[-\infty,x] := (-\infty,x]$, for~$x\in\R$, and~$[-\infty,+\infty] := \R$.
\end{nota}

A function~$\phi : \V \to \Z$ is said to be \emph{$\PL$-constructible} if there exists a finite covering~$\V = \bigsqcup_{a\in A}P_a$ by convex polyhedra such that~$\phi$ is constant on~$P_a$, for any~$a\in A$. Any such function can be written as a finite sum of indicator functions of \emph{closed} convex polyhedra. We denote by~$\CF[\PL]$ the group of~$\PL$-constructible functions,~$\CF[\PL,\cpct]$ the subgroup of compactly supported ones. Note that~$\CF[\PL,\cpct] \subset \CF[\PL] \subset \CF[\infty]$.

\subsection{$\gamma$-constructible functions}
We define here a specific class of constructible functions of particular interest in this paper due to their occurence in the context of sublevel-sets persistence, the so-called \emph{$\gamma$-constructible functions}. Let~$\gamma$ be a cone of~$\V$ such that:
\begin{equation}\tag{C1}
    \label{hyp:cone}
    \gamma \textit{ is a subanalytic closed proper convex cone with non-empty interior}.
\end{equation}

We say that a subset $U\subseteq \V$ is \emph{$\gamma$-open} if it is open and $U = U + \gamma$. The collection of $\gamma$-open subsets of $\V$ yields a topology on $\V$ called the \emph{$\gamma$-topology}, see \cite[Sec.~3.5]{KS90}. The closed subset of $\V$ for this topology, called $\gamma$-closed subsets, are the closed subsets~$S$ of~$\V$ such that~$S = S + \antipodal{\gamma}$. A subset of~$\V$ is called \emph{$\gamma$-locally closed} if it is the intersection of a $\gamma$-closed subset and a $\gamma$-open subset.
\begin{defi}
    \label{def:CF-g}
    We say that~$\phi\in\CF$ is \emph{$\gamma$-constructible} if 
    \begin{enumerate}
        \item~$\phi^{-1}(m)$ is subanalytic~$\gamma$-locally closed in~$\V$ for all~$m\in \Z$, and
        \item the family~$\left\{\phi^{-1}(m)\right\}_{m\in\Z}$ is locally finite.
    \end{enumerate}
\end{defi}
We denote by~$\CF[\g]$ the group of~$\gamma$-constructible functions on~$\V$ and~$\CF[\cpct,\g]$ the group of compactly supported~$\gamma$-constructible functions on~$\V$. 
\begin{lem}
    \label{lem:form-CF-g}
    Any compactly supported~$\gamma$-constructible function~$\phi$ on~$\V$ can be decomposed as a sum~$\phi = \sum_{\alpha\in A}m_\alpha \1_{Z_\alpha}$ where the set~$A$ is finite and the subsets~$Z_\alpha\subseteq \V$ are relatively compact subanalytic and~$\gamma$-locally closed.
\end{lem}
\begin{proof}
    Since the class of subanalytic~$\gamma$-locally closed subsets is stable under intersection, one can intersect the~$Z_\alpha$'s with the subset:
    \begin{equation*}
        Z_v := \left(v + \Int(\gamma)\right) \cap \left(-v + \antipodal{\gamma}\right),
    \end{equation*}
    where~$v$ is any element of the non-empty set~$\Int(\antipodal{\gamma})$ chosen so that~$\supp(\phi)\subseteq Z_v$. To choose such a~$v\in \Int(\antipodal{\gamma})$, consider any~$v_0\in \Int(\antipodal{\gamma})$, and remark that since~$\supp(\phi)$ is compact and~$Z_{v_0}$ has non-empty interior \cite[Lem.~5.7]{BP19}, there exists a~$\lambda>0$ such that~$\supp(\phi)\subseteq \lambda \cdot Z_{v_0}$.  
    Moreover, it is easy to show that~$\lambda \cdot Z_{v_0} = Z_{\lambda v_0}$
    so that~$v = \lambda v_0 \in \Int(\antipodal{\gamma})$ works. The subset~$Z_v$ is then clearly subanalytic and~$\gamma$-locally closed. Moreover, loc. cit. yields that~$Z_v$ is bounded, hence relatively compact. The fact that~$A$ can be chosen finite follows then from the locally finiteness of the sum.
\end{proof}
Schapira proved a characterization of $\gamma$-constructible functions of which we state below a specific case. We say that a closed set $A\subseteq \V$ is \emph{$\gamma$-proper} if the addition $s:\V\times\V\to\V$ is proper on $A \times \antipodal{\gamma}$.

\begin{prop}[\textnormal{\cite[Prop.~4.18]{S21}}]
    \label{prop:charac-gamma-CF}
    Let $\phi\in\CF[\infty]$ such that $\supp(\phi)$ is $\gamma$-proper. The function $\phi$ is $\gamma$-constructible if and only if $\phi = \phi \conv \1_{\antipodal{\gamma}}$. 
\end{prop}

As an easy consequence of this characterization, the following lemma states that the group of $\gamma$-constructible functions is closed under specific pushforwards. This fact will be useful in \Cref{sec:comp-conv-and-box,sec:reconstruction-EFourier}.
\begin{lem}
    \label{lem:pushforward-g-CF-polant}
    Let~$\phi\in\CF[\cpct,\g]$. For any~$\xi\in\dualcone{\antipodal{\gamma}}\setminus\{0\}$, one has~$\xi_*\phi\in\CF[\cpct,\g[\lambda]][\R]$ with~$\lambda = \R_{\leq 0}$. 
\end{lem}

\begin{proof}
    \Cref{lem:support-pushforward} shows that~$\supp(\xi_*\phi)\subseteq \xi(\supp(\phi))$, hence~$\xi_*\phi$ is compactly supported. Since $\xi_*\1_{\antipodal{\gamma}} = \1_{\antipodal{\lambda}}$, \Cref{prop:charac-gamma-CF} and \Cref{cor:pushforward-linear-conv} yield:
    \begin{equation*}
        \xi_*\phi = \xi_*(\phi\conv\1_{\antipodal{\gamma}})=(\xi_*\phi) \conv \1_{\antipodal{\lambda}}.
    \end{equation*}
    Hence, $\phi\in\CF[\g[\lambda]][\R]$ by \Cref{prop:charac-gamma-CF}.
\end{proof}

\begin{rk}
    It follows directly from the previous lemma that for any~$\xi\in\dualcone{\gamma}\setminus\{0\}$, one has~$\xi_*\phi\in\CF[\cpct,\g[\lambda]][\R]$ with~$\lambda = \R_{\geq 0}$. We will see in \Cref{prop:supp-Radon-g} that for any other~$\xi\in \dual{\V}\setminus(\dualcone{\gamma}\cup\dualcone{\antipodal{\gamma}})$, one has~$\xi_*\phi = 0$. 
\end{rk}

\section{Hybrid transforms}\label{sec:hybrid-transforms}
\subsection{General definition}
We now introduce the notion of hybrid transform, which is central to this article. Let us denote by~$\mathrm{F}(\dual{\V};\R)$ the set of functions from~$\dual{\V}$ to~$\R$. 
\begin{defi}
    \label{def:hybrid-transforms}
    For~$\kappa\in\Lloc$, the transform~$\transform*:\CF[\cpct]\rightarrow \mathrm{F}(\dual{\V};\R)$ is defined by: 
    \begin{equation*}
        \transform<\xi> = \int_{\R} \kappa(t)\,\xi_*\phi (t)\d t,
    \end{equation*}
    for any~$\phi\in\CF[\cpct]$ and any~$\xi\in\dual{\V}$.
\end{defi}
%
%
\begin{rks}[Generalizations]
    \label{rk:generalization-ht}
    In the rest of the article, we will consider the following generalizations of~$\transform*$ whenever necessary.
    \begin{enumerate}
        \item The transform~$\transform[\phi]$ is well-defined for~$\phi$ with non-compact support on the set of~$\xi\in\dual{\V}$ such that (i)~$\xi$ is proper on~$\supp(\phi)$ and (ii)~$\kappa\cdot \xi_*\phi\in\Lp$. 
        \item More generally, the definition of~$\transform<\zeta>$ still makes sense for any~$\phi\in\CF[][X]$ with~$X$ a real analytic manifold, and any morphism~$\zeta : X\to\R$ of real analytic manifolds such that (i)~$\zeta$ is proper on~$\supp(\phi)$ and (ii)~$\kappa\cdot \zeta_*\phi\in\Lp$.
    \end{enumerate}
\end{rks}

\begin{nota}
    When~$A\subseteq \R$, we use the simpler notation~$\transform[\phi]<\xi>[A] := \transform[\phi]<\xi>[1_A]$.
\end{nota}
%


In the course of this article, we will illustrate on many examples the interest that there can be in considering hybrid transforms. All the examples will illustrate the effect of the combination of a topological operation (the pushforward) and of a classical integral. We start with the following example.

\begin{ex}[Subanalytic curve]
    Let~$c:[0,1]\to\V$ be continuous subanalytic. One can consider the constructible function~$\1_Z$ where~$Z = \Ima(c)$ is compact and subanalytic in~$\V$. Since~$Z$ has volume zero, integral transforms using only the Lebesgue measure are zero when~$\dim(\V) \geq 2$. However, hybrid transforms are generally not, as highlighted by \Cref{fig:EF-spiral} and \Cref{ex:EL-sphere,ex:EL-sqcrck,ex:EF-sphere,ex:EF-sqcrck}. This is due to the fact that the pushforwards of~$\1_Z$ by linear forms convey topological information that is missed by the Lebesgue measure.
\end{ex}

As explained in introduction, \Cref{def:hybrid-transforms} generalizes existing transforms, starting with the following two.

\begin{ex}
	\label{ex:def-GR-EFourier}
    Considering~$\phi\in\CF[\cpct]$ and~$\eta\in\dual{\V}$, the \emph{Fourier transform} defined in \cite{GR11} is the hybrid transform:
    \begin{equation*}
        \GREFourier<\eta> = 
        \int_{0}^{+\infty} \int_\V \1_{\eta^{-1}(t)}\phi\d\Euler \d t = \transform[\phi]<\eta>[{\R_{\geq 0}}].
    \end{equation*}
    We call it \emph{GR-Euler-Fourier transform} in this paper, keeping the terminology \emph{Euler-Fourier} for the transform introduced in \Cref{def:EFourier}. Using \Cref{lem:xi-proper-supp-phi}, one can see that any constructible function~$\phi$ satisfying \Cref{ass:EL-well-def} below for a cone~$C$ satisfies also that~$\1_{\R_{\geq 0}}\cdot\eta_*\phi\in\Lp$ for any~$\eta\in\intpolant{C}$, so that we can extend the definition of the transform to such constructible functions~$\phi$ and linear forms~$\eta$. 
\end{ex}

\begin{ex}
    \label{ex:def-Euler-Bessel}
	Given an analytic norm on~$\V$, the \emph{Euler-Bessel transform} is defined in \cite{GR11} for any~$\phi\in\CF[\cpct]$ and any~$v\in \V$, by:
    \begin{equation*}
        \EBessel<v> 
        = \int_{0}^{+\infty} \int_\V \1_{\partial B(v,t)}\phi\d\Euler \d t,
    \end{equation*}
    where~$\partial B(v,t)$ denotes the sphere of radius~$t$ centered at~$v$ in~$\V$. Considering the morphism of real analytic manifolds~$\zeta_v = \|v-\cdot\|^2$ 
	and the locally integrable kernel~$\kappa : t\mapsto \1_{(0,+\infty)}(t)\cdot \frac{1}{2\sqrt{t}}$, the Euler-Bessel transform of~$\phi\in\CF[\cpct]$ is the hybrid transform~$\EBessel<v> = \transform[\phi]<\zeta_v>$.
\end{ex}

\subsection{Euler-Laplace transform}\label{sec:ELaplace}
\begin{defi}
    \label{def:ELaplace}
    The \emph{Euler-Laplace transform} of~$\phi\in\CF[\cpct]$ is defined for~$\xi\in\dual{\V}$ by:
    \begin{equation*}
        \ELaplace<\xi> = \int_{\R} e^{-t}\xi_*\phi(t) \d t.
    \end{equation*}
\end{defi}
As explained in \Cref{rk:generalization-ht}, we may extend the definition of the Euler-Laplace transform for any constructible function~$\phi\in\CF$ on the set of~$\xi\in\dual{\V}$ that are proper on~$\supp(\phi)$ and for which the right-hand integral is well-defined. We show in \Cref{prop:EL-well-def} that the following assumptions on~$\phi$ are sufficient for the set of such~$\xi$'s to contain an open cone of~$\dual{\V}$. 


\begin{assus}$~$
    \label{ass:EL-well-def}
    \begin{enumerate}
        \item\label{itm:CF-infty}~$\phi$ is constructible up to infinity,
        \item\label{itm:supp}~$\supp(\phi) \subseteq K + C$, where~$K$ is compact and~$C$ is a cone satisfying:
        \begin{equation}\tag{C2}
            \label{hyp:cone-empty-int}
            C \textit{ is a non-empty subanalytic closed proper convex cone}.
        \end{equation}
    \end{enumerate}
\end{assus}
\noindent Note that, compared to~\eqref{hyp:cone}, the cone~$C$ may have empty interior.
\begin{rk}
	\label{rk:ass-conv-cone}
	If there exist a compactly supported constructible function~$\phi_c$ on~$\V$ and a cone~$C$ satisfying~\eqref{hyp:cone-empty-int} such that~$\phi = \phi_c \conv \1_C$, then~$\phi$ satisfies \Cref{ass:EL-well-def} for the cone~$C$. The property on the support is easy to prove. Moreover,~$\phi_c$ is constructible up to infinity (\Cref{ex:cpctly-supped-CF-infty}) and so is~$\1_C$ by \cite[Lem.~2.17]{S21} since~$C$ is a subanalytic cone. The result follows then from the stability of this property by convolution, see Section~3.4 in loc. cit..
\end{rk}
\begin{ex}
	By the previous remark and \Cref{prop:sublevelCF-as-gammaification}, sublevel-sets constructible functions (\Cref{sec:def-sublevel-CF}) satisfy these assumptions. 
\end{ex}
\begin{prop}
    \label{prop:EL-well-def}
    If~$\phi\in\CF$ satisfies \Cref{ass:EL-well-def} for a cone~$C$, then~$\ELaplace$ is well-defined on~$\Int(\dualcone{C})$.
\end{prop}

\noindent The proposition follows from the next lemma.
\begin{lem}
    \label{lem:xi-proper-supp-phi}
    Let~$\phi\in\CF$. If~$\phi$ satisfies~\Cref{ass:EL-well-def}.\ref{itm:supp}, then any~$\xi\in\Int(\dualcone{C})\cup\intpolant{C}$ is proper on~$\supp(\phi)$. Moreover, for any~$\xi \in \Int(\dualcone{C})$ (resp.~$\xi \in \intpolant{C}$), one has~$\supp(\xi_*\phi) \subseteq [a,+\infty)$ (resp.~$\supp(\xi_*\phi) \subseteq (-\infty,a]$) for some $a\in\R$.
\end{lem}

\begin{proof}
    Let~$a< b$ be two real numbers and~$\xi\in \Int(\dualcone{C})$. The case~$\xi\in\intpolant{C}$ follows by multiplication by~$-1$. We prove that the space~$\xi^{-1}[a,b]\cap \supp(\phi)$ is compact. Since this space is closed and by assumption:
    \begin{equation*}
        \xi^{-1}[a,b]\cap \supp(\phi) \subseteq \xi^{-1}[a,b]\cap \left(K+C\right),
    \end{equation*}
    it is enough to show that this last space is compact. Suppose that there exist a sequence~$y_n= k_n + x_n\in \xi^{-1}[a,b]\cap \left(K+C\right)$ with~$k_n \in K$ and~$x_n \in C$ such that~$\|y_n\|\longrightarrow +\infty$. Since~$K$ is compact, we also have~$\|x_n\|\longrightarrow +\infty$. If we denote by~$S$ the unit sphere of~$\V$, the subset~$C\cap S$ is compact by closedness of~$C$. We can thus assume without loss of generality that~$x_n/\|x_n\| \longrightarrow v \in C\cap S$. Dividing the inequality~$a \leq \dualdot{\xi}{y_n} \leq b$ by~$\|x_n\|$ and taking the limit, we get that~$\dualdot{\xi}{v} = 0$. Yet,~$\xi\in \Int(\dualcone{C})$ so that~$\dualdot{\xi}{v} >0$, a contradiction.
\end{proof}

\begin{proof}[Proof of \Cref{prop:EL-well-def}]
    For any~$\xi\in \Int(\dualcone{C})$, it follows from \Cref{lem:xi-proper-supp-phi} that~$\xi$ is proper on~$\supp(\phi)$ and that~$\supp(\xi_*\phi)\subseteq [a,+\infty)$ for some $a\in\R$. The function~$\xi_*\phi$ is constructible up to infinity by \cite[Lem.~4.10]{S21}, thus takes only a finite number of distinct values. Together with the property on the support of~$\xi_*\phi$, this yields that~$t\in\R \mapsto e^{-t}\cdot\xi_*\phi(t)$ is integrable over~$\R$, hence the result.
\end{proof}
%
Before moving on to the examples, we present the simple relation between the Euler-Laplace transform and the usual Laplace transform, useful all along the paper.
\begin{lem}
    \label{lem:link-ELaplace-Laplace}
    If~$\phi\in\CF$ satisfies \Cref{ass:EL-well-def} for a cone~$C$, then for any~$\xi\in\Int(\dualcone{C})$ and any~$s>0$, one has 
    \begin{equation*}
        \ELaplace<s\,\xi> = \ELaplace[\xi_*\phi]<s> = s\,\Laplace[\xi_*\phi](s),
    \end{equation*}
    where~$\Laplace f (s) = \int_\R e^{- st} f(t) \d t$ is the classical (bilateral) Laplace transform.
\end{lem}
\begin{proof}
    The first equality follows from the functoriality of the pushforward:~$(s\,\xi)_*\phi = s_*\xi_*\phi$, and the second one from the fact that~$s_*\xi_*\phi(t) = \xi_*\phi(t/s)$, leading to the computation:
    \begin{equation*}
        \ELaplace[\xi_*\phi]<s> = s \int_{\R} e^{-su} \xi_*\phi(u)\d u = s\,\Laplace[\xi_*\phi](s).
    \end{equation*}
\end{proof}
We now turn to the examples. Once again, the reader's attention is drawn 
to the effect of the successive application of topological pushforward and of classical integral.
\begin{ex}[Interval]
    \label{ex:EL-interval}
    Let~$-\infty < a\leq b < +\infty$ and consider the constructible function~$\1_{\lceil a,b\rfloor}$ where~$\lceil a,b\rfloor$ is one of the intervals~$[a,b]$,~$(a,b]$,~$[a,b)$, or~$[a,b]$. We have, for any~$\xi\in\R$,
    \begin{equation*}
        \ELaplace[\1_{\lceil a,b\rfloor}]<\xi> = \sgn(\xi)\left(e^{-\xi\cdot a}-e^{-\xi\cdot b}\right),
    \end{equation*}
    and for~$\xi \in \R_{>0}$,
    \begin{equation*}
        \ELaplace[\1_{\lceil a,+\infty)}] <\xi> = e^{-\xi\cdot a}.
    \end{equation*}
\end{ex}
\begin{ex}[Rectangle]
    \label{ex:EL-char-rectangle}
    Let~$a < b$ and~$c < d$ be real numbers. We have, for~$\xi\in\dual{\left(\R^2\right)}$ such that~$\xi(a,c) < \xi(a,d) < \xi(b,c) < \xi(b,d)$, the formulae:
    \begin{equation}
        \label{eq:pushforward-examples}
        \begin{array}{lcr}
            \xi_*\1_{[a,b)\times[c,d)} = \1_{[\xi(a,c),\xi(a,d))} - \1_{[\xi(b,c),\xi(b,d))}, &\quad& \xi_*\1_{[a,b)\times[c,d]} = \1_{[\xi(a,c),\xi(b,c))}, \\[0.5em]
            \xi_*\1_{(a,b]\times[c,d]} = \1_{(\xi(a,d),\xi(b,d)]}, && \xi_*\1_{[a,b]\times[c,d]} = \1_{[\xi(a,c),\xi(b,d)]}.
        \end{array}
    \end{equation}
    Indeed, note that~$\xi_*\1_{[a,b]\times[c,d]} = \1_{[\xi(a,c),\xi(b,d)]}$ is given by~\Cref{lem:pushforward-convex} and the other equalities are obtained by additivity. For instance, $\1_{(a,b]\times[c,d]} = \1_{[a,b]\times[c,d]} - \1_{[a,a]\times[c,d]}$. Equation~\eqref{eq:pushforward-examples} yields:
    \begin{align*}
        \ELaplace[\1_{[a,b)\times[c,d)}]<\xi> &= e^{-\xi(a,c)} - e^{-\xi(a,d)} - e^{-\xi(b,c)} + e^{-\xi(b,d)}, \\[0.5em]
        \ELaplace[\1_{[a,b)\times[c,d]}]<\xi> &= e^{-\xi (a,c)} - e^{-\xi (b,c)}, \\[0.5em]
        \ELaplace[\1_{(a,b]\times[c,d]}]<\xi> &= e^{-\xi (a,d)} - e^{-\xi (b,d)}, \\[0.5em]
        \ELaplace[\1_{[a,b]\times[c,d]}]<\xi> &= e^{-\xi (a,c)} - e^{-\xi (b,d)}. 
    \end{align*} 
    Similar formulae can be obtained when~$\xi$ induces a different order on the vertices of the rectangle~$[a,b]\times[c,d]$. Unlike the classical Laplace transform, the Euler-Laplace transform distinguishes between the presence or absence of the edges of the rectangle.
\end{ex}
\begin{rk}
    The last Euler-Laplace transform computed in \Cref{ex:EL-char-rectangle} yields a counter-example to the formula~$\ELaplace[\phi\boxtimes\psi] = \ELaplace[\phi]\boxtimes\,\ELaplace[\psi]$, which is wrong in general. We will give a correct formula in \Cref{cor:comp-box-EL}.
\end{rk}
\begin{ex}[Sphere]
    \label{ex:EL-sphere}
    Assume that~$\V$ is equipped with a norm~$\|\cdot\|$. Consider~$r >0$ and~$\phi = \1_{\mathbb{S}_r}$ with~$\mathbb{S}_r = \{x\in \V\st \|x\|=r\}$. For any~$\xi\in\dual{\V}$, we have:
    \begin{equation*}
        \xi_*\1_{\mathbb{S}_r} = \left(1+(-1)^{\dim(\V)}\right)\cdot\1_{\left(-r\|\xi\|,r\|\xi\|\right)} + \1_{\left\{-r\|\xi\|\right\}} + \1_{\left\{r\|\xi\|\right\}},
    \end{equation*}
    and hence:
    \begin{equation*}
        \ELaplace[\1_{\mathbb{S}_r}](\xi) = 2\cdot \left(1+(-1)^{\dim(\V)}\right)\cdot \sinh\left(r\|\xi\|\right).
    \end{equation*}
    Note the amount of information extracted by this transform even though the constructible function under consideration is supported on a subset with zero volume.
\end{ex}
The Euler-Laplace and Laplace transforms are equal up to a normalization on some~$\gamma$-constructible functions (\Cref{ex:EL-gamma-voxels}). However, they are not on all~$\CF[\cpct,\g]$, as shown by the following example.
\begin{ex}[$\gamma$-triangle]
    \label{ex:EL-gamma-triangle}
    Let~$b\in\R$. Consider the triangle
    \begin{equation*}
        T = \Conv{\{(0,0),(1,0),(0,b)\}}\setminus\left\{(x,y)\in\R^2\st y = b-bx\right\},
    \end{equation*}
    represented in \Cref{fig:gT}. The subset~$T\subset\R^2$ is subanalytic and~$\gamma$-locally closed for the cone~$\gamma = (\R_{\leq 0})^2$. Then, the Euler-Laplace and classical Laplace transforms compare as follows. For~$\xi = (\xi_x, \xi_y)\in\left(\R_{\geq 0}\right)^2$,
    \begin{equation*}
        \ELaplace[\1_T]<\xi> = 
        \begin{cases}
            1 - e^{-\xi_x} &\mbox{if } b\xi_y \geq \xi_x, \\
            1 - e^{-b\xi_y} &\mbox{otherwise},
        \end{cases}
    \end{equation*}
    and 
    \begin{equation*}
        \Laplace\left[\1_T\right](\xi) \overset{\mathrm{def}}{=} \int_{\R^2} e^{-\dualdot{\xi}{x}}\,\1_T(x)\d x = \frac{\xi_x \left(1 - e^{-b\xi_y}\right) + b\xi_y\left(e^{-\xi_x} - 1\right)}{\xi_x(\xi_x - b\xi_y)\xi_y}.
    \end{equation*}
    These two transforms differ, as~$\ELaplace[\1_T]<\xi>$ does not depend on~$b\geq 1$ for any~$\xi\in\left(\R_{\geq 0}\right)^2$ such that~$\xi_y \geq \xi_x$. Similar formulae hold for other choices of~$\xi\in\R^2$.
\end{ex}
\begin{figure}[ht]
    \centering
    \includegraphics[scale = 0.8]{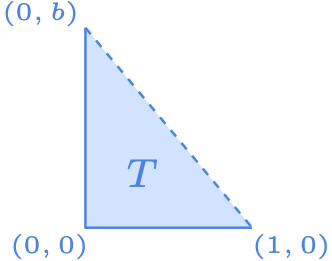} 
    \caption{The triangle~$T$ defined in \Cref{ex:EL-gamma-triangle} is represented as the light blue solid triangle and the dark blue solid angle defined by the points~$(0,b)$,~$(0,0)$ and~$(1,0)$. The dashed line indicates that points on this edge do not belong to~$T$.} 
    \label{fig:gT}
\end{figure}
\begin{ex}[Closed square minus a curve]
    \label{ex:EL-sqcrck}
    Consider the constructible function
    \begin{equation*}
        \phi = \1_S - \1_C,
    \end{equation*}
    where~$S = [-1/2,1/2]^2$ and~$C$ is the closed curve of~$\R^2$ represented by the dotted line in \Cref{fig:sqcrck}. Since~$C$ has zero volume, the classical Laplace transforms of~$\1_S$ and~$\1_S - \1_C$ are equal. However, their Euler-Laplace transforms differ, as shown in \Cref{fig:EL-sqcrck}.
\end{ex}

\begin{figure}[ht] 
    \centering
    \begin{subfigure}[b]{0.46\linewidth}
        \includegraphics[width=.83\linewidth]{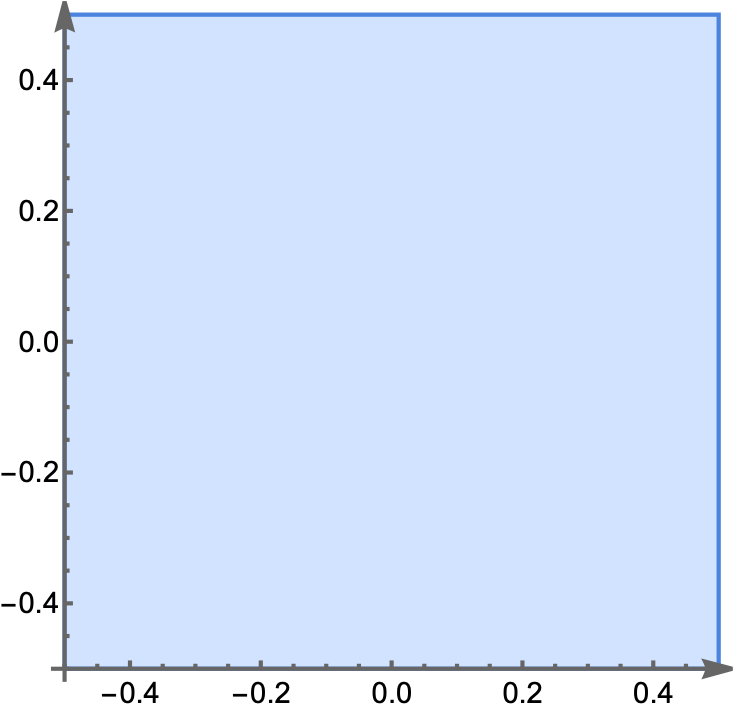} 
        \vspace{0.5em}
        \caption{$\1_S$} 
      \end{subfigure}
      \begin{subfigure}[b]{0.46\linewidth}
        \begin{flushright}
          \includegraphics[width=0.98\linewidth]{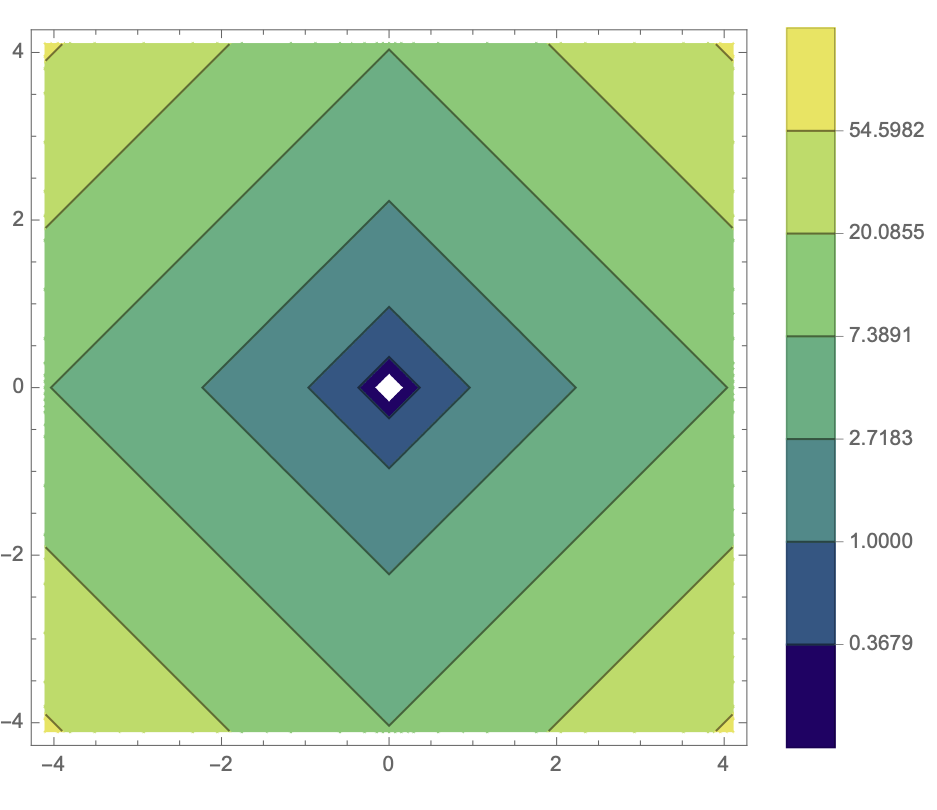} 
          \caption{$\ELaplace[\1_S]$} 
        \end{flushright}
      \end{subfigure} 
      \begin{subfigure}[b]{0.46\linewidth}
        \includegraphics[width=.85\linewidth]{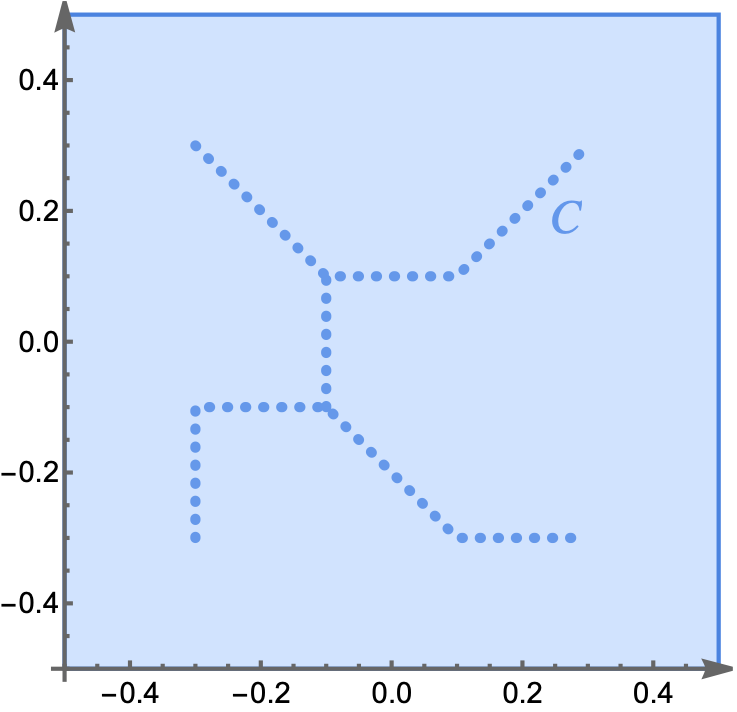} 
        \caption{$\1_S - \1_C$}
        \label{fig:sqcrck}
      \end{subfigure}
      \begin{subfigure}[b]{0.46\linewidth}
        \begin{flushright}
          \includegraphics[width=0.98\linewidth]{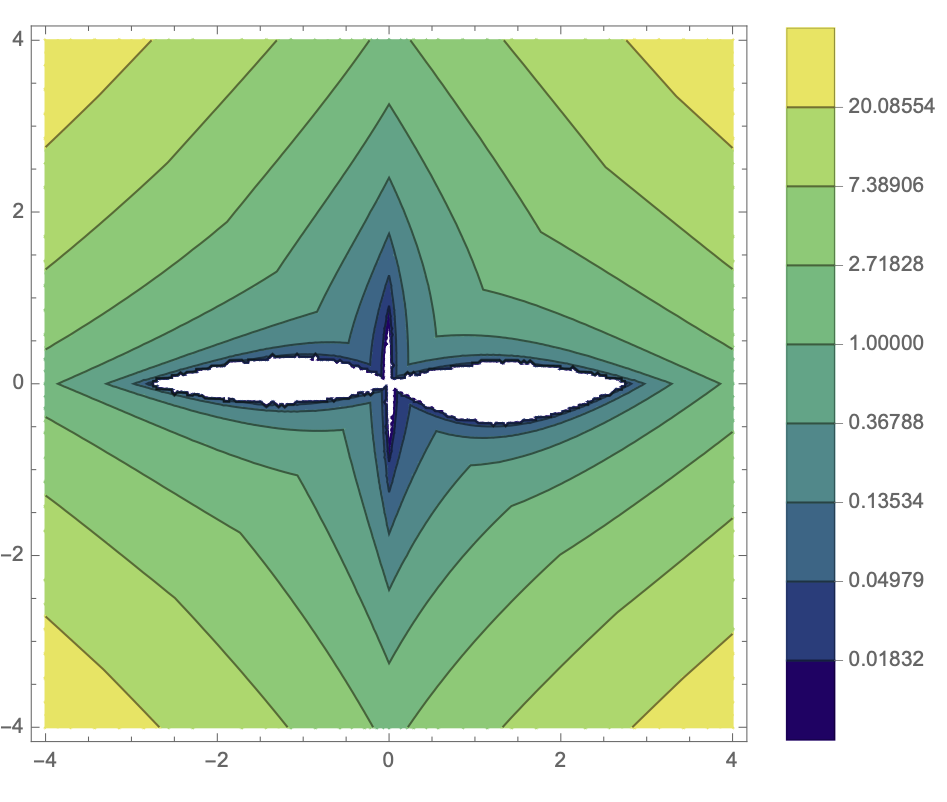} 
          \caption{$\ELaplace[\1_S - \1_C]$} 
        \end{flushright}
      \end{subfigure} 
    \caption{Euler-Laplace transforms of the constructible functions~$\1_S$ and~$\1_S-\1_C$ in \Cref{ex:EL-sqcrck}. The square~$S$ is represented by the light blue solid square and the closed curve~$C$ is represented by the dark blue dotted curve.}
    \label{fig:EL-sqcrck}
\end{figure}
%
%
\paragraph{Generalization of persistent magnitude to constructible sheaves.} 
Recently, Govc and Hepworth introduced the notion of \emph{persistent magnitude} for one-parameter persistence modules \cite[Def.~5.1]{GH21}. They also present a generalization of it to multi-parameter persistence modules using the classical Laplace transform. Here, we use the Euler-Laplace transform to give an alternative generalisation of persistent magnitude that benefits from the compatibility (\Cref{sec:compatibility}) and index formulae (\Cref{thm:index-theoretic-formulae}) that come from hybridization. 

The theory of multi-parameter persistence has been formulated in the language of derived sheaf theory by Kashiwara and Schapira in \cite{KS18A}. Following this formulation (and their notations), we consider a cone~$\gamma$ satisfying~\eqref{hyp:cone} and call \emph{multi-parameter persistence module on~$\V$} a constructible~$\gamma$-sheaf~$F\in\Db[\Rc,\gf]$. This definition is essentially equivalent to the usual one, missing only the so-called \emph{ephemeral persistence modules}; see \cite{BP19} for a detailed comparison of the two approaches. 

Below, we define the notion of magnitude not only for multi-parameter persistence modules but also for (derived) constructible sheaves. 

\begin{defi}
    The \emph{magnitude} of a sheaf~$F\in\Db[\cpct,\Rc]$ is the Euler-Laplace transform~$\ELaplace[\EPind(F)]$ where~$\EPind(F)\in\CF$ denotes the \emph{local Euler-Poincaré index} of~$F$ (see \cite[Sec.~9.7]{KS90}), defined for any~$x\in\V$ by:
    \begin{equation*}
        \EPind(F)(x)=\sum_{j\in\Z}(-1)^j\dim H^j(F_x).
    \end{equation*}
\end{defi}
\begin{rk}
    More generally, the magnitude of~$F\in\Db[\Rc]$ is well-defined when (i)~$F$ is constructible up to infinity in the sense of \cite[Def.~2.8]{S21} and (ii)~$\supp(F) \subseteq K + C$ with~$K$ compact and~$C$ a cone satisfying~\eqref{hyp:cone-empty-int}. Indeed, in that case,~$\EPind(F)$ is constructible up to infinity as proven by point~(i) in the proof of \cite[Thm.~4.4]{S21}, and satisfies \Cref{ass:EL-well-def}.\ref{itm:supp}.
\end{rk}
%
\begin{ex}
    For instance, if~$F \simeq F_c\conv \mathbf{k}_C$ with~$F_c$ compactly supported and constructible,
    then~$\EPind(F)$ satisfies \Cref{ass:EL-well-def} for the cone~$C$ by \Cref{rk:ass-conv-cone}.
\end{ex}

\begin{ex}
    In the case~$\V = \R$ and~$\gamma = \R_{\leq 0}$, we recover the definition introduced in \cite{GH21} of the persistent magnitude for 1-parameter persistence modules. Indeed, considering~$F\simeq \bigoplus_{i=1}^n \mathbf{k}_{[a_i,b_i)}$ with~$-\infty < a_i\leq b_i \leq +\infty$, \Cref{ex:EL-interval} yields for~$t\in \R_{>0}$:
    \begin{equation*}
        \ELaplace[\EPind(F)](t) = \sum_{i=1}^n e^{-a_it}-e^{-b_it},
    \end{equation*}
    with the convention that~$e^{-\infty} = 0$. 
\end{ex}

\subsection{Euler-Fourier transform}\label{sec:EFourier}
\begin{defi}
    \label{def:EFourier}
    The \emph{Euler-Fourier transform} of~$\phi\in\CF[\cpct]$ is defined for~$\xi\in\dual{\V}$ by:
    \begin{equation*}
        \EFourier<\xi> = \int_{\R} e^{-it}\xi_*\phi(t) \d t.
    \end{equation*}
\end{defi}


The following lemma relates the Euler-Fourier transform to the usual Fourier transform, its proofs is analogous to that of \Cref{lem:link-ELaplace-Laplace}.
\begin{lem}
    \label{lem:link-EFourier-Fourier}
    Let~$\phi\in\CF[\cpct]$. Then, for any~$\xi\in\dual{\V}$ and any~$s\ne 0$, one has 
    \begin{equation*}
        \EFourier<s\,\xi> = \EFourier[\xi_*\phi]<s> = |s|\,\Fourier[\xi_*\phi](s),
    \end{equation*}
    where~$\Fourier f(s) = \int_{\R} e^{-ist} f(t) \d t$ is the classical Fourier transform.
\end{lem}


\begin{ex}[Sphere]
    \label{ex:EF-sphere}
    Consider the setting of \Cref{ex:EL-sphere}. Then, for any~$\xi\in\dual{\V}$,
    \begin{equation*}
        \EFourier[\1_{\mathbb{S}_r}](\xi) = 2\cdot \left(1+(-1)^{\dim(\V)}\right)\cdot \sin\left(r\|\xi\|\right).
    \end{equation*}
\end{ex}
\begin{ex}[$\gamma$-triangle]
    \label{ex:EF-gamma-triangle}
    Consider the setting of \Cref{ex:EL-gamma-triangle}. Then, the Euler-Fourier and classical Fourier transforms compare as follows. For any~$\xi = (\xi_x, \xi_y)\in\left(\R_{\geq 0}\right)^2$,
    \begin{equation*}
        \EFourier[\1_T]<\xi> = 
        \begin{cases}
            i(e^{-i\xi_x}-1) &\mbox{if } b\xi_y \geq \xi_x, \\
            i(e^{-ib\xi_y}-1) &\mbox{else},
        \end{cases}
    \end{equation*}
    and 
    \begin{equation*}
        \Fourier\left[\1_T\right](\xi) \overset{\mathrm{def}}{=} \int_{\R^2} e^{-i\dualdot{\xi}{x}}\,\1_T(x)\d x = \frac{\xi_x \left(1 - e^{-ib\xi_y}\right) + b\xi_y\left(e^{-i\xi_x} - 1\right)}{\xi_x(b\xi_y - \xi_x)\xi_y}.
    \end{equation*}
    Again, these two transforms differ, as~$\EFourier[\1_T]<\xi>$ does not depend on~$b\geq 1$ for any~$\xi$ such that~$\xi_y \geq \xi_x$. See \Cref{fig:EF-gT} for an illustration.
\end{ex}
\begin{figure}
    \centering
    \begin{subfigure}[b]{0.49\linewidth}
      \centering
      \includegraphics[width=.9\linewidth]{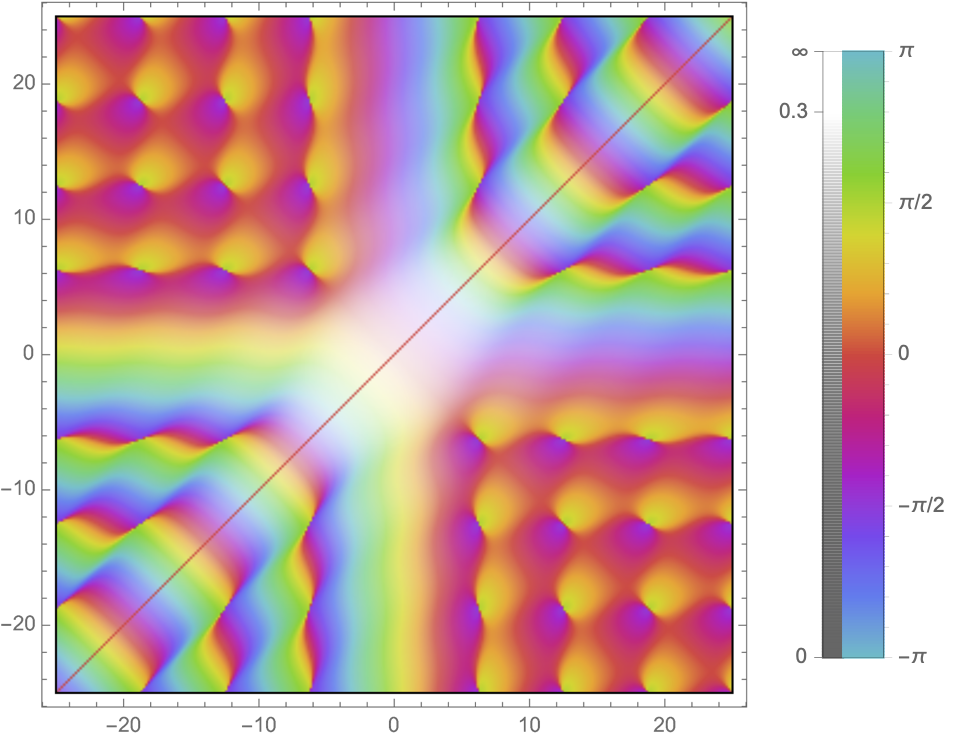} 
      \caption{$\Fourier[\1_T]$}
    \end{subfigure}
    \begin{subfigure}[b]{0.49\linewidth}
      \centering
      \includegraphics[width=.9\linewidth]{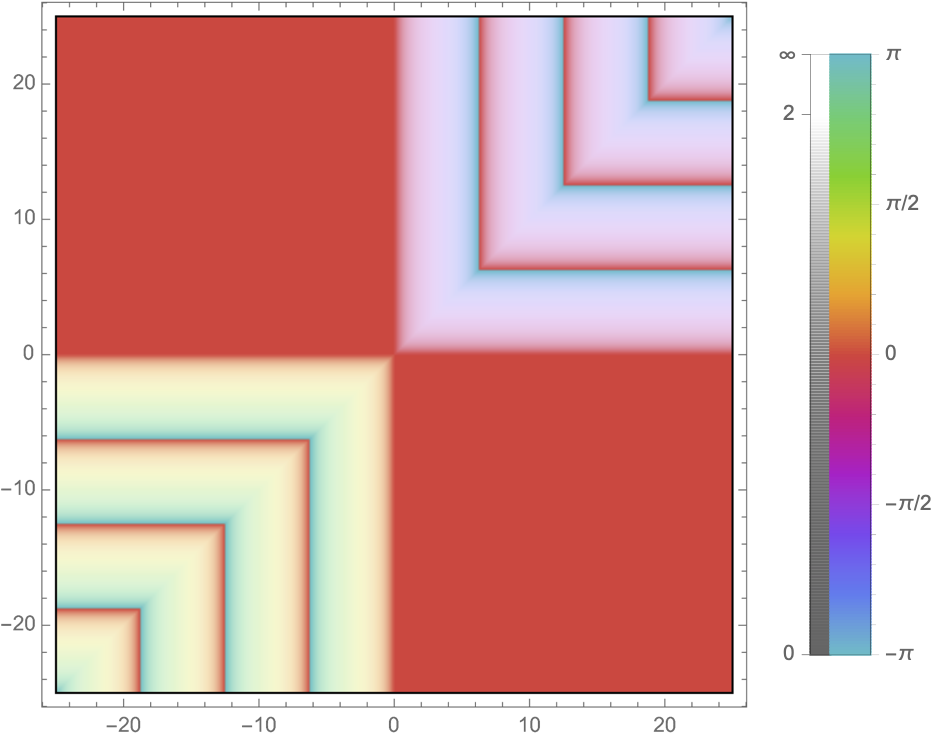}
      \caption{$\EFourier[\1_T]$} 
    \end{subfigure} 
    \caption{Fourier and Euler-Fourier transforms of the constructible function~$\1_T$ defined in \Cref{ex:EL-gamma-triangle}. Again, plots are done following \Cref{rk:computations}.}
    \label{fig:EF-gT} 
\end{figure}
\begin{ex}[Closed square minus a curve]
    \label{ex:EF-sqcrck}
    Consider the setting of \Cref{ex:EL-sqcrck}. Again, the (classical) Fourier transforms of~$\1_S$ and of~$\1_S - \1_C$ are equal. However, their Euler-Fourier transforms differ, as shown in \Cref{fig:EF-sqcrck}.
\end{ex}
\begin{figure}[ht] 
    \centering
    \begin{subfigure}[t]{0.5\linewidth}
        \centering
        \includegraphics[width=0.95\linewidth]{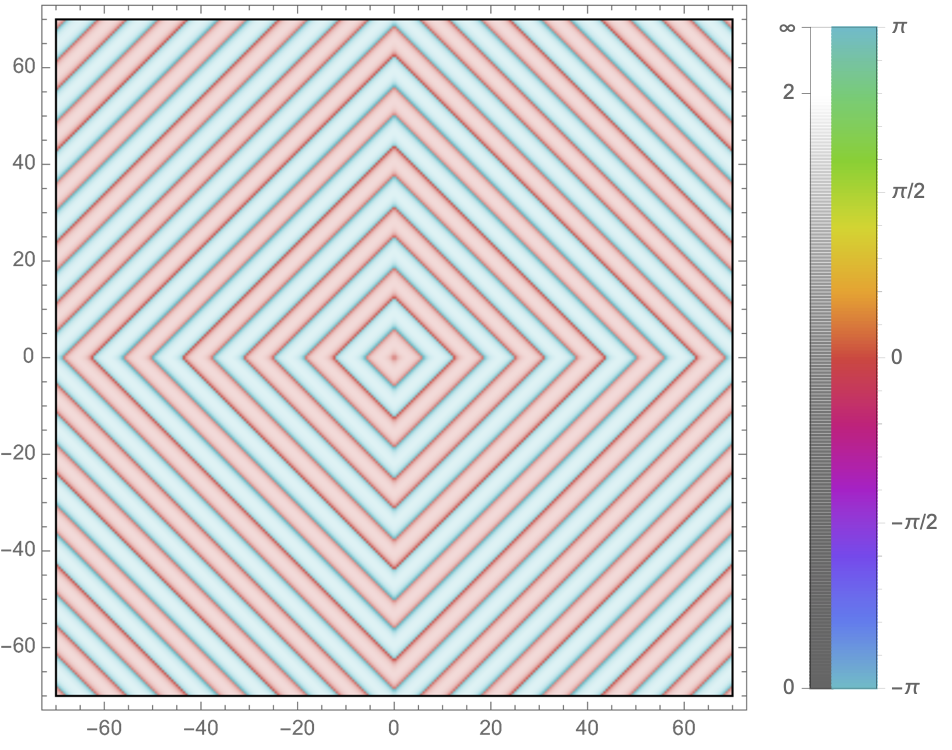} 
      \caption{$\EFourier[\1_S]$}
      \label{fig:EF-sqcrck-phi}
    \end{subfigure}
    \begin{subfigure}[t]{0.5\linewidth}
        \centering
        \includegraphics[width=0.93\linewidth]{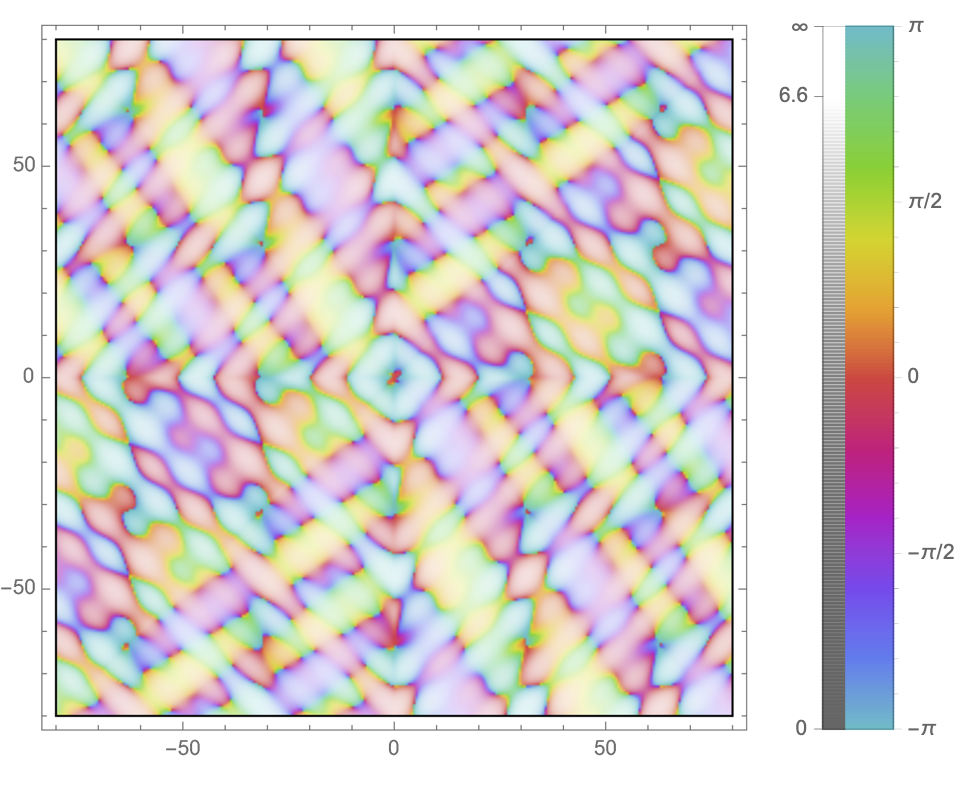} 
        \caption{$\EFourier[\1_S-\1_C]$} 
    \end{subfigure}
    \caption{Euler-Fourier transforms of the constructible functions~$\1_S$ and~$\1_S-\1_C$ in \Cref{ex:EL-sqcrck}. Again, plots are done following \Cref{rk:computations}.}
    \label{fig:EF-sqcrck} 
\end{figure}
\subsection{Computations}
In this section, we explain how to compute hybrid transforms of PL-constructible functions without computing any integral with respect to the Euler characteristic. 

Consider a kernel~$\kappa\in\Lloc$ and~$\phi\in\CF[\PL,\cpct]$. The transform~$\transform$ can be efficiently computed as follows. One can write~$\phi = \sum_{i\in I}m_i\cdot \1_{P_i}$ where the set~$I$ is finite, the coefficients~$m_i$ are integers and the subsets~$P_i$ are compact polytopes. By~$\Z$-linearity of hybrid transforms and the formula for the pushforward of a closed convex subset (\Cref{lem:pushforward-convex}), we have for any~$\xi\in\dual{\V}$,
\begin{equation}
    \label{eq:transform-of-PL}
    \transform<\xi> = \sum_{i\in I}m_i\cdot \int_{-h_{P_i}(-\xi)}^{h_{P_i}(\xi)}\kappa(t)\d t.
\end{equation}
Hence, if one can express on the one hand~$h_{P}(\xi)$ as an explicit function of~$\xi\in\dual{\V}$ and of the vertices of the compact polytope~$P$, and on the other hand the integral $\int_a^b \kappa(t)\d t$ as an explicit function of the real numbers~$a<b$, then one can compute the right-hand side of~\eqref{eq:transform-of-PL}. Each example in this paper is computed following the above methodology, computing explicit closed formulae by hand for~\eqref{eq:transform-of-PL} and plotting them using Mathematica \cite{Mathematica}.

\begin{rk}
    \label{rk:computations}
    Complex valued functions $g:\R^2 \to \C$ are plotted using the function \texttt{ComplexPlot} of the Wolfram language. For $x\in\R^2$, the argument of $g(x)$ is plotted using a fixed color function from $-\pi$ to $\pi$ and the absolute value of $g(x)$ is represented as a level of brightness of this color.
\end{rk}

\paragraph{Software.}
A software that automatically computes hybrid transforms of constructible functions defined on embedded cubical complexes is available on GitHub: \href{https://github.com/HugoPasse/Transforms-of-cubical-complexes}{\texttt{https://git} \texttt{hub.com/HugoPasse/Transforms-of-cubical-complexes}}. It is running in Python and C++. This is joint work with Steve Oudot and Hugo Passe.


%
\section{Regularity}\label{sec:regularity}
In this section, we consider a kernel~$\kappa\in\Lloc$ and study the regularity of hybrid transforms in the particular case of~$\PL$-constructible functions. While being less general, this class of functions is of prime interest in applications. 


\begin{minipage}{\textwidth}
    \begin{prop}[Continuity]
        \label{prop:continuity}
        Let~$\phi\in \CF[\PL]$.
        \begin{enumerate}
            \item\label{itm:continuity-compact} If~$\supp(\phi)$ is compact, then~$\transform$ is continuous on~$\dual{\V}$.
            \item\label{itm:continuity-non-compact} If~$\supp(\phi)\subseteq K + C$ with~$K$ convex compact and~$C \ne\{0\}$ a non-empty closed convex cone, then 
            \begin{enumerate}[label = \textnormal{(\alph*)}]
                \item~$\transform$ is continuous on~$\Int(\dualcone{C})$ when~$\kappa\in\Lp[\R_{\geq 0}]$,
                \item~$\transform$ is continuous on~$\intpolant{C}$ when~$\kappa\in\Lp[\R_{\leq 0}]$.
            \end{enumerate}
        \end{enumerate} 
    \end{prop}
\end{minipage}

\medskip

\begin{proof}
    It is sufficient to prove the result for~$\phi = \1_P$ where~$P$ is a closed convex polyhedron included in~$\supp(\phi)$ since any~$\phi \in \CF[\PL]$ is a finite~$\Z$-linear combination of such functions. By \Cref{lem:pushforward-convex}, for any~$\xi\in\dual{\V}$ proper on~$P$, we have that~$\xi_*\1_{P} = \1_{[-h_P(-\xi),\,h_P(\xi)]}$, and thus
    \begin{equation}
        \label{eq:expression-transform-indic-polyhedron}
        \transform[\1_P]<\xi> = \int_{-h_P(-\xi)}^{h_P(\xi)} \kappa(t) \d t,
    \end{equation}
    so that it is sufficient to study the continuity of~$h_P$. 
    
    As any convex function,~$h_P$ is continuous on~$\Int(\dom(h_P))$, where~$\dom(h_P) = \{\eta \in \dual{\V} \st h_P(\eta) < +\infty\}$ by \cite[Thm.~1.5.4]{Schnei14}. If~$P$ is compact, then~$\dom(h_P) = \dual{\V}$, hence~\ref{itm:continuity-compact}. 
    Now, suppose that~$P$ is not compact and assume~$\kappa\in\Lp[\R_{\geq 0}]$, the other case being symmetric. Using classical polyhedra theory, see \cite[Thm.~1.2]{Z12}, one can write~$P$ as a sum~$K' + C'$ where~$K'$ is a convex polytope and~$C'$ a closed convex polyhedral cone. Since we assumed~$P\subseteq \supp(\phi)$, we also have~$C'\subseteq C$. Then, for any~$\xi\in \Int(\dualcone{C})$, one has~$\xi\in \Int(\dualcone{C'})$, so that~$h_P(\xi) = +\infty$, and~$h_P(-\xi) = h_{K'}(-\xi)$. Since~$K'$ is compact,~$h_{K'}$ is continuous on~$\dual{\V}$, and so is~$h_P\circ a$ on~$\Int(\dualcone{C})$ with~$a(\xi)=-\xi$ the antipodal map. The result follows then from~\eqref{eq:expression-transform-indic-polyhedron}.
\end{proof}


\begin{prop}\label{prop:regularity}
    Let~$\phi\in \CF[\PL]$ and~$p\in\N$. Assume that~$\kappa$ is~$\Ck[p]$. There exists a finite family $\{\Gamma_1, \dots, \Gamma_m\}$ of open convex polyhedral cones of $\dual{\V}$ such that $\dual{\V} = \bigcup_{i=1}^m \closure{\Gamma_i}$ and:
    \begin{enumerate}
        \item\label{itm:smooth-compact} If~$\supp(\phi)$ is compact, then the restriction~$\transform_{|\Gamma_i}$ is~$\Ck[p+1]$ for all $i=1,...,m$.
        \item\label{itm:smooth-non-compact} If~$\supp(\phi)\subseteq K + C$ with~$K$ convex compact and~$C \ne\{0\}$ a non-empty closed convex cone, then 
        \begin{enumerate}[label = \textnormal{(\alph*)}]
            \item the restriction~$\transform_{|\Gamma_i\cap \Int(\dualcone{C})}$ is~$\Ck[p+1]$ for all $i=1,...,m$ when~$\kappa\in\Lp[\R_{\geq 0}]$,
            \item the restriction~$\transform_{|\Gamma_i\cap \intpolant{C}}$ is~$\Ck[p+1]$ for all $i=1,...,m$ when~$\kappa\in\Lp[\R_{\leq 0}]$.
        \end{enumerate}
    \end{enumerate}
        
\end{prop}
    
\begin{proof}
    The result~\ref{itm:smooth-non-compact} follows from~\ref{itm:smooth-compact} similarly as in the proof of \Cref{prop:continuity}. Moreover, result~\ref{itm:smooth-compact} follows from~\eqref{eq:expression-transform-indic-polyhedron} and the fact that if~$P$ is a polytope, the support function~$h_P$ is smooth outside the closed set~$E$ of linear forms which are orthogonal to at least one face of dimension~$1$ of~$P$. Indeed, on each connected component of $\dual{\V}\setminus E$, there exists a vertex~$v$ of~$P$ such that~$h_P(\xi) = \dualdot{\xi}{v}$. It is easy to check that each connected component $\Gamma_i$ of~$\dual{\V}\setminus E$ is an open convex polyhedral cone. Moreover, the set $E$ being a finite union of subspaces of $\dual{\V}$ of codimension at least $1$, it is closed and has empty interior. Hence, $\dual{\V} = \closure{\dual{\V}\setminus E} = \bigcup_{i=1}^m \closure{\Gamma_i}$.
\end{proof}

\begin{ex}
    For the Euler-Laplace and Euler-Fourier transforms, the previous result is well illustrated by \Cref{fig:EF-spiral,fig:EL-sqcrck,fig:EF-gT,fig:EF-sqcrck}.
\end{ex}

\section{Compatibility with operations}\label{sec:compatibility}
In this section, we consider~$\kappa\in\Lloc$ and investigate the compatibility of hybrid transforms with operations on constructible functions. The results use the general form of hybrid transforms defined in \Cref{rk:generalization-ht}.

\subsection{Direct image, duality and projection}
\begin{prop}[Direct image]
    \label{prop:comp-direct-image}
    Let~$\phi\in\CF$, and let~$f:\V \to \V'$ and~$\zeta : \V'\to\R$ be morphisms of real analytic manifolds. Assume that~$\zeta\circ f$ is proper on~$\supp(\phi)$ and that~$\kappa\cdot\zeta_*f_*\phi\in\Lp$. Then,
    \begin{equation*}
        \transform[f_*\phi]<\zeta> = \transform[\phi]<f^*\zeta>.
    \end{equation*}
    In particular, if~$f$ and~$\zeta$ are linear maps, denoting~$\transpose f : \dual{\V'}\to\dual{\V}$ the dual map, we get:
    \begin{equation*}
        \transform[f_*\phi]<\zeta> = \transform<\transpose \hspace{-.1em}f(\zeta)>.
    \end{equation*}
\end{prop}

\begin{proof}
    By functoriality of the pushforward, we have~$\zeta_*f_* = \left(\zeta\circ f\right)_* = (f^*\zeta)_*~$.
\end{proof}


\begin{ex}
    Let~$x_0\in \V$ and consider the map~$\tau_{x_0}:\V \to \V$ given by~$x\mapsto x+x_0$. For~$\phi\in\CF[\cpct]$, we have~$\tau_{x_0\,*}\phi(x) = \phi(x- x_0)$ for any~$x\in \V$. Moreover, \Cref{prop:comp-direct-image} yields for any~$\xi\in\dual{\V}$,    
    \begin{align*}
        \ELaplace[\tau_{x_0\,*}\phi]<\xi> &= e^{-\dualdot{\xi}{x_0}}\,\cdot \ELaplace[\phi]<\xi>, \\
        \EFourier[\tau_{x_0\,*}\phi]<\xi> &= e^{-i\dualdot{\xi}{x_0}}\cdot \EFourier[\phi]<\xi>.
    \end{align*}
\end{ex}

\begin{prop}[Duality]
    \label{prop:comp-duality}
    Let~$\phi\in\CF$ and~$\zeta : \V\to\R$ be a morphism of real analytic manifolds. Assume that~$\zeta$ is proper on~$\supp(\phi)$ and that~$\kappa\cdot\zeta_*\phi\in\Lp$. 
    Then,~$\kappa\cdot\zeta_*(\D_\V\phi)\in\Lp$, and:
    \begin{equation*}
        \transform[\D_\V\phi]<\zeta> = -\transform<\zeta>.
    \end{equation*}
\end{prop}
%
\begin{proof}
    By \cite[Thm.~2.5~(iii)]{S91}, we have that~$\zeta_*\left(\D_\V\phi\right)=\D_\R(\zeta_*\phi)$, so it suffices to show that:
    \begin{equation}
        \label{eq:equality-integrals-duality}
        \int_{\R}\kappa(t)\D_\R(\zeta_*\phi)(t)\d t = -\int_{\R}\kappa(t)\zeta_*\phi(t)\d t,
    \end{equation}
    provided that the integrals make sense. Yet, a direct computation yields~$\D_\R\1_{[a,b]} = - \D_\R\1_{(a,b)}$ for any two real numbers~$a < b$ and~$\zeta_*\phi$ is equal to a finite~$\Z$-linear combination of such functions outside a discrete set of points as any constructible function on~$\R$. Thus, 
    \begin{equation*}
        \D_\R(\zeta_*\phi) = - \zeta_*\phi,
    \end{equation*}
    outside a set of Lebesgue measure zero. Since by assumption~$\kappa\cdot\zeta_*\phi\in\Lp$, this provides both the integrability of the integrands of \eqref{eq:equality-integrals-duality} and the equality between the integrals involved.
\end{proof}

\begin{prop}[Projection formula]
    \label{prop:comp-projection-formula}
    Let~$\phi\in\CF$,~$\theta\in\CF[][\R]$ and let~$\zeta : \V\to\R$ be a morphism of real analytic manifolds. Assume that~$\zeta$ is proper on~$\supp(\phi)$, that~$\kappa\cdot\zeta_*\phi\in\Lp$ and that~$\kappa\cdot\theta\cdot\zeta_*\phi\in\Lp$. Then,
    \begin{equation*}
        \transform[\phi\cdot\zeta^*\theta]<\zeta> = \transform[\phi]<\zeta>[\kappa\cdot\theta].
    \end{equation*}
\end{prop}
\begin{proof}
    The result follows from the formula for constructible functions~$\zeta_*\left(\phi\cdot\zeta^*\theta\right) = \theta\cdot\zeta_*\phi$, which follows from the corresponding property for constructible sheaves \cite[Prop.~2.6.6]{KS90} and the function-sheaf correspondence \cite[Thm.~9.7.1]{KS90}.
\end{proof}

\subsection{Convolution and box product}\label{sec:comp-conv-and-box}

\begin{prop}[Convolution for~$\EL$]
    \label{prop:comp-conv-EL}
    Let~$\phi$ and $\psi$ be two constructible functions on~$\V$ satisfying \Cref{ass:EL-well-def} for a cone~$C$. Then, we have on~$\Int(\dualcone{C})$,
    \begin{equation*}
        \ELaplace[\phi\conv \psi] = \ELaplace[\phi\conv\1_{C}]\cdot\ELaplace[\psi] + \ELaplace[\phi]\cdot\ELaplace[\psi\conv\1_{C}] - \ELaplace[\phi]\cdot\ELaplace[\psi].
    \end{equation*}
\end{prop}

\begin{proof}
    By \Cref{cor:pushforward-linear-conv}, one has~$\xi_*\left(\phi\conv \psi\right) = (\xi_*\phi)\conv(\xi_*\psi)$. Thus, using \Cref{lem:link-ELaplace-Laplace}, we have:
    \begin{equation*}
        \ELaplace[\phi\conv \psi]<\xi> = \ELaplace[(\xi_*\phi)\conv(\xi_*\psi)]<1>.
    \end{equation*}
    The result follows then from the following claim, proven afterwards.
    \begin{claim}
        \label{claim:conv-EL-real-CF}
        If~$\theta,\theta' \in \CF[][\R]$ both satisfy \Cref{ass:EL-well-def} for the cone~$\R_{\geq 0}$, then
        
        \begin{equation*}
            \begin{split}
                \ELaplace[\theta\conv \theta']<1> &= \ELaplace[\theta\conv\1_{\R_{\geq 0}}]<1>\cdot\ELaplace[\theta']<1> \\[0.2em]
                &\qquad\quad + \ELaplace[\theta]<1>\cdot\ELaplace[\theta'\conv\1_{\R_{\geq 0}}]<1> \\[0.2em]
                &\qquad\qquad\quad- \ELaplace[\theta]<1>\cdot\ELaplace[\theta']<1>.
            \end{split}
        \end{equation*}
    \end{claim}
    Indeed, the functions~$\xi_*\phi$ and~$\xi_*\psi$ both satisfy \Cref{ass:EL-well-def} for the cone~$\R_{\geq 0}$ by \Cref{lem:xi-proper-supp-phi}. Thus, \Cref{claim:conv-EL-real-CF} yields:
    \begin{equation*}
        \begin{split}
            \ELaplace[(\xi_*\phi)\conv(\xi_*\psi)]<1> 
            &= \ELaplace[(\xi_*\phi)\conv\1_{\R_{\geq 0}}]<1>\cdot\ELaplace[\xi_*\psi]<1> \\
            &\qquad\quad + \ELaplace[\xi_*\phi]<1>\cdot\ELaplace[(\xi_*\psi)\conv\1_{\R_{\geq 0}}]<1> \\
            &\qquad\qquad\quad - \ELaplace[\xi_*\phi]<1>\cdot\ELaplace[\xi_*\psi]<1>.
        \end{split}
    \end{equation*} 
    Hence the result, by \Cref{lem:link-ELaplace-Laplace}, using \Cref{cor:pushforward-linear-conv} and \Cref{lem:pushforward-convex} to get:    
    \begin{align*}
        (\xi_*\phi)\conv\1_{\R_{\geq 0}} &= \xi_*\left(\phi\conv\1_{C}\right),\\
        (\xi_*\psi)\conv\1_{\R_{\geq 0}} &= \xi_*\left(\psi\conv\1_{C}\right).
    \end{align*}
    
    \bigskip

    Let us now prove \Cref{claim:conv-EL-real-CF}.
    \Cref{lem:support-pushforward} yields that~$\theta\conv\theta'$ satisfies \Cref{ass:EL-well-def} for the cone~$\R_{\geq 0}$, so that~$\ELaplace[\left(\theta\conv \theta'\right)]$ is well-defined on~$\R_{>0} \ni 1$. Consider now decompositions of~$\theta$ and~$\theta'$:        
        \begin{align*}
            \theta &= \sum_{i\in I} m_i \1_{[a_i,b_i]},\\
            \theta' &= \sum_{j\in J} n_j \1_{[c_j,d_j]},
        \end{align*}
        where~$I$ and~$J$ are finite,~$m_i$ and~$n_j$ are integers,~$a_i$,~$c_j$ are real numbers, and~$b_i, d_j \in\R\cup \{+\infty\}$. We abusively denoted~$[x,+\infty] := [x,+\infty)$ for~$x\in\R$ for simplicity. 
        One has:
        
        \begin{equation*}
            \theta\conv\theta' 
            = \sum_{(i,j)\in I\times J} m_i n_j \1_{[a_i,b_i]}\conv\1_{[c_j,d_j]}
            = \sum_{(i,j)\in I\times J} m_i n_j \1_{[a_i+c_j,b_i+d_j]}.
        \end{equation*}
        Therefore, we have:        
        \begin{align*}
            \ELaplace[\theta\conv \theta']<1> 
            &= \int_\R e^{-t}\left(\theta\conv \theta'\right)(t) \d t \\
            &= \sum_{(i,j)\in I\times J} m_i n_j \int_{a_i+c_j}^{b_i+d_j} e^{-t} \d t \\
            &= \sum_{(i,j)\in I\times J} m_i n_j \left(e^{-a_i-c_j} - e^{-b_i-d_j}\right) \\
            &= AC - BD \\
            &= A(C-D) + (A-B)C - (A-B)(C-D),
        \end{align*}
        using the convention that~$e^{-\infty}=0$, and where we denoted:        
        \begin{align*}
            A := \sum_{i\in I} m_i e^{-a_i}, &\quad& B := \sum_{i\in I} m_i e^{-b_i}, \\[1em]
            C := \sum_{j\in J} n_j e^{-c_j}, &\quad& D := \sum_{j\in J} n_j e^{-d_j}.
        \end{align*}
        Moreover, for~$x\in\R$ and~$y\in \R\cup \{+\infty\}$, we have: 
        \begin{equation}
            \label{eq:conv-closed-intervals}
            \1_{[x,y]}\conv\1_{\R_{\geq 0}} = \1_{[x,+\infty)},
        \end{equation}
        so that:        
        \begin{align*}
            \theta\conv\1_{\R_{\geq 0}} &= \sum_{i\in I} m_i \1_{[a_i,+\infty)}, \\
            \theta'\conv\1_{\R_{\geq 0}} &= \sum_{j\in I} n_j \1_{[c_j,+\infty)}.
        \end{align*}
        Thus, \Cref{ex:EL-interval} yields:        
        \begin{align*}
            A &= \ELaplace[\theta\conv\1_{\R_{\geq 0}}]<1>, & \hspace{7em } C &= \ELaplace[\theta'\conv\1_{\R_{\geq 0}}]<1>, \\
            A-B &= \ELaplace[\theta]<1>, &  C - D &= \ELaplace[\theta']<1>,
        \end{align*}
        which proves \Cref{claim:conv-EL-real-CF}, and finishes the proof of \Cref{prop:comp-conv-EL}.
\end{proof}

\begin{rk}
    \label{rk:conv-two-cones-to-one}
    If~$\phi$ and~$\psi$ are constructible functions satisfying \Cref{ass:EL-well-def} for two different cones~$C'$ and~$C''$ respectively, then they satisfy \Cref{ass:EL-well-def} for the cone~$C = \Conv{C' \cup C''}$.
\end{rk}

\begin{cor}
    \label{cor:comp-conv-EL-stable-cone}
    Let~$\phi$ and~$\psi$ be two constructible functions on $\V$ satisfying \Cref{ass:EL-well-def} for cones~$C'$ and~$C''$ respectively. Assume in addition that~$\phi = \phi \conv \1_{C'}$ and $\psi = \psi \conv \1_{C''}$. 
    Then, we have on~$\Int(\dualcone{C'})\cap\Int(\dualcone{C''})$,
    \begin{equation*}
        \ELaplace[\phi\conv \psi] = \ELaplace[\phi]\cdot\ELaplace[\psi].
    \end{equation*}
\end{cor}
\begin{proof}
    Following \Cref{rk:conv-two-cones-to-one}, consider $C = \Conv{C'\cup C''}$. Then $\Int(\dualcone{C}) = \Int(\dualcone{C'})\cap\Int(\dualcone{C''})$ and the result follows from \Cref{prop:comp-conv-EL} and the fact that $\1_{\Gamma}\conv \1_{\Gamma'} = \1_{\Gamma}$, for any two closed convex proper cones~$\Gamma$ and $\Gamma'$ such that $\Gamma' \subseteq \Gamma$.
\end{proof}

\begin{rk}
    If $\phi$ satisfies \Cref{ass:EL-well-def} for a cone~$C'$, then it is constructible up to infinity and its support is $\gamma$-proper for the cone $\gamma = \antipodal{C'}$. Hence, the assumption that $\phi = \phi \conv \1_C$ is equivalent to that of $\phi$ being $\gamma$-constructible by \Cref{prop:charac-gamma-CF}.
\end{rk}
    

Any~$\eta\in\dual{(\V\times\V')}$ can naturally be written~$\eta = s\circ(\xi \times \xi')$ for~$(\xi,\xi')\in\dual{\V}\times\dual{\V'}$ and~$s : \R \times \R \to \R$ the addition, so that we have the following corollary:
\begin{cor}[Box product for~$\EL$]
    \label{cor:comp-box-EL}
    Let~$\phi\in\CF$ and~$\psi\in\CF[][\V']$ both satisfy \Cref{ass:EL-well-def} for cones~$C\subseteq \V$ and~$C'\subseteq\V'$ respectively. For any~$(\xi,\xi')\in\Int(\dualcone{C})\times\Int(\dualcone{C'})$, we have:
    \begin{align*}
        \ELaplace[\phi\boxtimes \psi]<\eta> &= \ELaplace[\phi\conv\1_{C}]<\xi>\cdot\ELaplace[\psi]<\xi'> + \ELaplace[\phi]<\xi>\cdot\ELaplace[\psi\conv\1_{C'}]<\xi'>\\[0.5em]
        &\hspace{18em} - \ELaplace[\phi]<\xi>\cdot\ELaplace[\psi]<\xi'>,
    \end{align*}
    with~$\eta = s\circ (\xi \times \xi')$.
\end{cor}
\begin{proof}
    Since~$\eta = s \circ (\xi\times \xi')$, \Cref{lem:int-boxtimes-pushforward} implies that 
        $\eta_*\left(\phi\boxtimes \psi\right) = (\xi_*\phi)\conv(\xi'_*\psi).$
     Hence, 
    
     \begin{equation*}
        \ELaplace[\phi\boxtimes \psi]<\eta>
        = \ELaplace[\eta_*\left(\phi\boxtimes \psi\right)]<1> = \ELaplace[(\xi_*\phi)\conv(\xi'_*\psi)]<1>,
    \end{equation*}
    and the result follows from \Cref{lem:link-ELaplace-Laplace} and the compatibility with convolution (\Cref{prop:comp-conv-EL}).
\end{proof}
\begin{cor}
    In the setting of \Cref{cor:comp-conv-EL-stable-cone}, we have for any~$(\xi,\xi')\in\Int(\dualcone{C})\times\Int(\dualcone{C'})$,
    \begin{equation*}
        \ELaplace[\phi\boxtimes \psi]<\eta> = \ELaplace[\phi]<\xi>\cdot\ELaplace[\psi]<\xi'>,
    \end{equation*}
    with~$\eta = s\circ (\xi \times \xi')$.
\end{cor}

\begin{ex}[Interpretation of the Laplace transform on~$\gamma$-voxels]
    \label{ex:EL-gamma-voxels}
    We call \emph{$\gamma$-voxel} a subset of~$\R^d$ of the form~$[a_1,b_1)\times\dots\times[a_d,b_d)$ where~$a_i < b_i$ are real numbers. Consider~$\phi = \sum_{i\in I} m_i \1_{V_i}\in \CF[][\R^d]$ where the set~$I$ is finite, the coefficients~$m_i$ are integers and the subsets~$V_i$ are~$\gamma$-voxels. Then, for any~$\xi=(\xi_1,\dots,\xi_d)\in(\R_{\geq 0})^d$,
    \begin{equation}
        \label{eq:EL-L-gamma-voxels}
        \ELaplace[\phi]<\xi> = \Laplace[\phi](\xi)\cdot \prod_{k=1}^d\xi_k.
    \end{equation} 
    Indeed, the equality is true for~$d=1$ and extends naturally to~$\gamma$-voxels thanks to the compatibility formula for the box product (\Cref{cor:comp-box-EL}). This relation gives a new interpretation of the Laplace transform on such constructible functions. One could wonder whether the relation~\eqref{eq:EL-L-gamma-voxels} can be generalized for all~$\gamma$-constructible functions. However, \Cref{ex:EL-gamma-triangle} shows that such a generalization is not obvious.
\end{ex}

\medskip 

Let~$\gamma$ be a cone of~$\V$ satisfying~\eqref{hyp:cone}.

\begin{prop}[Convolution for~$\EF$]
    \label{prop:comp-conv-EF}
    Let~$\phi,\psi\in\CF[\cpct,\g]$. For~$\xi\in \dual{\V}$,
    \begin{equation*}
        \EFourier[\phi\conv \psi]<\xi> = 
        \begin{cases}
            \ \ i \cdot\EFourier[\phi]<\xi>\cdot\EFourier[\psi]<\xi> &\mbox{if }\xi\in \dualcone{\antipodal{\gamma}},\\
            - i \cdot\EFourier[\phi]<\xi>\cdot\EFourier[\psi]<\xi> &\mbox{if }\xi\in \dualcone{\gamma}.
        \end{cases}
    \end{equation*}
\end{prop}
\begin{rk}
    For any~$\xi\in\dual{\V}\setminus(\dualcone{\gamma}\cup \dualcone{\antipodal{\gamma}})$, the compatibility formula still holds. Indeed, since~$\phi\conv\psi$ is~$\gamma$-constructible by \Cref{prop:charac-gamma-CF}, both sides of the equality are zero by \Cref{prop:supp-Radon-g} and~\eqref{eq:pushforward-as-Radon}. 
\end{rk}
\begin{proof}
    Suppose that~$\xi\in\dualcone{\antipodal{\gamma}}$, the other case being similar. We have:
    \begin{equation*}
        \EFourier[\phi\conv\psi]<\xi> = \EFourier[\xi_*\left(\phi\conv\psi\right)]<1>,
    \end{equation*}
    and~$\xi_*\left(\phi\conv \psi\right) = (\xi_*\phi)\conv(\xi_*\psi)$ by \Cref{cor:pushforward-linear-conv}. Since \Cref{lem:pushforward-g-CF-polant} ensures that~$\theta=\xi_*\phi$ and~$\theta'=\xi_*\psi$ are both in~$\CF[\cpct,\g[\lambda]]$ with~$\lambda = \R_{\leq 0}$, it is sufficient to prove:
    \begin{equation*}
        \EFourier[\theta\conv \theta']<1> = i\cdot \EFourier[\theta]<1>\cdot\EFourier[\theta']<1>,
    \end{equation*}
    for any~$\theta,\theta'\in\CF[\cpct,\g[\lambda]]$. By bilinearity of the convolution, it is even sufficient to prove the result for~$\theta = \1_{[a,b)}$ and~$\theta' = \1_{[c,d)}$ where~$a<b$ and~$c<d$ are real numbers. Suppose now that~$a+d \leq b+c$, the case~$a+d \geq b+c$ being proven in a similar fashion. Then, 
    \begin{equation*}
        \1_{[a,b)} \conv \1_{[c,d)} = \1_{[a+c,a+d)} - \1_{[b+c,b+d)}.
    \end{equation*}
    Therefore, we get:    
    \begin{align*}
        \EFourier[\1_{[a,b)}\conv \1_{[c,d)}]<1> 
        &= i\left(e^{-i(a+d)}- e^{-i(a+c)}\right) - i\left(e^{-i(b+d)}- e^{-i(b+c)}\right)\\
        &= -i\left(e^{-ib}- e^{-ia}\right) \left(e^{-id} - e^{-ic}\right)\\
        &= i \ \EFourier[\1_{[a,b)}]<1> \EFourier[\1_{[c,d)}]<1>.
    \end{align*}
\end{proof}

As for the Euler-Laplace transform, one gets the following corollary for the box product. Let us consider a cone~$\gamma'$ of~$\V'$ satisfying~\eqref{hyp:cone}.
\begin{cor}[Box product for~$\EF$]
    \label{cor:comp-box-EF}
    Let~$\phi\in\CF[\cpct,\g]$ and~$\psi\in\CF[\cpct,\g[\gamma']][\V']$. For any~$(\xi,\xi')\in\dual{\V}\times\dual{\V'}$, we have:
    \begin{equation*}
        \EFourier[\phi\boxtimes \psi]<\eta> = 
        \begin{cases}
            \ \ i \cdot\EFourier[\phi]<\xi>\cdot\EFourier[\psi]<\xi'> &\mbox{if }(\xi,\xi')\in \dualcone{\antipodal{\gamma}}\times \dualcone{\antipodal{\gamma'}},\\
            - i \cdot\EFourier[\phi]<\xi>\cdot\EFourier[\psi]<\xi'> &\mbox{if }(\xi,\xi')\in \dualcone{\gamma}\times \dualcone{\gamma'},
        \end{cases}
    \end{equation*}
    with~$\eta = s\circ (\xi \times \xi')$. 
\end{cor}
\begin{rk}
    Note that not all possibilities of~$(\xi,\xi')$ are treated in the previous corollary, as such an equality is not true in general.
\end{rk}
\begin{ex}[Interpretation of the Fourier transform on~$\gamma$-voxels]
    In the setting of \Cref{ex:EL-gamma-voxels}, we have for any~$\xi=(\xi_1,\dots,\xi_d)\in(\R_{\geq 0})^d$,
    \begin{equation}
        \label{eq:EF-F-gamma-voxels}
        \EFourier[\phi]<\xi> = i^{d-1}\cdot \Fourier[\phi](\xi)\cdot \prod_{k=1}^d\xi_k.
    \end{equation} 
    Again, this relation gives a new interpretation of the Fourier transform on such constructible functions. \Cref{ex:EF-gamma-triangle} shows that a generalization of such a relation for all~$\gamma$-constructible functions is not obvious.
\end{ex}

\begin{rk}[On stability]
    For each integer~$k\geq 1$, consider:
    \begin{equation*}
        \phi_k = \sum_{(i,j)} \1_{[\frac{i}{k},\frac{i+1}{k})\times[\frac{j}{k},\frac{j+1}{k})},
    \end{equation*}
    where the sum is over all pairs~$(i,j)\in\integint{k-1}^2$ such that~$i+j \leq k-1$. Then, the sequence~$(\phi_k)_{k\geq 1}$ converges to the~$\gamma$-triangle~$\1_T$ of \Cref{ex:EL-gamma-triangle} in~$\Lp[\R^2][p]$ for~$p\in[1,+\infty]$. Thus, over the domain~$\left(\R_{\geq 0}\right)^2$, the sequence~$\left(\EFourier[\phi_k]\right)_{k\geq 1} = \left(\Fourier[\phi_k]\right)_{k\geq 1}$ converges to~$\Fourier[\1_T]$ in~$\mathrm{L}^\infty$. However, this last function differs from~$\EFourier[\1_T]$, as shown in \Cref{ex:EF-gamma-triangle}. Hence, if a stability statement holds for the Euler-Fourier transform, it should be for other norms on constructible functions and on the Euler-Fourier transforms. The stability of hybrid transforms goes beyond the scope of this article and will be the object of future work.
\end{rk}

\section{Reconstruction formula for~$\EF$}\label{sec:reconstruction-EFourier}
In this section, we establish a reconstruction formula (\Cref{thm:reconstruction-EFourier}) for the Euler-Fourier transform of~$\gamma$-constructible functions. We state our results in \Cref{sec:reconstruct-results} and postpone the proof of two propositions to \Cref{sec:supp-Radon-gamma} and \Cref{sec:Radon-from-EFourier}.

\subsection{Results}\label{sec:reconstruct-results}
Consider a cone~$\gamma$ of~$\V$ satisfying~\eqref{hyp:cone}.

\paragraph{Reconstruction of~$\xi_*\phi$ from~$\EFourier$.}

If~$h:\dual{\V}\to \R$ is such that for any~$\xi\in\dual{\V}\setminus\{0\}$, the map:
\begin{equation*}
    \widetilde{h}_\xi : s \longmapsto \frac{\1_{\R\setminus\{0\}}(s)}{|s|}\cdot h(s\,\xi),
\end{equation*}
satisfies that the following limit exists:
\begin{equation*}
    \Fourier^{-1}\left[\widetilde{h}_\xi\right] (t) := \limit_{A\to+\infty}\int_{-A}^A e^{ist}\cdot \widetilde{h}_\xi(s) \d s,
\end{equation*}
then we can define the following map for all~$\xi\in\dual{\V}$ and~$t\in \R$,
\begin{equation}
    \Fourier'(h)(\xi,t) :=
    \begin{cases}
        \displaystyle\frac{1}{2\pi}\Fourier^{-1}\left[\widetilde{h}_\xi\right](t^+) &\mbox{ if } \xi \in\dualcone{\antipodal{\gamma}}\setminus\{0\}, \\[1em]
        \displaystyle\frac{1}{2\pi}\Fourier^{-1}\left[\widetilde{h}_\xi\right](t^-) &\mbox{ if } \xi \in\dualcone{\gamma}\setminus\{0\}, \\[1em]
        0 &\mbox{ else,}
    \end{cases}
\end{equation}
where~$g (t^\pm) = \limit_{s\to t^\pm} g (s)$ for any function~$g$ defined in a neighborhood of~$t$.

\begin{prop}
    \label{prop:Radon-fromEFourier}
    Let~$\phi\in\CF[\cpct,\g]$. Then,~$\Fourier'\left(\EFourier\right)$ is well-defined, and for all~$\xi\in\dual{\V}$ and~$t\in \R$, 
    \begin{equation*}
        \Fourier' \left(\EFourier\right)(\xi,t) = \xi_*\phi(t).
    \end{equation*}
\end{prop}
The above proposition is proven in \Cref{sec:Radon-from-EFourier}. For~$\xi\in \dualcone{\gamma}\cup \dualcone{\antipodal{\gamma}}\setminus\{0\}$, the proof boils down to inverting the classical Fourier transform. For other~$\xi$'s, we show that the pushforward~$\xi_*\phi$ is zero, hence so are both sides of the equality. To do so, we use that~$\xi_*\phi(t)$ is an evaluation of the Radon transform of~$\phi$ and prove an explicit description of the support of this transform, as explained in the next paragraph (\Cref{prop:supp-Radon-g}).

\paragraph{Support of the Radon transform.}
For the sake of readability, we use the notations~$\projective = \projective[\V]$ and~$\dprojective = \dprojective[\V]$ for the projective compactifications of~$\V$ and~$\dual{\V}$ respectively. Let us denote by $S$ the incidence relation for projective duality:
\begin{equation*}
    S = \Big\{ \big([v:\lambda],[\xi:t]\big)\in\projective\times\dual{\projective} \st \dualdot{\xi}{v} + \lambda t = 0 \Big\},
\end{equation*}
and denote by~$p : S\to \projective$ and~$q : S\to \dual{\projective}$ the restrictions of the canonical projections.
\begin{equation*}
    \begin{tikzcd}
        & \projective\times \dual{\projective} &                    \\
        & S \arrow[ld, "p"'] \arrow[rd, "q"] \arrow[u, symbol = \subseteq]              &                    \\
    \projective &                                      & \dual{\projective}
    \end{tikzcd}
\end{equation*}
The \emph{Radon transform} of~$\phi\in\CF[][\projective]$ is the constructible function~$\Radon(\phi)\in\CF[][\dual{\projective}]$ defined in \cite{S95} as:
\begin{equation*}
    \Radon(\phi) = q_* p^* \phi.
\end{equation*}
Since~$S$ is a compact subset of~$\projective\times\dual{\projective}$, the map~$q$ is proper and~$\Radon$ is well-defined.
Any~$\phi\in\CF[\cpct]$ naturally yields a constructible function~$j_*\phi$ on~$\projective$, as explained in \Cref{ex:cpctly-supped-CF-infty}. In that case, Sec.~5 in loc. cit. ensures that for any~$y = [\xi:t]\in\dual{\projective}\setminus\{h_\infty\}$,
\begin{equation}\label{eq:pushforward-as-Radon}
    \Radon(j_*\phi)(y) = \int_\V \phi \cdot \1_{\xi^{-1}(t)} \d\Euler = \xi_*\phi(t),
\end{equation}
and that~$\supp(\Radon(j_*\phi))\subseteq \dual{K}$, where
\begin{equation*}
    \dual{K} := q\left(p^{-1}\left(\supp(\phi)\right)\right) = \left\{ [\xi:t]\in\dual{\projective} \st \xi^{-1}(t)\cap\supp(\phi) \ne \emptyset \right\}
\end{equation*} 
is a compact subset of~$\dual{\projective}\setminus\{h_\infty\}$. The next proposition refines this inclusion when~$\phi$ is~$\gamma$-constructible, hence allowing us to prove \Cref{prop:Radon-fromEFourier}. Its proof is in \Cref{sec:supp-Radon-gamma}. 
\begin{prop}
    \label{prop:supp-Radon-g}
    Let~$\phi\in\CF[\cpct,\g]$. Then,~$\supp\left(\Radon(j_*\phi)\right)\subseteq \dual{K}_\gamma$, where
    \begin{equation*}
        \dual{K}_\gamma := \big\{ [\xi:t] \in \dual{K} \st \xi\in \dualcone{\gamma}\cup \dualcone{\antipodal{\gamma}} \big\}
    \end{equation*}
    is a compact subset of~$\dual{\projective}\setminus\{h_\infty\}$.
\end{prop}

\paragraph{Reconstruction formula.}
For any~$\phi\in\CF[\cpct,\g]$, \Cref{prop:Radon-fromEFourier} implies that the map~$\Fourier'\left(\EFourier\right)$ induces a map on~$\dual{\projective}$. Combining this proposition with~\eqref{eq:pushforward-as-Radon}, we get:
    \begin{equation}\label{eq:Radon-from-EFourier}
        \Fourier'\left(\EFourier\right) = \Radon(j_*\phi).
    \end{equation}
To obtain a reconstruction formula, we are then left with inverting the Radon transform using Schapira's formula \cite[Cor.~5.1]{S95}. For that, define for any~$\psi\in\CF[][\dual{\projective}]$,
\begin{equation*}
    \Radon'(\psi) = p_*q^*\psi.
\end{equation*}
Again, the map~$p$ is proper, so~$\Radon'$ is well-defined. It is almost a left inverse for~$\Radon$. 
\begin{thm}[\textnormal{\cite[Cor.~5.1]{S95}}]
    \label{thm:Radon-inversion}
    Let~$\phi\in\CF[][\projective]$. Then,
    \begin{equation*}
        \Radon'\circ\Radon(\phi) =
        \begin{cases}
            \phi &\mbox{ if } \dim(\V) \mbox{ is odd},\\
            -\phi + \int_{\projective}\phi\d\Euler &\mbox{ if } \dim(\V) \mbox{ is even}.
        \end{cases}
    \end{equation*}
\end{thm}
\noindent When the function~$\phi$ is~$\gamma$-constructible on~$\V$, the map~$\Radon'$ becomes a left inverse for~$\Radon$ up to a sign, by the following lemma.
\begin{lem}
    \label{lem:int-CF-c-g-null}
    If~$\phi\in\CF[\cpct,\g]$, then~$\int_\V \phi \d\Euler = 0$.
\end{lem}
\begin{proof}
    Choose an element~$\xi$ in the non-empty set~$\dualcone{\antipodal{\gamma}}\setminus\{0\}$. By \Cref{rk:int-pushforward}, we have:
    \begin{equation*}
        \int_\V \phi \d\Euler = \int_\R \xi_*\phi \d\Euler.
    \end{equation*}
    Moreover, \Cref{lem:pushforward-g-CF-polant} shows that~$\xi_*\phi\in\CF[\cpct,\g[\lambda]][\R]$ with~$\lambda = \R_{\leq 0}$. \Cref{lem:form-CF-g} ensures then that one can write~$\xi_*\phi = \sum_{i=1}^n m_i \1_{[a_i,b_i)}$ where~$m_i$ are integers and~$a_i < b_i$ are real numbers. The result follows from the fact that~$\int_\V\1_{[a_i,b_i)} \d\Euler = 0$.
\end{proof} 
Putting all together, we get the following reconstruction result for the Euler-Fourier transform.
\begin{thm}
    \label{thm:reconstruction-EFourier}
    Let~$\phi\in\CF[\cpct,\g]$. Then,
    \begin{equation*}
        \Radon'\circ\Fourier'\left(\EFourier\right) = (-1)^{\dim(\V)+1} \phi.
    \end{equation*}
\end{thm}
\begin{rk}
    We abuse notations to alleviate the formula by identifying~$j_*\phi$ and~$\phi$ in the right-hand side.
\end{rk}
\begin{proof}
    The results follows from~\eqref{eq:Radon-from-EFourier} and \Cref{thm:Radon-inversion}, since by \Cref{lem:int-CF-c-g-null},
    \begin{equation*}
        \int_\P j_*\phi \d\Euler \ \overset{\textnormal{(Rk.~\ref{rk:int-pushforward}})}{=} \ \int_\V \phi \d\Euler = 0.
    \end{equation*}
\end{proof}

\subsection{Proof of \Cref{prop:supp-Radon-g}}\label{sec:supp-Radon-gamma}
%
We prove the proposition below and prove the necessary lemmas afterwards.
\begin{proof}[Proof of \Cref{prop:supp-Radon-g}]
    The compactness of~$\dual{K}_\gamma$ is easy to prove.
    Now, it is sufficient to prove the result on the support of~$\Radon(j_*\phi)$ for~$\phi = \1_Z$ with~$Z\subseteq \V$ relatively compact, subanalytic and~$\gamma$-locally closed since any element of~$\CF[\cpct,\g]$ is a finite~$\Z$-linear combination of such functions (\Cref{lem:form-CF-g}). We thus prove that~$\Radon(j_*\1_Z)$ vanishes on the complement of the closed set~$\dual{K}_\gamma$. Since we already know that~$\Radon(j_*\1_Z)$ vanishes on the complement of~$\dual{K}$, we consider~$[\xi:t]\in \dual{K}\setminus \dual{K}_\gamma$ and we note that in that case~$\xi \in \dual{\V}\setminus\left(\dualcone{\gamma}\cup\dualcone{\antipodal{\gamma}}\right)$.

    \medskip
    
    Write~$Z' = Z\cap \xi^{-1}(t)$,~$h=\Ker(\xi)$ and~$\gamma' = \gamma \cap h$. By~\eqref{eq:pushforward-as-Radon}, we have the following expression for the Radon transform of~$j_*\1_Z$:
    \begin{equation*}
        \Radon(j_*\1_Z)([\xi:t]) = \int_\V \1_{Z'} \d\Euler.
    \end{equation*}
    If~$Z'$ is empty, we clearly have~$\Radon(j_*\1_Z)([\xi:t]) = 0$. Otherwise, taking~$x\in Z'$, we get:
    \begin{equation*}
        \int_\V \1_{Z'} \d\Euler = \int_\V \tau_{(-x)\, *}\1_{Z'} \d\Euler = \int_{\V} \1_{\tau_x^{-1}(Z')} \d\Euler = \int_{h} \1_{\tau_x^{-1}(Z')} \d\Euler,
    \end{equation*}
    where~$\tau_u : v\in\V \mapsto v+u\in \V$ for any~$u\in\V$.
     The function~$\1_{\tau_x^{-1}(Z')}$ is compactly supported and~$\gamma'$-constructible on~$h$:
    \begin{enumerate}
        \item~$\gamma'$ is a cone of~$h$ satisfying~\eqref{hyp:cone} by \Cref{lem:cap-ker-cone},
        \item~$\tau_x^{-1}(Z')$ is a subanalytic~$\gamma'$-locally closed subset of~$h$ by  \Cref{lem:cap-ker-g-locally-closed} applied to the subset~$\tau_x^{-1}(Z)$, since~$\tau_x^{-1}(Z') = \tau_x^{-1}(Z)\cap h$, 
        \item~$\tau_x^{-1}(Z')$ is relatively compact in~$h$ since it is the intersection of the closed subset~$h$ and the relatively compact subset~$\tau_x^{-1}(Z)$ in~$\V$.
    \end{enumerate}
    Hence, its Euler integral is zero (\Cref{lem:int-CF-c-g-null}):
    \begin{equation*}
        \int_{h} \1_{\tau_x^{-1}(Z')} \d\Euler = 0.
    \end{equation*}
\end{proof}
\begin{lem}
    \label{lem:cap-ker-cone}
    Let~$\xi\in\dual{\V}\setminus\left(\dualcone{\gamma}\cup\dualcone{\antipodal{\gamma}}\right)$ and write~$h=\Ker(\xi)$. Then,~$\gamma \cap h$ is a cone of~$h$ satisfying~\eqref{hyp:cone}.
\end{lem}
\begin{proof}
    Let~$\gamma' := \gamma \cap  h$. The fact that~$\gamma'$ is a closed convex cone of~$ h$ follows from the fact that~$\gamma$ and~$ h$ are, and that these properties are stable under intersection. It is also clear that~$\gamma'$ is proper since~$\gamma$ is. For subanalyticity,~$\gamma'$ is subanalytic in~$h$ by \cite[Prop.~8.2.2 (iii)]{KS90} applied to the inclusion~$h\hookrightarrow \V$. 

    Let us then prove that the interior of~$\gamma'$ in~$ h$, denoted by~$\Int_h(\gamma')$, is non-empty. Since~$\xi\not\in\left(\dualcone{\gamma}\cup\dualcone{\antipodal{\gamma}}\right)$, we have that :    
    \begin{align*}
        \{\xi < 0\} \cap \gamma &\ne \emptyset, \\
        \{\xi > 0\} \cap \gamma &\ne \emptyset.
    \end{align*}
    Since~$\gamma$ is closed and convex with non-empty interior, classical convex analysis ensures that~$\closure{\Int(\gamma)} = \gamma$, which allows to prove that in fact:    
    \begin{align*}
        \{\xi < 0\} \cap \Int(\gamma) &\ne \emptyset, \\
        \{\xi > 0\} \cap \Int(\gamma) &\ne \emptyset.
    \end{align*}
    Since~$\Int(\gamma)$ is convex it is in particular connected, and since~$\xi:\V\to\R$ is continuous, the last two equations yield that~$ h\cap\Int(\gamma)\ne\emptyset$. To conclude, we prove that~$h\cap\Int(\gamma) \subseteq \Int_h(\gamma')$. Indeed, if~$x\in h\cap\Int(\gamma)$, then~$x + \ball{0}{\eps}\subseteq \gamma$ for~$0 < \eps \ll 1$. Since~$x\in h$, we have that 
    \begin{equation*}
        x +  \big(h\cap \ball{0}{\eps}\big) =  h \cap \big(x + \ball{0}{\eps}\big) \subseteq  h\cap\gamma = \gamma',
    \end{equation*}
    hence~$x\in\Int_h(\gamma)$ since~$x +  \big(h\cap \ball{0}{\eps}\big)$ is an open neighborhood of~$x$ in~$h$.
\end{proof}
\begin{lem}
    \label{lem:cap-ker-g-locally-closed}
    Let~$Z\subseteq \V$ be a~$\gamma$-locally closed subset and~$\xi\in\dual{\V}$. Write~$h=\Ker(\xi)$,~$Z' = Z\cap h$ and~$\gamma' = \gamma \cap h$. 
    Then,~$Z'$ is a~$\gamma'$-locally closed subset of~$h$.
\end{lem}
\begin{proof}
    Write the~$\gamma$-locally closed subset~$Z$ as~$Z = U\cap S$ with~$U$ open and~$S$ closed such that~$U+\gamma\subseteq U$ and~$S + \antipodal{\gamma} \subseteq S$, so that~$Z' = \big(U\cap h\big) \cap \big(S\cap h\big)$.
    The subsets~$U\cap h$ and~$S\cap h$ are respectively open and closed in~$h$, the subset~$S\cap h$ is stable by~$\antipodal{\gamma'}$ and~the subset~$U\cap h$ is stable by~$\gamma'$.
\end{proof}

\subsection{Proof of \Cref{prop:Radon-fromEFourier}}
\label{sec:Radon-from-EFourier}
Recall from classical Fourier theory the following easy lemma. 
\begin{lem}
    \label{lem:inversion-Fourier-rlc}
    If~$f:\R\to \R$ is a integrable, piecewise smooth and right-continuous (resp. left-continuous) function, then for all~$t\in \R$,
    \begin{equation*}
        f(t) = \frac{1}{2\pi}\Fourier^{-1}\Fourier(f)(t^+) \qquad \Big(\mbox{resp. } \frac{1}{2\pi}\Fourier^{-1}\Fourier(f)(t^-)\Big).
    \end{equation*}
\end{lem}
\begin{proof}
    For~$f : \R \to \R$ integrable and piecewise smooth, the following inversion formula \cite[Thm.~7.5]{V03} holds for the Fourier transform:
    \begin{equation}
        \label{eq:inversion-Fourier-1d}
        \frac{1}{2\pi}\Fourier^{-1}\Fourier(f)(t) = \frac{1}{2}\left(f (t^-) + f (t^+) \right),
    \end{equation}
    for all~$t\in \R$. In particular, for every point~$t$ at which~$f$ is continuous, the right-hand side is equal to~$f(t)$, hence the result.
\end{proof}

\medskip

\begin{proof}[Proof of \Cref{prop:Radon-fromEFourier}]
    \Cref{lem:link-EFourier-Fourier} yields:
    \begin{equation}
        \label{eq:EFourier-Fourier}
        \widetilde{\left(\EFourier\right)}_\xi = \1_{\R\setminus\{0\}}\cdot\Fourier(\xi_*\phi),
    \end{equation}
    and the map on the right-hand side is almost everywhere equal to~$\Fourier(\xi_*\phi)$, the Fourier transform of an integrable and piecewise smooth function over~$\R$. Hence, \cite[Thm.~7.5]{V03} ensures that the limit
    \begin{equation*}
        \limit_{A\to+\infty}\int_{-A}^A e^{ist}\cdot \widetilde{\left(\EFourier\right)}_\xi(s) \d s
    \end{equation*}
    exists, so that~$\Fourier'\left(\EFourier\right)$ is well-defined.

    \medskip
    
    Let now~$\xi\in\dual{\V}$ and~$t\in\R$. If both are zero, the result is clear. Otherwise, denote~$y = [\xi:t]\in\dual{\projective}$. If~$y\not\in\dual{K}_\gamma$, then \Cref{prop:supp-Radon-g} yields the result by~\eqref{eq:pushforward-as-Radon} and the definition of~$\Fourier'$. If~$y\in\dual{K}_\gamma$, then the constructible function~$\xi_*\phi$ is right-continuous when~$\xi\in \dualcone{\antipodal{\gamma}}\setminus\{0\}$ (resp. left-continuous when~$\xi\in \dualcone{\gamma}\setminus\{0\}$) since \Cref{lem:pushforward-g-CF-polant} ensures that it is compactly supported and~$\lambda$-constructible on~$\R$ with~$\lambda = \R_{\leq 0}$ (resp.~$\lambda = \R_{\geq 0}$). Therefore, \Cref{lem:inversion-Fourier-rlc} ensures that 
    \begin{equation}
        \label{eq:pushforward-inverse-Fourier}
        \xi_*\phi(t) = 
        \begin{cases}
            \displaystyle\frac{1}{2\pi}\Fourier^{-1}\Fourier(\xi_*\phi)(t^+) &\mbox{ if } \xi \in\dualcone{\antipodal{\gamma}}\setminus\{0\} \\[1em]
            \displaystyle\frac{1}{2\pi}\Fourier^{-1}\Fourier(\xi_*\phi)(t^-) &\mbox{ if } \xi \in\dualcone{\gamma}\setminus\{0\}.
        \end{cases}
    \end{equation}
    Hence the result, since
    \begin{equation*}
        \Fourier^{-1}\Fourier(\xi_*\phi) = \Fourier^{-1}\left(\1_{\R\setminus\{0\}}\cdot\Fourier(\xi_*\phi)\right) \overset{\eqref{eq:EFourier-Fourier}}{=} \Fourier^{-1}\left(\widetilde{\left(\EFourier\right)}_\xi\right).
    \end{equation*}
\end{proof}

\section{Sublevel-sets persistent cohomology}\label{sec:persistent-cohomology}

Recall that $M$ is a compact real analytic manifold. Let~$Z$ be a locally closed subanalytic subset of~$M$ and let~$f:M\to \V$ be a continuous subanalytic map. Let~$\gamma$ be a cone of~$\V$ satisfying~\eqref{hyp:cone-empty-int}. Note that~$\gamma$ may have empty interior in this section.

\subsection{Sublevel-sets constructible function}\label{sec:def-sublevel-CF}
We recall the sheaf-theoretic formulation of the sublevel-sets persistent homology due to Kashiwara and Schapira \cite[Sec.~1.2]{KS18A}, focusing here only on constructible functions instead of constructible sheaves. Define the \emph{$\gamma$-epigraph} of~$f$ by:

\begin{align*}
    \epigraph[f][\gamma] = \left\{ (x,v)\in M\times\V;\ f(x)-v\in\gamma \right\}.
\end{align*}
Denote also by~$\epigraph[f][]$ the usual graph of~$f$. The set~$\epigraph[f][\gamma]$ is closed and subanalytic in~$M\times \V$ by \cite[Prop.~8.2.2~(iii)]{KS90} since, denoting~$\sigma : (x,v,w) \in M\times \V\times \V \mapsto (x,v+w)\in M\times \V$, we have $\epigraph[f][\gamma] = \sigma(\epigraph[f][]\times \antipodal{\gamma})$.
Denote by~$p:M\times \V \to \V$ and~$q:M\times \V \to M$ the canonical projections, so that we get the following diagram: 
\begin{equation*}
    \begin{tikzcd}
        & \epigraph[f][\gamma] &                    \\
        & M\times \V \arrow[ld, "q"'] \arrow[rd, "p"] \arrow[u, symbol = \supset]              &                    \\
    M &                                      & \V
    \end{tikzcd}
\end{equation*}

Define the \emph{sublevel-sets constructible function associated to~$f$ on~$Z$} as:
\begin{equation}
    \label{eq:def-sublevelCF}
    \sublevelCF{f_{|Z}} 
    = p_*\left(\1_{\epigraph[f][\gamma]}\cdot q^*\1_Z\right)
    = p_*\1_{\epigraph[f][\gamma]\cap (Z\times \V)},
\end{equation}
and the \emph{level-sets constructible function associated to~$f$ on~$Z$} as:
\begin{equation}
    \levelCF{f_{|Z}}
    = p_*\left(\1_{\epigraph[f][]}\cdot q^*\1_Z\right)
    = p_*\1_{\epigraph[f][]\cap (Z\times \V)},
\end{equation}
The previous constructible functions are well-defined since~$p$ is proper on~$\epigraph[f][\gamma]$ \cite[Thm.~1.11]{KS18A} and since the properness still holds on the closure of~$\epigraph[f][\gamma]\cap(Z\times \V)$ and also for~$\gamma = \{0\}$. Note also that \Cref{lem:support-pushforward} ensures that the level-sets constructible function is compactly supported.
\begin{rk}
    \label{rk:explicit-expression-sublevelCF}
    For~$v\in \V$, one has~$\sublevelCF{f_{|Z}}(v) = \Euler\big(f^{-1}(v+\gamma)\cap Z\big)$.
\end{rk}
\begin{nota}
    When~$\V = \R$ and~$\gamma = \R_{\leq 0}$ (resp.~$\gamma = \R_{\geq 0}$), we denote~$\sublevelCF{g_{|Z}}[-] = \sublevelCF{g_{|Z}}[\gamma]$ (resp.~$\sublevelCF{g_{|Z}}[+] = \sublevelCF{g_{|Z}}[\gamma]$) the sublevel-sets constructible function associated to a continuous subanalytic map~$g:M\to\R$ on~$Z$.
\end{nota}
\begin{ex}
    The function~$\sublevelCF{f}$ is sometimes called \emph{Euler characteristic curve} for~$\V = \R$, and \emph{Euler characteristic surfaces} for~$\V = \R^2$, see for instance \cite{Bel21}.
\end{ex}
\begin{ex}
    \label{ex:ECT-LECT-SELECT}
    Considering a~subset~$Z \subseteq \R^d$ relatively compact subanalytic and locally closed, the \emph{Euler characteristic transform} defined in \cite{T14} is, for~$\xi\in\dual{\V}$ and~$t\in \R$,
    \begin{equation*}
        \ECT(Z)(\xi,t) = \Euler \big(\{x \in Z \st \dualdot{\xi}{x} \leq t\}\big) = \sublevelCF{\xi}[-](t).
    \end{equation*}
    If~$f : \R^d \to \R$ is continuous subanalytic and~$\xi\in\R^d$, denote by~$(\xi,f):\R^d \to \R^2$ the map defined by~$x\mapsto (\dualdot{\xi}{x},f(x))$. The \emph{Lifted Euler Characteristic Transform along~$f$} recently defined in \cite{KM21} is then, for any~$(h,t)\in\R^2$,
    \begin{equation*}
        \mathrm{LECT}(f)(\xi,h,t) = \sublevelCF{(\xi,f)}(h,t),
    \end{equation*} 
    where~$\gamma = \R_{\leq 0} \times \{0\} \subset \R^2$. The \emph{Super Lifted Euler Characteristic Transform along~$f$} defined in loc. cit. is, for any~$(h,t)\in\R^2$,
    \begin{equation*}
        \mathrm{SELECT}(f)(\xi,h,t) = \sublevelCF{(\xi,f)}[\gamma'](h,t),
    \end{equation*}
    where~$\gamma' = \R_{\leq 0} \times \R_{\geq 0} \subset \R^2$.
\end{ex}
The following proposition precises the relationship between the level-sets and the sublevel-sets constructible functions. It is key to the study of hybrid transforms in the context of sublevel-sets persistence, and more specifically to the proof of \Cref{prop:proj-sublevelsets-cf}, which is itself key to index-theoretic formulae (\Cref{thm:index-theoretic-formulae}). It can be derived from an analogous result on sheaves by Berkouk and Petit \cite[Prop.~4.17]{BP22} via the function-sheaf correspondence. Here, we give a proof that does not make use of the correspondence.
\begin{prop}
    \label{prop:sublevelCF-as-gammaification}
    In the preceding situation,~$
        \sublevelCF{f_{|Z}} = \levelCF{f_{|Z}} \conv \1_{\antipodal{\gamma}}$.
\end{prop}
%
\begin{proof}
    Since~$\epigraph[f][\gamma]\cap (Z\times\V) = \sigma\left(\epigraph[f][]\cap (Z\times\V)\times\antipodal{\gamma}\right)$, \Cref{rk:pushforward-indicatingCF-injective} yields that:
    \begin{equation}
        \label{eq:indic-epigraph-as-gammaification}
        \1_{\epigraph[f][\gamma]\cap (Z\times\V)} = \sigma_*\1_{\epigraph[f][]\cap (Z\times\V)\times\antipodal{\gamma}} = \sigma_*\left(\1_{\epigraph[f][]\cap (Z\times\V)}\boxtimes\1_{\antipodal{\gamma}}\right).
    \end{equation}
    Moreover, by \Cref{lem:int-boxtimes-pushforward},
    \begin{equation}
        \label{eq:p-times-id-on-gammaification}
        (p \times \id_\V)_*\left(\1_{\epigraph[f][]\cap (Z\times\V)}\boxtimes\1_{\antipodal{\gamma}}\right) = \left(p_*\1_{\epigraph[f][]\cap (Z\times\V)}\right) \boxtimes \1_{\antipodal{\gamma}}.
    \end{equation} 
    Thus, denoting by~$s : \V\times \V \to \V$ the addition, we have~$p\circ \sigma = s \circ (p \times \id_\V)$, and:
    \begin{align*}
        \sublevelCF{f_{|Z}} 
        &\hspace{0.2em}\overset{\eqref{eq:indic-epigraph-as-gammaification}}{=}\hspace{0.2em} p_*\sigma_*\left(\1_{\epigraph[f][]\cap (Z\times\V)}\boxtimes\1_{\antipodal{\gamma}}\right)\\
        &\hspace{0.2em}\overset{\eqref{eq:p-times-id-on-gammaification}}{=}\hspace{0.2em} s_*\left(p_*\1_{\epigraph[f][]\cap (Z\times\V)} \boxtimes \1_{\antipodal{\gamma}} \right)\\
        &\hspace{0.55em}=\hspace{0.55em} \levelCF{f_{|Z}} \conv \1_{\antipodal{\gamma}}.
    \end{align*}
\end{proof}
\begin{rk}
    The above proposition ensures that~$\sublevelCF{f_{|Z}}$ satisfies \Cref{ass:EL-well-def} for the cone~$\antipodal{\gamma}$ by \Cref{rk:ass-conv-cone}.
\end{rk}

\subsection{From multi-parameter to one-parameter persistence}
In this section, we show that the cone~$\gamma$ used to define sublevel-sets persistent homology of a multivalued map~$f : M \to \V$ is not relevant to define hybrid transforms with linear parameters~$\xi \in \dual{\V}$ (\Cref{rk:transform-def-and-back-to-R}). To that end, we show the following two results, stating that the pushforward by a linear form sends multi-parameter (sub)level-sets constructible functions to one-parameter ones. The first is a well-known lemma proven for completeness and used in the next one.

\begin{lem}
    \label{lem:pushforward-levelCF}
    For any morphism of real analytic manifolds~$\zeta : \V \to \R$, we have that~$\zeta_*\levelCF{f_{|Z}} = \levelCF{\zeta\circ f_{|Z}}$.
\end{lem}
\begin{proof}
    Consider the following commutative diagram
    \begin{equation}
        \label{eq:diag-commuting-xi-sum}
        \begin{tikzcd}
            (Z\times \R)\cap \epigraph[f][] \arrow[d, "{(\id_M\times \zeta)}"] \arrow[r, "p"] & \V \arrow[d, "\zeta"] \\
            (Z\times \R)\cap \epigraph[\zeta\circ f][]\arrow[r, "p'"]                             & \R    
        \end{tikzcd},
    \end{equation}
    where~$p':M\times\R\to\R$ is the canonical projection. We have:
    \begin{align*}
        \zeta_*\levelCF{f_{|Z}} 
        &= \zeta_*p_*\1_{\epigraph[f][]\cap (Z\times \V)} \\
        &= p'_*(\id_M\times\zeta)_*\1_{\epigraph[f][]\cap (Z\times \V)} \\
        &= p'_*\1_{\epigraph[\zeta\circ f][]\cap (Z\times \R)} \\
        &= \levelCF{\zeta\circ f_{|Z}},
    \end{align*}
    where the third equality follows from \Cref{rk:pushforward-indicatingCF-injective} and the equality:
    \begin{equation*}
        \left(\id_M\times\zeta\right)\left(\epigraph[f][]\cap (Z\times \V)\right) = \epigraph[\zeta\circ f][]\cap (Z\times \R).
    \end{equation*}
\end{proof}
\begin{prop}$~$
    \label{prop:proj-sublevelsets-cf}
    \begin{enumerate}
        \item For all~$\xi \in \Int(\dualcone{\antipodal{\gamma}})$,~$\xi_*\sublevelCF{f_{|Z}} = \sublevelCF{\xi\circ f_{|Z}}[-]$.
        \item For all~$\xi\in \Int(\dualcone{\gamma})$,~$\xi_*\sublevelCF{f_{|Z}} = \sublevelCF{\xi\circ f_{|Z}}[+]$.
    \end{enumerate}
\end{prop}

\begin{proof}
    For any~$\xi \in \Int(\dualcone{\antipodal{\gamma}})$, we have
    \begin{equation*}
        \begin{array}{rcl}
            \xi_*\sublevelCF{f}
            &\overset{\textnormal{Cor.}~\ref{cor:pushforward-linear-conv}}{=} &\xi_*\levelCF{f_{|Z}} \conv \xi_*\1_{\antipodal{\gamma}}\\[0.4em]
            &\overset{\textnormal{Lem.}~\ref{lem:pushforward-convex}}{=} & \xi_*\levelCF{f_{|Z}} \conv \1_{\R_{\geq 0}} \\[0.4em]
            &\overset{\textnormal{Lem.}~\ref{lem:pushforward-levelCF}}{=} & \levelCF{\xi\circ f_{|Z}} \conv \1_{\R_{\geq 0}} \\[0.4em]
            &\overset{\textnormal{Prop.}~\ref{prop:sublevelCF-as-gammaification}}{=}&  \sublevelCF{\xi\circ f_{|Z}}[-] ,
            
        \end{array}
    \end{equation*}
    and similarly replacing~$\R_{\geq 0}$ by~$\R_{\leq 0}$ if~$\xi \in \Int(\dualcone{\gamma})$.
\end{proof}
\begin{ex}
    As in \Cref{ex:ECT-LECT-SELECT}, consider a subset~$Z \subseteq \R^d$ relatively compact subanalytic and locally closed. Then, the Euler characteristic transform is the pushforward, for~$\xi\in\dual{\V}$ and~$t\in \R$,
    \begin{equation*}
        \ECT(Z)(\xi,t) = \xi_*\sublevelCF{f}(t),
    \end{equation*}
    where~$f = \id_{\R^d}$ and where~$\gamma$ is any cone satisfying~\eqref{hyp:cone-empty-int} such that~$\intpolant{\gamma}\ni \xi$.
\end{ex}

\begin{cor}
    \label{cor:transform-def-and-back-to-R}
    Let~$\kappa  \in\Lloc\cap\Lp[\R_{\geq 0}]$. The transform~$\transform[\sublevelCF{f_{|Z}}]$ is well-defined on~$\Int(\dualcone{\antipodal{\gamma}})$, and for any~$\xi\in\Int(\dualcone{\antipodal{\gamma}})$,
    \begin{equation*}
        \transform[\sublevelCF{f_{|Z}}]<\xi>=\transform[\sublevelCF{\xi\circ f_{|Z}}[-]]<1> = \int_{\R} \kappa(t)\sublevelCF{\xi\circ f_{|Z}}[-](t)\d t.
    \end{equation*}
    A similar result holds with~$\sublevelCF{\xi\circ f}[+]$ for~$\kappa  \in\Lloc\cap\Lp[\R_{\leq 0}]$ and~$\xi\in \Int(\dualcone{\gamma})$.
\end{cor}
\begin{proof}
    The well-definedness follows from \Cref{prop:sublevelCF-as-gammaification} and \Cref{prop:EL-well-def}, while the formula follows from \Cref{prop:proj-sublevelsets-cf}. 
\end{proof}
\begin{rk}
    \label{rk:transform-def-and-back-to-R}
    \Cref{cor:transform-def-and-back-to-R} implies that for any~$\xi\in\intpolant{\gamma}$, the transform~$\transform[\sublevelCF{f}]<\xi>$ is nothing but the integral (with respect to the Lebesgue measure) transform with kernel~$\kappa$ of~$\sublevelCF{\xi\circ f}[-]$. In particular, the cone~$\gamma$ such that~$\intpolant{\gamma}\ni \xi$ does not play any role. Hence, the study of hybrid transforms of sublevel-sets constructible functions for vector-valued filtrations can be reduced to the ones for real-valued filtrations.
\end{rk}
This study motivates the following definition, which gets rid of superfluous information when dealing with hybrid transforms of sublevel-sets constructible functions. Although \emph{a priori} invisible, the hybrid nature underpins the definition. 
\begin{defi}[Sublevel-sets transform]
    \label{def:sublevel-sets-transform}
    Let~$\kappa  \in\Lloc\cap\Lp[\R_{\geq 0}]$. We call \emph{sublevel-sets transform of~$f$ over~$Z$}, and denote by~$\subleveltransform$, the transform defined for~$\xi \in \dual{\V}$~by:
    \begin{equation*}
        \subleveltransform<\xi> = \transform[\sublevelCF{\xi\circ f_{|Z}}[-]]<1> = \int_\R \kappa(t) \, \Euler[\{\xi\circ f\leq t\}\cap Z] \d t.
    \end{equation*}
\end{defi}

\begin{ex}[Generalization of Morse magnitude]
    \label{def:sublevel-sets-magnitude}
    We can define the \emph{sublevel-sets magnitude of~$(Z,f)$} as the sublevel-sets transform of~$f$ over~$Z$ with kernel~$\kappa : t \mapsto e^{-t}$. In other words, for~$\xi\in\dual{\V}$,
    \begin{equation}
        \label{eq:magnitude-multi-filtration}
        |\xi\cdot (Z,f)| = \subleveltransform<\xi>[f_{|Z}] = \int_\R e^{-t} \,\Euler[\{\xi\circ f\leq t\}\cap Z] \d t.
    \end{equation}
    When~$Z = M$ and~$f : M \to \R$ is a Morse function, the previous definition specializes in the notion of Morse magnitude of \cite[Sec.~6]{GH21}. More generally, when~$\xi\circ f$ is a Morse function, we have: 
    \begin{equation*}
        |\xi\cdot (M,f)| = |(M,\xi\circ f)|_{\mathrm{Morse}},
    \end{equation*}
    with the notations of loc. cit.. 
\end{ex}

\begin{ex}[Generalization of Euler characteristic of barcodes]
    \label{def:multi-parameter-ecb}
    In \cite{BobBor12}, Bobrowski and Borman introduced the \emph{Euler characteristic of barcodes} for sublevel-sets persistent homology associated to a continuous subanalytic map~$g:M\to \R$ restricted to the range~$(-\infty, a)$ for~$a\in\R$ as follows:
    \begin{equation}
        \Euler^a_g = \int_{-\infty}^{a} \sublevelCF{g}[-](t) \d t = \subleveltransform[\1_{(-\infty,a)}]<1>[g].
    \end{equation}
    \Cref{def:sublevel-sets-transform} allows to generalize this notion to multi-parameter sublevel-sets persistent cohomology as follows. For a continuous subanalytic map~$f:M\to \V$ eventually restricted to a subanalytic locally closed subset~$Z$ of~$M$, we define for~$\xi\in\dual{\V}$,
    \begin{equation}
        \label{eq:gen-EC-barcodes}
        \Euler^a_{f_{|Z}}(\xi) = \subleveltransform[\1_{(-\infty,a)}]<\xi>[f_{|Z}].
    \end{equation}
\end{ex}

\section{Index-theoretic formulae}\label{sec:index-theoretic}
\subsection{Main results}\label{sec:index-theoretic-sublevelsets}
In this section, we define continuous Euler integration of continuous subanalytic functions and prove the formulae expressing (sub)level-sets transforms as continuous Euler integral transforms. We call these formulae \emph{index-theoretic formulae} following the terminology of \cite{GR11} to emphasize that the link between hybrid transforms and continuous Euler integrals is based on their expressions as sums of homological critical values. 

The extension of Euler integration to definable functions was introduced by Baryshnikov and Ghrist in \cite{Bar10}. Then, Bobrowski and Borman defined in \cite{BobBor12} a similar extension to the so-called \emph{tame} functions, which coincides with the first definition on tame functions which are also definable. We will use the definition of \cite{BobBor12} of continuous Euler integration in this paper, restricting ourselves to continuous subanalytic functions on compact real analytic manifolds. Although slightly less general, this framework allows us to use the theory of constructible functions.
\begin{defi}
    Let~$Z$ be a locally closed subanalytic subset of~$M$ and~$g : M\to\R$ be continuous and subanalytic. The \emph{continuous Euler upper integral of~$g$ on~$Z$} is defined~by:
    \begin{equation*}
        \int_Z g \cdEuler = \int_{0}^{+\infty} \Eulerc[\{g > u\} \cap Z] - \Eulerc[\{g \leq -u \} \cap Z] \d u,
    \end{equation*}
    and the \emph{continuous Euler lower integral of~$g$ on~$Z$} by:
    \begin{equation*}
        \int_Z g \fdEuler = \int_{0}^{+\infty} \Eulerc[\{g \geq u\}\cap Z] - \Eulerc[\{g < -u\}\cap Z] \d u.
    \end{equation*}
\end{defi}

\begin{rk}
    \label{rk:decomp-int-cdEuler}
    If~$Z'$ is a locally closed relatively compact subanalytic subset of~$M$ and~$Z$ a closed subset of~$Z'$ that is subanalytic in~$M$, then the classical distinguished triangle \cite[Eq.~(2.6.33)]{KS90} yields:
    \begin{equation*}
        \Eulerc[Z'] = \Eulerc[Z'\setminus Z] + \Eulerc[Z].
    \end{equation*}
    Therefore, in such a situation, we have:
    \begin{equation*}
        \int_{Z'} g \cdEuler = \int_{Z'\setminus Z} g \cdEuler + \int_Z g \cdEuler, 
    \end{equation*}
    and a similar equation for the lower integral.
\end{rk}

Until the end of this section, let~$Z$ be a locally closed subanalytic subset of~$M$ and~$f:M\to \V$ be a continuous subanalytic map. Moreover, consider a \emph{real valued} kernel~$\kappa\in\Lloc$. The case of a complex kernel follows from the study of its real and imaginary parts. Choose~$x_0 \in \R\cup\{\pm\infty\}$ such that~$\mathcal{K}:x\in \R \mapsto \int_{x_0}^x \kappa(t) \d t$ is well-defined over~$\R$. For the sake of readability, we extend the definition of~$\mathcal{K}$ to any~$x\in \R\cup\{\pm\infty\}$ such that the integral is well-defined. 
\begin{thm}[Index-theoretic formula for sublevel-sets]
    \label{thm:index-theoretic-formulae}
    Let~$-\infty \leq a < b \leq +\infty$. Assume that~$\kappa\cdot\1_{(-\infty,b)}\in\Lp[\R_{\geq 0}]$ and that~$\mathcal{K}$ is subanalytic. For any~$\xi\in\dual{\V}$,
    \begin{enumerate}
        \item\label{itm:index-theoretic-formula-increasing} if~$\mathcal{K}_{|(a,b)}$ is strictly increasing, then
        \begin{align*}
            \subleveltransform[\left[\kappa\1_{(a,b)}\right]]<\xi>[f_{|Z}] &= \mathcal{K}(b)\cdot\Euler(\{\xi\circ f\leq b\}\cap Z)- \mathcal{K}(a)\cdot\Euler(\{\xi\circ f\leq a\}\cap Z) \\[0.5em]
            &\hspace{8em} - \int_{\{a < \xi\circ f \leq b\}\cap Z} \mathcal{K}(\xi\circ f) \cdEuler,
        \end{align*}
        \item\label{itm:index-theoretic-formula-decreasing} if~$\mathcal{K}_{|(a,b)}$ is strictly decreasing, then
        \begin{align*}
            \subleveltransform[\left[\kappa\1_{(a,b)}\right]]<\xi>[f_{|Z}] &= \mathcal{K}(b)\cdot\Euler(\{\xi\circ f\leq b\}\cap Z) - \mathcal{K}(a)\cdot\Euler(\{\xi\circ f\leq a\}\cap Z) \\[0.5em]
            &\hspace{8em} - \int_{\{a < \xi\circ f \leq b\}\cap Z} \mathcal{K}(\xi\circ f) \fdEuler,
        \end{align*}
    \end{enumerate} 
    with~$|\mathcal{K}(b)| < + \infty$ and the convention that~$\mathcal{K}(a)\cdot \Euler(\{\xi\circ f\leq a\}\cap Z) = 0$ when~$a = -\infty$.
\end{thm}
\noindent The proof relies on the following technical lemma describing sublevel-sets constructible functions, stated and proven before the proof of \Cref{thm:index-theoretic-formulae}.
\begin{lem}
    \label{lem:expression-sublevelCF}
    Let~$g : M \to \R$ be a continuous subanalytic function. There exist a finite family of integers~$\{m_i\}_{1\leq i \leq n}$ and of real numbers~$-\infty < c_1  \leq \dots \leq c_n  \leq c_{n+1} < +\infty$, such that:
    \begin{enumerate}
        \item\label{itm:sublevelCF-as-finite-sum} One has $\levelCF{g_{|Z}} = \sum_{i=1}^n m_i \1_{[c_i,c_{i+1}]}$, and $\sublevelCF{g_{|Z}}[-] = \sum_{i=1}^n m_i \1_{[c_i,+\infty)}$.
        \item\label{itm:sublevelCF-restric} For all~$-\infty \leq a < b \leq +\infty$, denoting~$Z_{a,b} = \{ a < g \leq b \}\cap Z$, one has:
        \begin{equation*}
            \sublevelCF{g_{|Z_{a,b}}}[-] = \sum_{a < c_i \leq b } m_i\1_{[c_i,+\infty)}.
        \end{equation*}
        \item\label{itm:expression-sublevelCF-composition-homeo} If~$\mathcal{K}:\R\to\R$ is a continuous subanalytic function that is strictly monotonic on an interval containing~$\Ima(g)$, then
        \begin{equation*}
            \sublevelCF{\mathcal{K}(g)_{|Z}}[-] = \sum_{i=1}^n m_i \1_{[\mathcal{K}(c_i),+\infty)}.
        \end{equation*}
    \end{enumerate}
\end{lem}
\begin{proof}[Proof of \Cref{lem:expression-sublevelCF}.]
    Result~\ref{itm:sublevelCF-as-finite-sum} is a straightforward consequence of \Cref{prop:sublevelCF-as-gammaification} and of the convolution of indicator functions of closed intervals~\eqref{eq:conv-closed-intervals}. To prove~\ref{itm:sublevelCF-restric}, note that~$\epigraph[g][]\cap Z_{a,b}\times \R = \epigraph[g][]\cap Z\times (a,b]$, so that~$\1_{\epigraph[g][]\cap Z_{a,b}\times \R} = \1_{\epigraph[g][]\cap Z\times \R}\cdot p^*\1_{(a,b]}$, and hence:
    \begin{equation*}
        \levelCF{g_{|Z_{a,b}}} 
        = p_*\left(\1_{\epigraph[g][]\cap Z\times \R}\cdot p^*\1_{(a,b]}\right) 
        = \1_{(a,b]} \cdot \levelCF{g_{|Z}}.
    \end{equation*}
    By \Cref{prop:sublevelCF-as-gammaification}, we have~$\sublevelCF{g_{|Z_{a,b}}}[-] = \levelCF{g_{|Z_{a,b}}} \conv \1_{\R_{\geq 0}}$, so that:
    \begin{equation*}
        \sublevelCF{g_{|Z_{a,b}}}[-] = \sum_{i=1}^n m_i \1_{[c_i,c_{i+1}] \cap (a,b]} \conv \1_{\R_{\geq 0}}.
    \end{equation*}
    The result follows then from direct calculations, since:
    \begin{equation*}
        [c_i,c_{i+1}] \cap (a,b] = 
        \begin{cases}
            [c_i,c_{i+1}] &\mbox{if } a \leq c_i \leq c_{i+1} \leq b, \\
            [c_i,b] &\mbox{if } a \leq c_i \leq b \leq c_{i+1}, \\
            (a,c_{i+1}] &\mbox{if } c_i \leq a \leq c_{i+1} \leq b, \\
            (a,b] &\mbox{if } c_i \leq a < b \leq c_{i+1}. 
        \end{cases}
    \end{equation*}

    Let us now prove~\ref{itm:expression-sublevelCF-composition-homeo}. Suppose that~$a,b\in \R$ are such that~$\Ima(g)\subseteq [a,b]$ and~$\mathcal{K}$ is strictly increasing on~$[a,b]$. In this setting, the~$c_i$'s appearing in the decomposition of~$\sublevelCF{g_{|Z}}[-]$ can be chosen in~$[a,b]$. If~$u < \mathcal{K}(a)$, one has that~$\Euler[\{\mathcal{K}(g)\leq u\}\cap Z] = 0$ and if~$u > \mathcal{K}(b)$, then one has~$\Euler[\{\mathcal{K}(g)\leq u\}\cap Z] = \Eulerc(Z)$, hence the equality for such values of~$u$ by the fact that $\Eulerc[Z] = \sum_{i=1}^n m_i$.
    Now, if~$u \in \big[\mathcal{K}(a),\mathcal{K}(b)\big]$, we have:
    \begin{equation*}
        \1_{[c_i,+\infty)}\left(\mathcal{K}^{-1}(u)\right) = \1_{[c_i,b]}\left(\mathcal{K}^{-1}(u)\right) = \1_{[\mathcal{K}(c_i),\mathcal{K}(b)]}(u) = \1_{[\mathcal{K}(c_i),+\infty)}(u),
    \end{equation*}
    so that we compute:
    \begin{align*}
        \Euler[\left\{\mathcal{K}(g)\leq u\right\}\cap Z]
        &= \Euler[\left\{g\leq \mathcal{K}^{-1}(u)\right\}\cap Z] \\
        &= \sum_{i=1}^n m_i\cdot \1_{[c_i,+\infty)}\left(\mathcal{K}^{-1}(u)\right)\\
        &= \sum_{i=1}^n m_i\cdot \1_{[\mathcal{K}(c_i),+\infty)}(u).
    \end{align*}
\end{proof}
\begin{rk}
    \label{rk:index-formula-cdEuler}
    In the setting of the previous lemma, \cite[Prop.~2.4]{BobBor12} ensures that:
    \begin{equation*}
        \int_Z g \cdEuler = \sum_{i=1}^n m_i\cdot c_i.
    \end{equation*}
\end{rk}
\begin{proof}[Proof of \Cref{thm:index-theoretic-formulae}]
    Let us first prove~\ref{itm:index-theoretic-formula-increasing}. The function~$g = \xi\circ f$ being continuous and subanalytic, we can consider its sublevel-sets constructible function on~$Z$ written as in \Cref{lem:expression-sublevelCF}.\ref{itm:sublevelCF-as-finite-sum}. The fact that~$\kappa\cdot\1_{(-\infty,b)}\in\Lp[\R_{\geq 0}]$ then ensures that~$\kappa\cdot\1_{(-\infty,b)}\cdot\sublevelCF{g_{|Z}}[-]$ is integrable over~$\R$. Thus, the left-hand side of the equation to be proven is well-defined. 
    Hence the result, by the computations:
    \begin{align*}
        \int_a^b \kappa(t)\,\sublevelCF{g_{|Z}}[-](t)\d t
        &= \sum_{i=1}^n m_i \int_\R \1_{(a,b)\cap [c_i,+\infty)}(t)\,\kappa(t)\d t \\
        &= \sum_{a < c_i \leq b} m_i \left(\int_{c_i}^{b} \kappa(t)\d t\right) + \sum_{c_i \leq a} m_i \left(\int_{a}^{b} \kappa(t)\d t\right) \\
        &= \sum_{a < c_i \leq b} m_i \big(\mathcal{K}(b)-\mathcal{K}(c_i)\big) + \sum_{c_i \leq a} m_i \big(\mathcal{K}(b)- \mathcal{K}(a)\big) \\
        &= \mathcal{K}(b)\cdot\sum_{c_i \leq b} m_i - \mathcal{K}(a)\cdot\sum_{c_i \leq a} m_i - \sum_{a < c_i \leq b} m_i\,\mathcal{K}(c_i) \\
        &= \mathcal{K}(b)\cdot\Euler(\{g\leq b\}\cap Z) - \mathcal{K}(a)\cdot\Euler(\{g\leq a\}\cap Z) - \int_{Z_{a,b}} \mathcal{K}(g) \cdEuler,
    \end{align*}
    where~$Z_{a,b} = \{a < g \leq b\}\cap Z$, and where the last equality follows from \Cref{rk:index-formula-cdEuler} and \Cref{lem:expression-sublevelCF}.\ref{itm:sublevelCF-restric} and~\Cref{lem:expression-sublevelCF}.\ref{itm:expression-sublevelCF-composition-homeo} that yield:
    \begin{equation*}
        \sublevelCF{\mathcal{K}(g)_{|Z_{a,b}}}[-] = \sum_{a < c_i \leq b } m_i\1_{[\mathcal{K}(c_i),+\infty)}.
    \end{equation*}
    For~\ref{itm:index-theoretic-formula-decreasing}, note that~$-\mathcal{K}$ satisfies the assumption of~\ref{itm:index-theoretic-formula-increasing} and that:
    \begin{equation}
        \label{eq:link-lower-upper-Euler-int}
        \int_{Z'} \mathcal{K}(g) \fdEuler = -\int_{Z'} -\mathcal{K}(g) \cdEuler.
    \end{equation}
    for any subanalytic locally closed subset~$Z'$ of~$M$. This last equality is clear from the definition, and has first been proven in \cite[Lem.~4]{Bar10}.
\end{proof}
%
%
From the index-theoretic formula for sublevel-sets transforms, we deduce a formula for hybrid transforms of the level-sets constructible function~$\levelCF{f}$ for more general parameters~$\zeta : \V \to \R$ morphisms of real analytic manifolds. 
\begin{cor}[Index-theoretic formula for level-sets]
	\label{cor:index-theoretic-levelCF}
    Let~$-\infty \leq a < b \leq +\infty$ and assume that~$\mathcal{K}$ is subanalytic. 
    For any~$\zeta : \V\to \R$ morphism of real analytic manifolds,
    \begin{enumerate}
        \item if~$\mathcal{K}_{|(a,b)}$ is strictly increasing, then
        \begin{equation*}
            \transform[\levelCF{f_{|Z}}]<\zeta>[\kappa\1_{(a,b)}] = \int_{\widetilde{Z}_{a,b}} \mathcal{K}(\zeta\circ f) \fdEuler - \int_{\widetilde{Z}_{a,b}} \mathcal{K}(\zeta\circ f) \cdEuler,
        \end{equation*}
        \item if~$\mathcal{K}_{|(a,b)}$ is strictly decreasing, then
        \begin{equation*}
            \transform[\levelCF{f_{|Z}}]<\zeta>[\kappa\1_{(a,b)}] = \int_{\widetilde{Z}_{a,b}} \mathcal{K}(\zeta\circ f) \cdEuler - \int_{\widetilde{Z}_{a,b}} \mathcal{K}(\zeta\circ f) \fdEuler,
        \end{equation*}
    \end{enumerate}
    where~$\widetilde{Z}_{a,b} = \{a \leq \zeta\circ f \leq b\}\cap Z$.
\end{cor}
\begin{proof}
    We show the first result, the second being proven identically. By \Cref{lem:pushforward-levelCF}, it is sufficient to prove the case~$f = g : M \to \R$ and~$\zeta = \id_\R$. To use \Cref{thm:index-theoretic-formulae}, one must first reduce to a situation where its conditions are met. Since~$\Ima(g)$ is compact, consider~$-\infty < a' < b' < + \infty$ such that~$(a',b') \supseteq \Ima(g) \supseteq \supp(\levelCF{g})$ and thus:
	\begin{equation*}
		\transform[\levelCF{g_{|Z}}]<1>[\left[\kappa\1_{(a,b)}\right]] = \transform[\levelCF{g_{|Z}}]<1>[\left[\kappa\1_{(c,d)}\right]],
	\end{equation*}
    where~$c = \max(a,a')$ and~$d = \min(b,b')$. Now, for any~$t\in \R\cup\{\pm\infty\}$, we have:
    \begin{equation}
        \label{eq:EC-levels-sublevels}
        \Euler[\{g = t\}\cap Z] = \Euler[\{g\leq t\}\cap Z] + \Euler[\{g\geq t\}\cap Z] - \Euler[Z].
    \end{equation}
    This yields~$\levelCF{g_{|Z}} = \sublevelCF{g_{|Z}}[-] + \sublevelCF{g_{|Z}}[+] - \Euler[Z]\cdot\1_\R$, and thus:
    \begin{equation}
        \label{eq:decomposition-transform-levelCF}
		\transform[\levelCF{g_{|Z}}]<1>[\left[\kappa\1_{(c,d)}\right]] = \transform[\sublevelCF{g_{|Z}}[-]]<1>[\left[\kappa\1_{(c,d)}\right]] + \transform[\sublevelCF{g_{|Z}}[+]]<1>[\left[\kappa\1_{(c,d)}\right]] - \Euler[Z] \int_c^d \kappa(t) \d t.
    \end{equation}
    \Cref{thm:index-theoretic-formulae} yields first:
	\begin{equation}
        \label{eq:sublevelCF-index-theoretic-levelCF}
        \begin{split}
            \transform[\sublevelCF{g_{|Z}}[-]]<1>[\left[\kappa\1_{(c,d)}\right]] 
            &= \mathcal{K}(d)\cdot\Euler(\{g\leq d\}\cap Z) - \mathcal{K}(c)\cdot\Euler(\{g\leq c\}\cap Z) \\[0.5em]
            &\qquad - \int_{\{c < g \leq d\}\cap Z} \mathcal{K}(g) \cdEuler.
        \end{split}
	\end{equation}
	Moreover, using that~$\sublevelCF{g_{|Z}}[+](t) = \sublevelCF{-g_{|Z}}[-](-t)$, we get:
    \begin{equation*}
        \transform[\sublevelCF{g_{|Z}}[+]]<1>[\left[\kappa\1_{(c,d)}\right]] = \transform[\sublevelCF{-g_{|Z}}[-]]<1>[\left[\widetilde{\kappa}\1_{(-d,-c)}\right]],
    \end{equation*}
    where we denoted~$\widetilde{\kappa}(t) = \kappa(-t)$. In that case,~$\widetilde{\mathcal{K}}(x) = \int_{x_0}^x \widetilde{\kappa}(t)\d t$ satisfies~$\widetilde{\mathcal{K}}(t) = -\mathcal{K}(-t)$, so it is subanalytic and strictly increasing. Thus, \Cref{thm:index-theoretic-formulae} yields:
    \begin{equation}
        \label{eq:suplevelCF-index-theoretic-levelCF}
        \begin{split}
            \transform[\sublevelCF{g_{|Z}}[+]]<1>[\left[\kappa\1_{(c,d)}\right]] 
            &= \mathcal{K}(d)\cdot\Euler(\{g\geq d\}\cap Z) - \mathcal{K}(c)\cdot\Euler(\{g\geq c\}\cap Z) \\[0.5em]
            &\qquad + \int_{\{c \leq g < d\}\cap Z} \mathcal{K}(g) \fdEuler.
        \end{split}
    \end{equation}
    Putting~\eqref{eq:sublevelCF-index-theoretic-levelCF} and~\eqref{eq:suplevelCF-index-theoretic-levelCF} back into~\eqref{eq:decomposition-transform-levelCF} yields:
    \begin{align*}
        \transform[\levelCF{g_{|Z}}]<1>[\left[\kappa\1_{(c,d)}\right]] 
        &= \mathcal{K}(d)\cdot\big(\Euler[\{g\leq d\}\cap Z] + \Euler[\{g\geq d\}\cap Z] - \Euler[Z]\big)\qquad \\
        &\qquad - \mathcal{K}(c)\cdot\big(\Euler[\{g\leq c\}\cap Z] + \Euler[\{g\geq c\}\cap Z] - \Euler[Z]\big)\\
        &\qquad \quad + \int_{\{c \leq g < d\}\cap Z} \mathcal{K}(g) \fdEuler - \int_{\{c < g \leq d\}\cap Z} \mathcal{K}(g) \cdEuler.
    \end{align*}
    Yet, \Cref{rk:decomp-int-cdEuler} and~\eqref{eq:EC-levels-sublevels} yield:
    \begin{align*}
        \int_{\widetilde{Z}_{c,d}} \mathcal{K}(g) \fdEuler &= \mathcal{K}(d)\cdot \Euler[\{g=d\}\cap Z] + \int_{\{c \leq g < d\}\cap Z} \mathcal{K}(g) \fdEuler, \\[0.5em]
        \int_{\widetilde{Z}_{c,d}} \mathcal{K}(g) \cdEuler &= \mathcal{K}(c)\cdot \Euler[\{g=c\}\cap Z] + \int_{\{c < g \leq d\}\cap Z} \mathcal{K}(g) \cdEuler.
    \end{align*}
    Thus, we get:
    \begin{equation*}
        \transform[\levelCF{g}]<1>[\left[\kappa\1_{(c,d)}\right]] = \int_{\widetilde{Z}_{c,d}} \mathcal{K}(g) \fdEuler - \int_{\widetilde{Z}_{c,d}} \mathcal{K}(g) \cdEuler.
    \end{equation*}
    Hence the result, since~$\{g \geq a'\} = \{g \leq b'\} = M$ by definition of~$a'$ and~$b'$, so that:
    \begin{equation*}
        \widetilde{Z}_{c,d} \overset{\mathrm{def}}{=} \{c \leq g \leq d\}\cap Z = \{a \leq g \leq b\}\cap Z \overset{\mathrm{def}}{=} \widetilde{Z}_{a,b}.
    \end{equation*}
\end{proof}
\begin{rk}
	This corollary and \Cref{lem:pushforward-levelCF} also hold with identical proofs for~$f : M \to X$ and~$\zeta : X \to \R$ morphisms of real analytic manifolds, but we shall not make use of such a general statement in this paper.
\end{rk}
%
\subsection{Applications to known transforms}\label{sec:applications-known-transforms}
In this section, we show that the results of \Cref{sec:index-theoretic-sublevelsets} yield new results for the GR-Euler-Fourier and the Euler-Bessel transforms.

\begin{cor}
    \label{cor:index-theoretic-GR-EFourier-I}
    Let~$\gamma$ be a cone satisfying~\eqref{hyp:cone-empty-int}. For any~$\xi\in\Int(\dualcone{\gamma})$, we have:
    \begin{equation*}
        \EF^\mathrm{GR}\left[\sublevelCF{f_{|Z}}\right](\xi) = \int_{\{\xi\circ f\geq 0\}\cap Z} \xi\circ f \fdEuler.      
    \end{equation*}
\end{cor}
\begin{proof}
    \Cref{prop:proj-sublevelsets-cf} yields~$(-\xi)_*\sublevelCF{f_{|Z}} = \sublevelCF{-\xi\circ f_{|Z}}[-]$, so that~$\EF^\mathrm{GR}\left[\sublevelCF{f_{|Z}}\right](\xi) = \Euler^0_{f_{|Z}}(-\xi)$.
\end{proof}


\begin{cor}
    \label{cor:index-theoretic-GR-EFourier-II}
    For any~$\xi\in\dual{\V}$, we have:
    \begin{equation*}
        \EF^\mathrm{GR}\left[\levelCF{f_{|Z}}\right](\xi) = \int_{\{\xi\circ f \geq 0\}\cap Z} \xi\circ f \fdEuler - \int_{\{\xi\circ f \geq 0\}\cap Z} \xi\circ f \cdEuler.
    \end{equation*}
\end{cor}
\begin{ex}
    \label{ex:index-theoretic-GR-EFourier-II-Euclidean-space}
    Let $Z$ be a relatively compact subset of $\V$. Consider a compact real analytic submanifold~$M$ of~$\V$ containing~$Z$ and~$f : M \hookrightarrow \V$ the inclusion, so that~$\levelCF{f_{|Z}} = \1_Z$. \Cref{cor:index-theoretic-GR-EFourier-II} yields the following generalization \cite[Thm.~4.4]{GR11}:
    \begin{equation*}
        \EF^\mathrm{GR}\left[\1_Z\right](\xi) = \int_{\{\xi \geq 0\}\cap Z} \xi \fdEuler - \int_{\{\xi \geq 0\}\cap Z} \xi \cdEuler.
    \end{equation*}
\end{ex}

The parameter~$\zeta : \V \to \R$ appearing in the definition of the Euler-Bessel transform is not linear, so an index-theoretic formula for sublevel-sets constructible functions is out of reach. Yet, \Cref{cor:index-theoretic-levelCF} yields the following formula for level-sets constructible functions.
\begin{cor}
    \label{cor:index-theoretic-EBessel}
    For any~$v\in \V$, we have:
	\begin{equation*}
		\EBessel[\levelCF{f_{|Z}}]<v> = \int_Z \|v-f\| \fdEuler - \int_Z \|v-f\| \cdEuler,
	\end{equation*}
	where~$\|v-f\| : x\in M \mapsto \|v-f(x)\|$.
\end{cor}
\begin{ex}
    In the setting of \Cref{ex:index-theoretic-GR-EFourier-II-Euclidean-space}, \Cref{cor:index-theoretic-EBessel} yields the following generalization of \cite[Thm.~4.2]{GR11}:
    \begin{equation*}
        \EBessel[\1_Z]<v> = \int_Z \|v-\cdot\| \fdEuler - \int_{Z} \|v-\cdot\| \cdEuler.
    \end{equation*}
\end{ex}

\subsection{Mean Euler-Bessel transform for random filtrations}
Bobrowski and Borman computed the expected value of continuous Euler integrals of Gaussian related fields \cite[Thm.~4.1]{BobBor12}. Combined with an index theoretic formula \cite[Prop.~6.2]{BobBor12}, their result allowed them to prove an expectation result for the Euler characteristic of barcodes \cite[Thm.~6.3]{BobBor12}. In this section, we proceed in the same way to state expectation results for the Euler-Bessel transform, using our index-theoretic formula (\Cref{cor:index-theoretic-EBessel}).

Denote by~$d$ the dimension of~$M$ and let~$f:M\to \R^k$ with~$k\geq 1$ be a~$k$-dimensional Gaussian field with iid components all having zero mean and unit variance. Note that if~$\xi\in\dual{\left(\R^k\right)}$ is such that~$\|\xi\| = 1$, then a direct checking ensures that~$\xi\circ f : M\to \R$ is also a centered Gaussian field with unit variance. 
To state the result, we use classical geometric quantities known as the Lipschitz-Killing curvatures~$\LKcurv{M}$ with respect to the metric induced by the centered unit-variance Gaussian field~$\xi\circ f : M \to \R$. We refer to \cite[Sec.~7.6]{AT10} for the definition of the curvatures and to \cite[Sec.~12.2]{AT10} for the definition of the metric induced by a Gaussian field. 
\begin{cor}[Mean Euler-Bessel of level-sets]
    \label{cor:mean-EBessel}
    Suppose that the components of~$f$ are almost surely Morse functions. Then, for any~$v\in \V$,
    \begin{equation*}
        \expect{\EBessel[\levelCF{f}]<v>} =
        \begin{cases}
            \displaystyle\  2\cdot \sum_{j=1}^d \left(2\pi\right)^{-j/2}\LKcurv{M}c_j(v) &\mbox{if } d \mbox{ is odd}, \\[1em]
            \ 0 &\mbox{if } d \mbox{ is even}.
        \end{cases}
    \end{equation*}
    where 
    \begin{equation*}
        \begin{split}
            c_j(v) &= \sum_{i=0}^{+\infty} \frac{e^{-\|v\|^2/2}\|v\|^{2i}}{2^i i!} \sum_{l=0}^{\partint{(j-1)/2}}\sum_{m=0}^{j-1-2l}\1_{\{k\geq j-m-2l-2i\}} \binom{k+2i-1}{j-1-m-2l} \\[0.5em]
            &\qquad \times \frac{(-1)^{m+l}(j-1)!\Gamma((k+2i-j-2m+2l+1)/2)}{m!l!2^{(j-1-2m)/2}\Gamma((k+2i)/2)}.
        \end{split}
    \end{equation*}
\end{cor}

\begin{proof}
    Let us first remark that the following equations are straightforward consequences of \cite[Thm.~4.1]{BobBor12}:
    \begin{align*}
        \expect{\int_M d_v(f) \cdEuler} &= \Euler[M]\expect{d_v(f)} - \sum_{j=1}^d (2\pi)^{-j/2}\LKcurv{M}\int_\R \GMcurv{d_v^{-1}(-\infty, u]} \d u,\\
        \expect{\int_M d_v(f) \fdEuler} &= \Euler[M]\expect{d_v(f)} + \sum_{j=1}^d (2\pi)^{-j/2}\LKcurv{M}\int_\R \GMcurv{d_v^{-1}[u,+\infty)} \d u,
    \end{align*}
    where we denoted by~$d_v : u \in \R^k \mapsto \|u-v\| \in \R$, and where the Gaussian Minkowski functionals~$\mathcal{M}_j$ are defined by the tube formula \cite[eq.~(10.9.11)]{AT10}. Thus, the index theoretic formula for the Euler-Bessel transform (\Cref{cor:index-theoretic-EBessel}) yields:
    \begin{equation}
        \label{eq:use-index-EBessel}
        \begin{split}
            \expect{\EBessel[\levelCF{f}]<v>} 
            &= \expect{\int_M d_v(f) \fdEuler} - \expect{\int_M d_v(f) \cdEuler} \\
            &= \sum_{j=1}^d (2\pi)^{-j/2}\LKcurv{M} \int_\R \GMcurv{d_v^{-1}[u,+\infty)} + \GMcurv{d_v^{-1}(-\infty, u]} \hspace{-0.4em}\d u.
        \end{split}
    \end{equation}
    
    Following \cite[Sec.~5.2]{T06}, we use the expressions of the densities and their derivatives of noncentral~$\chi_k^2$ random variables and tube formulae to compute the Gaussian-Minkowski curvatures involved:
    \begin{equation}
        \label{eq:GMcurv-nc-chi-2}
        \begin{split}
            \GMcurv{d_v^{-1}(-\infty,u]} &= \1_{(0,+\infty)}(u) \cdot \diff*[j-1]{f_{v,k}}{x}{x=u}, \\
            \GMcurv{d_v^{-1}[u,+\infty)} &= (-1)^{j-1}\1_{(0,+\infty)}(u) \cdot \diff*[j-1]{f_{v,k}}{x}{x=u},
        \end{split}
    \end{equation}
    for any~$u\in \R$, where we denoted by~$f_{v,k}$ the density of the square root of a noncentral~$\chi_k^2$ random variable with noncentrality parameter~$\|v\|^2$, i.e. for any~$x\in \R$,
    \begin{equation*}
        f_{v,k}(x) = \sum_{i=0}^{+\infty} e^{-\|v\|^2/2} \frac{\|v\|^{2i}}{2^{i} i!} \cdot \frac{x^{k+2i-1}e^{-x^2/2}}{2^{(k+2i-2)/2}\Gamma((k+2i)/2)}.
    \end{equation*}
    This yields the following expression:
    \begin{equation*}
        \expect{\EBessel[\levelCF{f}]<v>} = \sum_{j=1}^d (2\pi)^{-j/2}\LKcurv{M} \left(1+(-1)^{j-1}\right)\int_0^{+\infty} \GMcurv{d_v^{-1}(-\infty,u]} \d u.
    \end{equation*}
    The result for even dimensions~$d$ follows then from the fact that Lipschitz-Killing curvatures~$\LKcurv{M}$ vanish for odd values of~$j$ and~$1+(-1)^{j-1}$ vanishes for even values of~$j$. For odd~$d$, the result follows from the expression:
    \begin{align*}
            \diff*[j-1]{f_{v,k}}{x}{x=u} &= \sum_{i=0}^{+\infty} e^{-\|v\|^2/2} \frac{\|v\|^{2i}}{2^{i} i!} \cdot \frac{u^{k+2i-j}e^{-u^2/2}}{\Gamma((k+2i)/2)2^{(k+2i-2)/2}} \sum_{l=0}^{\partint{(j-1)/2}} \sum_{m=0}^{j-1-2l} \1_{\{k \geq j-m-2l-2i\}}  \\
            & \qquad \times \binom{k+2i-1}{j-1-m-2l}\frac{(-1)^{m+l}(j-1)!}{m!l!2^l}u^{2m+2l},
    \end{align*}
    for~$u\geq 0$ (see \cite[Sec.~5.2]{T06}) and the clear following one:
    \begin{equation*}
        \int_0^{+\infty} u^{k+2i-j+2m+2l}e^{-u^2/2} \d u = 2^{(k+2i-j+2m+2l-1)/2}\Gamma\left((k+2i-j+2m+2l+1)/2\right).
    \end{equation*}
\end{proof}

\printbibliography

@misc{Mathematica,
  author = {Wolfram Research{,} Inc.},
  title = {Mathematica, {V}ersion 12.3},
  note = {Champaign, IL, 2021}
}

@preprint{KM21,
      title={Representing Fields without Correspondences: the Lifted Euler Characteristic Transform}, 
      author={Henry Kirveslahti and Sayan Mukherjee},
      year={2021},
      eprint={2111.04788},
      archivePrefix={arXiv preprint},
}

@preprint{BP22,
  title = {Projected distances for multi-parameter persistence modules},
  author = {Berkouk, Nicolas and Petit, François},
  year = {2022},
  eprint={2206.08818},
  archivePrefix={arXiv preprint},
}

@book{Z12,
  title={Lectures on polytopes},
  author={Ziegler, G{\"u}nter M},
  volume={152},
  year={2012},
  publisher={Springer Science \& Business Media}
}

@article{H76,
author = {Hardt, R. M.},
journal = {Inventiones mathematicae},
pages = {207-218},
title = {Triangulation of Subanalytic Sets and Proper Light Subanaly-tic Maps},
volume = {38},
year = {1976},
}

@article{S88,
     author = {Schapira, Pierre},
     title = {Cycles lagrangiens, fonctions constructibles et applications},
     journal = {S\'eminaire \'Equations aux d\'eriv\'ees partielles (Polytechnique) dit aussi "S\'eminaire Goulaouic-Schwartz"},
     publisher = {Ecole Polytechnique, Centre de Math\'ematiques},
     year = {1988-1989},
}

@article{BP19,
    title={Ephemeral persistence modules and distance comparison},
    author={Berkouk, Nicolas and Petit, Francois},
    year={2021},
    volume={21},
    number={1},
    pages={247-–277},
    journal={Algebraic and Geometric Topology},
    publisher={Mathematical Sciences Publishers}
}

@inproceedings{GRS83,
  title={A kinetic framework for computational geometry},
  author={Guibas, Leo and Ramshaw, Lyle and Stolfi, Jorge},
  booktitle={24th Annual Symposium on Foundations of Computer Science (sfcs 1983)},
  pages={100--111},
  year={1983},
  organization={IEEE Computer Society}
}

@preprint{CMT18,
  title={How many directions determine a shape and other sufficiency results for two topological transforms},
  author={Curry, Justin and Mukherjee, Sayan and Turner, Katharine},
  archivePrefix={arXiv preprint},
  eprint = {1805.09782},
  year={2018}
}

@article{BG09,
  title={Target enumeration via Euler characteristic integrals},
  author={Baryshnikov, Yuliy and Ghrist, Robert},
  journal={SIAM Journal on Applied Mathematics},
  volume={70},
  number={3},
  pages={825--844},
  year={2009},
  publisher={SIAM}
}

@book{R15,
  title={Convex analysis},
  author={Rockafellar, Ralph Tyrell},
  year={2015},
  publisher={Princeton University Press}
}

@book{Schnei14,
  title={Convex bodies: the Brunn--Minkowski theory},
  author={Schneider, Rolf},
  number={151},
  year={2014},
  publisher={Cambridge University Press}
}

@inproceedings{V88,
  title={Some integral calculus based on Euler characteristic},
  author={Viro, Oleg Yanovich},
  booktitle={Topology and geometry—Rohlin seminar},
  pages={127--138},
  year={1988},
  organization={Springer}
}

@article{T06,
  title={A Gaussian kinematic formula},
  author={Taylor, Jonathan E},
  journal={Annals of probability},
  volume={34},
  number={1},
  pages={122--158},
  year={2006},
  publisher={Institute of Mathematical Statistics}
}

@book{AT10,
  title={Random fields and geometry},
  author={Adler, Robert J. and Taylor, Jonathan E.},
  year={2009},
  publisher={Springer Science \& Business Media}
}

@article{T14,
  title={Persistent homology transform for modeling shapes and surfaces},
  author={Turner, Katharine and Mukherjee, Sayan and Boyer, Doug M.},
  journal={Information and Inference: A Journal of the IMA},
  volume={3},
  number={4},
  pages={310--344},
  year={2014},
  publisher={Oxford University Press}
}

@misc{Bel21,
      title={Euler Characteristic Surfaces}, 
      author={Gabriele Beltramo and Rayna Andreeva and Ylenia Giarratano and Miguel O. Bernabeu and Rik Sarkar and Primoz Skraba},
      year={2021},
      eprint={2102.08260},
      archivePrefix={arXiv preprint},
}

@article{S91,
  title={Operations on constructible functions},
  author={Schapira, Pierre},
  journal={Journal of Pure and Applied Algebra},
  volume={72},
  number={1},
  pages={83--93},
  year={1991},
  publisher={Elsevier}
}

@inproceedings{S95,
  title={Tomography of constructible functions},
  author={Schapira, Pierre},
  booktitle={International Symposium on Applied Algebra, Algebraic Algorithms, and Error-Correcting Codes},
  pages={427--435},
  year={1995},
  organization={Springer}
}

@book{V03,
  title={Fourier analysis and its applications},
  author={Vretblad, Anders},
  volume={223},
  year={2003},
  publisher={Springer Science \& Business Media}
}

@article{BobBor12,
title = "Euler integration of Gaussian random fields and persistent homology",
author = "Omer Bobrowski and Borman, {Matthew Strom}",
year = "2012",
volume = "4",
pages = "49--70",
journal = "Journal of Topology and Analysis",
publisher = "World Scientific Publishing Co. Pte Ltd",
number = "1",
}

@article {Bar10,
	author = {Baryshnikov, Yuliy and Ghrist, Robert},
	title = {Euler integration over definable functions},
	volume = {107},
	number = {21},
	pages = {9525--9530},
	year = {2010},
	publisher = {National Academy of Sciences},
	journal = {Proceedings of the National Academy of Sciences}
}

@Book{KS90,
  title     = {Sheaves on Manifolds},
  publisher = {Springer-Verlag},
  year      = {1990},
  author    = {Kashiwara, Masaki and Schapira, Pierre},
  volume    = {292},
  series    = {Grundlehren der Mathematischen Wissenschaften},
  address   = {Berlin},
}

@article{KS18A,
	Author = {Kashiwara, Masaki and Schapira, Pierre},
	Da = {2018/10/01},
	Isbn = {2367-1734},
	Journal = {Journal of Applied and Computational Topology},
	Number = {1},
	Pages = {83--113},
	Title = {Persistent homology and microlocal sheaf theory},
	volume = {2},
	Year = {2018},
}

@unpublished{S21,
      title={Constructible sheaves and functions up to infinity}, 
      author={Pierre Schapira},
      year={2020},
      eprint={2012.09652},
      note={arXiv preprint}
}

@inproceedings{Cur12,
  title={Euler calculus with applications to signals and sensing},
  author={Curry, Justin and Ghrist, Robert and Robinson, Michael},
  booktitle={Proceedings of Symposia in Applied Mathematics},
  volume={70},
  pages={75--146},
  year={2012}
}

@article{GR11,
	year = 2011,
	publisher = {{IOP} Publishing},
	volume = {27},
	number = {12},
	author = {Robert Ghrist and Michael Robinson},
	title = {{E}uler{\textendash}{B}essel and {E}uler{\textendash}{F}ourier transforms},
	journal = {Inverse Problems}
}

@article{GH21,
   title={Persistent magnitude},
   volume={225},
   number={3},
   journal={Journal of Pure and Applied Algebra},
   publisher={Elsevier BV},
   author={Govc, Dejan and Hepworth, Richard},
   year={2021}
}

\end{document}